%% file: main_arXiv_v2.tex
\documentclass[10pt, a4paper, twoside]{amsart}

\input{AMIN_style}
\input{AMIN_commands}

\graphicspath{{Fig_arXiv/}}
\pdfoutput=1

\title[]{Macroscopic Noisy Bounded Confidence Models \\ with Distributed Radical Opinions}

\author[]{M. A. S. Kolarijani, A. V. Proskurnikov, \lowercase{and} P. Mohajerin Esfahani}
\thanks{M.A.S. Kolarijani and P. Mohajerin Esfahani are with Delft Center for Systems and Control, Delft University of Technology, 
Delft, The Netherlands. Email: \texttt{\{M.A.SharifiKolarijani, P.MohajerinEsfahani\}@tudelft.nl}.}
\thanks{A.V. Proskurnikov is with Department of Electronics and Telecommunications, Politecnico di Torino, Turin, Italy 
and also with the Institute for Problems of Mechanical Engineering, Russian Academy of Sciences (IPME RAS). Email: \texttt{anton.p.1982@ieee.org}.}
\begin{document}

\begin{abstract}
In this article, we study the nonlinear Fokker-Planck (FP) equation that arises as a mean-field (macroscopic) approximation of bounded confidence opinion dynamics, where opinions are influenced by environmental noises and opinions
of radicals (stubborn individuals). 
The distribution of radical opinions serves as an infinite-dimensional exogenous input to the FP~equation, visibly influencing the steady opinion profile. We establish mathematical properties of the FP~equation. 
In particular, we (i)~show the well-posedness of the dynamic equation, (ii)~provide existence result accompanied by a quantitative global estimate for the corresponding stationary solution, and (iii)~establish an explicit lower bound on the noise level that guarantees exponential convergence of the dynamics to stationary state. 
Combining the results in (ii) and (iii) readily yields the \emph{input-output stability} of the system for sufficiently large noises. 

Next, using Fourier analysis, the structure of opinion clusters under the uniform initial distribution is examined. 
Specifically, two numerical schemes for identification of order-disorder transition and characterization of initial clustering behavior are provided. 
The results of analysis are validated through several numerical simulations of the continuum-agent model (partial differential equation) and the corresponding discrete-agent model (interacting stochastic differential equations) for a particular distribution of radicals.
\end{abstract}

\maketitle

%===============================================================================
\section{Introduction}
\label{sec:intro}
%===============================================================================

Recent decades have witnessed enormous progress in study of complex systems and their system-theoretic
properties~\cite{BoccalettiReview:2006,LiuBarabasi:2016}. The main effort has been invested
into study of ``self-organization'' and ``spontaneous order'' phenomena~\cite{StrogatzSync} that have inspired
the development of synchronization and consensus theory~\cite{Wu2007,MesbahiEgerBook}. Paradoxically,
these regular behaviors arising from local interactions between subsystems (agents, nodes)
of a complex system are studied much better than various ``irregular'' dynamic effects such as persistent disagreement and clustering, exhibited by many real-world systems. Although some culprits of this asynchrony and dissent (e.g. symmetries and other special structures in the coupling mechanisms, exogenous forces acting on some nodes, heterogeneous
dynamics of nodes, etc.) have been revealed in the literature~\cite{Pogromsky:02,WuZhouChen:09,XiaCao:11,Pecora:14,LuLiBai2017}, only a few mathematical models have been proposed that are sufficiently ``rich'' to capture the diversity
of clustering behaviors in real-world networks and, at the same time, admit rigorous analysis. Long before the
recent ``boom'' in complex systems, the lack of such models was realized in mathematical sociology.
The problem of disclosing mechanisms preventing consensus and maintaining enduring disagreement between
individuals~\cite{Abelson:1964} is nowadays referred to as the community cleavage problem or Abelson's diversity
puzzle~\cite{Nakamura:16,Friedkin:2015}. 
The interdisciplinary area of sociodynamical modeling~\cite{Castellano:2009,XiaWangXuan:2011,AcemogluOzdaglar:2011,Friedkin:2015,ProTempo:2017-1,ProTempo:2018,ProRavazzi:2018,Mastroeni:2019}
has attracted enormous attention of the research community and is primarily concerned with
mechanisms of \emph{opinion formation} under social influence.

Only few models, proposed in the literature to describe opinion formation processes, have been
secured by experimental evidence. Such models, however, play an important role and contribute, in various aspects,
in comprehending complex systems' behaviors such as birth, death and evolution of clusters
in systems of interacting particles, and in developing algorithms for control of these behaviors.
This explains explosion of interest in models of ``opinion formation'' in systems
and control literature. From the control-theoretic prospect, most of these models are simply networks of
interacting agents, obeying the first-order integrator model. However, the term ``opinion'' is now widespread and used
to denote the scalar or multi-dimensional state of an agent, even if this state does not have a clear sociological
interpretation\footnote{From the sociological viewpoint, opinions are cognitive orientations of individuals towards some objects or topics~\cite{Friedkin:2015}.} (belonging, e.g., to an abstract manifold~\cite{Aydogdu:2017}). The opinion is thus
some value of interest, held by an agent and updated, based on displayed opinions of the other agents.

Nowadays, linear models of opinion dynamics, extending the classical French-DeGroot system in various directions (allowing, e.g., stubborn agents, asynchronous interactions and repulsion of opinions~\cite{Acemoglu:2013,Friedkin:2015,ProTempo:2017-1,LiuChenBasar:2017}) have been thoroughly studied.
These models are sufficient to explain consensus and disagreement in social groups, as well as formation of special opinion profiles (e.g., bimodal distributions, standing for opinion polarization), however, general mechanisms leading to emergence and destruction of unequal clusters are still far from being well understood. To explain them, more complicated \emph{nonlinear} models have been proposed, mimicking some important features of social influence. One feature observed in social and biological systems is the \emph{homophily}~\cite{McPherson:2001}, or tendency of individuals to bond with similar ones. Homophily is related to \emph{biased assimilation}~\cite{LordRoss:1979} effects: individuals readily accept opinions consistent with their views and
tend to dismiss and discount opinions contradicting to their own views. Mathematically, coupling between close opinions
is stronger than that of distant opinions, which is modeled by introducing \emph{opinion-dependent} influence
weights. Although the possibility of such nonlinearities in opinion dynamics models was mentioned in
the pioneering work~\cite{Abelson:1964}, substantial progress has been primarily achieved in analysis of
\emph{bounded confidence} models proposed several decades later as extensions of the deterministic~\cite{Krause} and
randomized gossip-based~\cite{Deffuant00} consensus algorithms for multi-agent networks. Bounded confidence models
stipulate that a social actor is insensitive to opinions beyond its bounded \emph{confidence set} (usually, this set is
an open or closed ball, centered at the actor's own opinion), which makes the graph of interactions among the agents
\emph{distance-dependent}. A detailed survey of bounded confidence models and relevant mathematical results can be found in~\cite{ProTempo:2018}. Bounded confidence models exhibit convergence of the opinions to some steady values, which can reach consensus or split into several disjoint clusters. 
If the state-dependent interaction graph of the system is symmetric, this follows from general properties of iterative averaging procedures, and can alternatively be proved by exploring a special Lyapunov function (``kinetic energy'')~\cite{MotschTadmor:2013,ProTempo:2018,Etesami:2019}.
In the general case of \emph{asymmetric} interaction graphs, such a convergence has been proved only in special situations~\cite{Chazelle17,Etesami:2019},
but seems to be a generic behavior~\cite{MirtaBullo:2012,EtesamiBasar:2015,Chazelle17}.

Opinions in real social groups, however, usually do not terminate at steady values yet oscillate, which is usually explained by two factors. The first reason explaining opinion fluctuation is exogenous influence, which can be interpreted as some ``truth'' available to some individuals~\cite{HegselmannKrause:2006} or a position shared by a group of close-minded opinion leaders or stubborn individulas (``radicals'')~\cite{HegselmannKrause:2015,ZhaoZhang:16,Porfiri07}.
Important results on stability of the HK model with radicals and more general ``inertial'' bounded confidence models were obtained in~\cite{Chazelle17}.
Typically, the exogenous signal is supposed to change slowly compared to the opinion evolution and is thus replaced by a constant; the main concern is the dependence between
the constant input and the resulting opinion profile. Numerical results, reported in~\cite{HegselmannKrause:2015,ZhaoZhang:16}
demonstrate high sensitivity of the opinion clusters to the radical's opinion and reveal some
counter-intuitive effects, e.g., an increase in the number of radicals sometimes \emph{decreases} the number of their followers. The second culprit of persistent opinion fluctuation is uncertainty in the opinion dynamics,
usually modeled as a random drift of each opinion.
The presence of a random excitation can be interpreted as ``free will'' and unpredictability of a human's decision~\cite{Pineda09}; besides this, randomized opinion dynamics models
are broadly adopted in statistical physics~\cite{BenNaim:2003,Grauwin12,Pineda13,CarroToral:13}
to study phase transitions in systems of interacting particles.

Even for the classical models from~\cite{Krause,Deffuant00}, disclosing the relation between the initial and the terminal opinion profiles
remains a challenging problem (including, e.g., the $2R$-conjecture~\cite{Blondel:2007,Wang17}).
In presence of noise, the analysis becomes even more difficult; some progress in the study of the interplay between
confidence range and noise level has been achieved in recent works~\cite{HuangMantion:2013,SuChenHong:17}.
One of the important directions in analysis of bounded confidence models is examination of their
\emph{asymptotic} properties as the number of social actors becomes very large $N\to\infty$ and
their individual opinions are replaced by infinitesimal ``elements''. The arising \emph{macroscopic} approximations
of agent-based models describe the evolution of the distribution of opinion (usually supposed to have a density)
and are referred to as density-based~\cite{LorenzSurvey:2007}, continuum-agent~\cite{Blondel:2010,HendrickxOlshevsky:16},
Eulerian~\cite{MirtaJiaBullo:2014,CanutoFagnani:2012}, kinetic~\cite{BoudinSalvarani:2016}, hydrodynamical~\cite{MotschTadmor:2013} or
mean-field~\cite{Wang17,Nordio18} models of opinion formation. In the continuous-time situation, the density obeys a nonlinear \emph{Fokker-Planck} (FP) equation.
To study clustering behavior of the macroscopic bounded confidence models, efficient numerical methods have been
proposed that are based on Fourier analysis~\cite{Pineda13,Garnier17,Wang17}.

From practical viewpoint, it is convenient to consider opinions staying in a predefined interval, e.g., $[0,1]$. The
HK and Deffuant-Weisbuch (DW) models, as well as their continuous-time counterparts~\cite{ProTempo:2018},
imply that starting within the interval, opinions never escape from it. This property, however, is destroyed
by arbitrarily small noises. To keep the opinions bounded, some ``boundary conditions'' are usually introduced.
The \emph{absorbing} boundary condition assumes that the opinions are saturated at the
extreme values $0$ and $1$~\cite{Pineda13,SuChenHong:17}; an important result from~\cite{SuChenHong:17} demonstrates
that arbitrarily small noises in this situation destroy clusters and lead to approximate consensus (the maximal deviation
of opinions is proportional to the noise level). More interesting are opinion dynamics with the \emph{periodic}
boundary condition, wrapping the interval $[0,1]$ into a circle. The opinion density on the circle corresponds to
a \emph{$1$-periodic} solution of the FP equation on the real line~\cite{Garnier17,Chazelle15,Wang17}.
A disadvantage of the simple periodic boundary condition is the merging of two extreme opinion values $0$ and $1$. To distinguish between these extreme opinions, we incorporate an ``almost reflective'' (precisely, an \emph{even $2$-periodic}) boundary condition. Dealing with the macroscopic FP equation, the opinion density is then conveniently represented
by an even $2$-periodic solution on the real line. This paper is primarily concerned with mathematical properties
of such solutions.

\textbf{Main Contributions.} In this paper, we advance the theory of macroscopic modeling of bounded confidence dynamics.
We consider a bounded confidence model with environmental noise which also includes radical opinions, which are not concentrated
at a single point (as in~\cite{MirtaJiaBullo:2014,HegselmannKrause:2015,HegselmannKrause:2006}) but rather distributed
over the interval $[0,1]$. The FP equation acquires an (infinite-dimensional) exogenous \emph{input}, describing the
density and total mass of the radical opinions. 
This setup allows us to consider the interplay between the noise and the distributed radicals concerning the behavior of the system. 
In particular, for the macroscopic FP equation,
\begin{itemize}
\item[(i)] the criteria for the existence, uniqueness, and regularity
of an even periodic solution are establish (Theorem~\ref{wellposed});
\item[(ii)] the existence of stationary solution is studied and a global estimate is provided that bounds the deviation of the stationary state from the uniform distribution (Theorem~\ref{StatinarySolution});
\item[(iii)] a sufficient condition is presented for exponential convergence of the dynamics to stationary state (Theorem~\ref{StabilityStatinarySolution}). Combining this result with the global estimate of item (ii) renders input-output stability of the system (Corollary~\ref{Long-term behavior}). 
\end{itemize}
Developing ideas from~\cite{Pineda13,Garnier17,Wang17}, we then use
Fourier analysis to characterize the clustering behavior of the system under the uniform initial distribution. Specifically, two numerical schemes are presented to analyze
\begin{itemize}
\item[(iv)] the interplay between the relative number (mass) of radical agents (with respect to normal agents) and the critical noise level for order-disorder transition (Section~\ref{sec: order-disorder approx}), and
\item[(v)] the impact of the noise and the radical opinions density on the number and timing of the initial clustering behavior (Section~\ref{sec: initial cluster approx}).
\end{itemize}
These schemes are then validated through several numerical simulations of the large-scale agent-based and macroscopic density-based models for a particular distributions of radical opinions.

The paper in organized as follows. 
Section~\ref{sec:model and results} introduces the macroscopic opinion dynamics model
in question. 
Here, we also present our main theoretical results regarding well-posedness and stability of the model. 
The next two sections are concerned with technical proofs of these results. 
Section~\ref{sec:wellposed} is devoted to the proofs of well-posedness of the dynamics. 
In Section~\ref{sec:stationary and stability}, we examine the properties of the corresponding stationary equation and provide the technical proofs for theoretical results on stability of stationary state. 
In Section~\ref{sec: Fourier}, two numerical schemes are presented for characterization of the clustering behavior of the model. 
These general schemes are then tested in Section~\ref{sec:NumSim} for a particular distribution of the radical opinions.
These results are accompanied by numerical simulations of both the agent-based and the macroscopic models. 
Section~\ref{sec:conclusion} concludes the paper.

\textbf{Notations.} The convolution of two functions $f$ and $g$ is denoted by $f \star g = \int f(x) \ g(y-x) \ \diff y$. 
We note that in our case one of the functions has a compact support, so the integral always exists.
For a function $f(t,x)$ we use $f_x$ (respectively, $f_t$) to denote the derivatives with respect to (w.r.t.) $x$ (respectively, $t$), so that $f_{xx}$ is the second partial derivative w.r.t. $x$. 
We also use the notation $\partial^i_{x} f$ for the $i$-th order derivative w.r.t. $x$.
Let $X = [0,1]$ and $\tilde{X} = [-1,1]$. We use $\mathcal{P}(X)$ to denote the the space of probability densities on $X$. That is, $\rho \in \mathcal{P}(X)$ if $\int_X \rho(x) \ \diff x=1$ and $\rho(x) \geq 0$ for all $x \in X$. 
We also use $\mathcal{P}_{e}(\tilde{X})$ to denote the space of probability densities on $X$, extended evenly to $\tilde{X}$. That is, $\mathcal{P}_{e}(\tilde{X})$ is the space of all functions $\rho: \tilde{X} \rightarrow [0,\infty)$ such that $\int_X \rho(x) \ \diff x=1$  and $\rho(x) = \rho(-x) \geq 0$ for all $x\in \tilde{X}$.
$L^p (\tilde{X})$ denotes the Banach space of all measurable functions $f:\tilde{X} \rightarrow \R$ for which $\Vert f \Vert_{L^p (\tilde{X})} < \infty$. 
$H^k(\tilde{X})$ for $k \in \N$ is used to denote the Sobolev space $W^{k,2}(\tilde{X})$. 
We use the subscripts \emph{per} (respectively, \emph{ep}) to denote the closed subspace of \textit{periodic} (respectively, \textit{even periodic}) functions in the corresponding function space.
We denote the dual space of $H^1_{per}(\tilde{X})$ by $H^{-1}_{per}(\tilde{X})$ and we use $\langle\cdot,\cdot\rangle$ to denote the corresponding paring of $H^1_{per}(\tilde{X})$ and $H^{-1}_{per}(\tilde{X})$. 
We use $\rightarrow$ and $\rightharpoonup$ to denote strong and weak convergences, respectively, in an appropriate Banach space. 
A brief overview of function spaces relevant to this study is provided in Appendix~\ref{Append: prelimin}. 

%===============================================================================
\section{Model Description and Main Theoretical Results}
\label{sec:model and results}
%===============================================================================

\subsection{Macroscopic Model of Opinion Formation}
\label{sec:model}

The conventional bounded confidence model describes opinion formation process in a network of $N>1$ agents.
All agents have the same \emph{confidence range} $R>0$. Agent $i$'s opinion at time $t\ge 0$, denoted by $x_i(t)\in\mathbb{R}$,
is (directly) influenced only by the opinions of agents $j$, such that $|x_j(t)-x_j(t)|\le R$.
One of the simplest continuous-time bounded confidence models is as follows~\cite{MotschTadmor:2013}
\begin{equation}\label{eq.classical}
\dot x_i(t)=\frac{1}{N}\sum_{j=1}^Nw\big(x_j(t)-x_i(t)\big),\quad
w(\xi)=\begin{cases} \xi,\,|\xi|\le R\\0,\,|\xi|>R.\end{cases}
\end{equation}
It can be shown~\cite{ProTempo:2018} that the opinions obeying the model~\eqref{eq.classical} always converge
$x_i(t) \to x_i^{s}$ as $t\to\infty$, with $w(x_i^{s}-x_j^{s})=0$ for all $i,j$.
This corresponds to either consensus ($x_i^{s}=x_j^{s}$ for all $i,j$) of the terminal opinions or their splitting into clusters, comprising one or several coincident opinions. In the latter situation, the distance between
each two clusters is greater than $R$.

Dynamics of real opinions (as well as physical processes, portrayed by ``opinion dynamics'' models) often do not exhibit convergence to steady values, and the fluctuation of opinions persists. In order to capture this effect, random uncertainties can be introduced into the model mimicking ``free will'' and unpredictability of a human's decision~\cite{Pineda09}. The simplest of these uncertainties is an additive random noise. The model~\eqref{eq.classical} is then replaced by the system of nonlinear SDE
\begin{equation}\label{eq.classical-sde}
\diff x_i(t)=\frac{1}{N}\sum_{j=1}^Nw\big(x_j(t)-x_i(t)\big) \diff t+\sigma \diff W_i(t),
\end{equation}
where $W_i$ are independent standard Wiener processes and $\sigma>0$ characterizes the noise level.

Since the dynamics of a stochastic system~\eqref{eq.classical-sde} becomes quite complicated as the number of agents grows, the
standard approach to examine it is the \emph{mean-field} (or macroscopic) approximation, considering the opinion profile $(x_i(t))_{i=1}^N$ as a \emph{random sampling} drawn from some (time-varying) probability distribution of the opinion. Precisely, it can be shown~\cite{Dawson83,Oelschlager84,Gartner88} that empirical distributions $N^{-1}\sum_{i=1}^N\delta_{x_i(t)}$ converge (in the weak sense) as $N\to\infty$ to a distribution, whose density $\rho(t,x)$ obeys the FP equation
\begin{eqnarray} \label{pde1+}
%\left\{
\begin{array}{l}	
\rho_t = \big[\rho \ (w \star \rho )\big]_x +\frac{\sigma^2}{2} \rho_{xx},\quad t\ge 0, \ x\in\mathbb{R}.
\end{array}
%\right.
\end{eqnarray}

A natural extension of the bounded confidence dynamics allows the presence of $N_r\ge 1$ \emph{radicals} (stubborn agents, zealots) that do not assimilate others' opinions, however, influence them directly or indirectly. 
Typically, the radicals' opinions are supposed to be constant (or changing very slowly compared to the opinion formation of ``normal'' agents).
Indexing the ``normal'' individuals $1$ through $N$ and the radicals $(N+1)$ through $N+N_r$, the opinion dynamics becomes as follows
\begin{equation}\label{eq.classical-sde1}
\begin{aligned}
&\diff x_i(t)=\frac{1}{N}\sum_{j=1}^{N+N_r}w\big(x_j(t)-x_i(t)\big)\diff t+\sigma \diff W_i(t),\quad i=1,\ldots,N\\
&\dot x_i(t)=0,\quad i=N+1,\ldots,N+N_r.
\end{aligned}
\end{equation}
Often it is supposed that the radicals share a common opinion $x_i\equiv T$ for $i=N+1,\ldots,N+N_r$, which may also be
considered as some ``truth'' perceived by some individuals~\cite{HegselmannKrause:2006} or, more generally, an exogenous signal~\cite{HegselmannKrause:2015}. The ratio $M=N_r/N$ can be treated as the relative ``weight'' or ``strength'' of this external opinion. More generally, one can assume that the radicals' opinions are spread over $\mathbb{R}$. Supposing that
$N,N_r\to\infty$, the relative mass of the radicals $M$ remains constant, and their empirical distribution
$N_r^{-1}\sum_{i=1}^{N_r}\delta_{x_{N+i}}$ converges (in the weak sense) to a distribution with sufficiently smooth density $\rho_r$, the density of the ``normal'' opinions obeys the modified FP equation as follows
\begin{eqnarray} \label{pde1++}
%\left\{
\begin{array}{l}	
\rho_t = \big[\rho \ (w \star (\rho +M\rho_r))\big]_x + \frac{\sigma^2}{2} \rho_{xx},\quad t\ge 0,\ x\in\mathbb{R}.
\end{array}
%\right.
\end{eqnarray}

Note that the classical bounded confidence dynamics~\eqref{eq.classical}, being a special case of continuous-time consensus protocol, has an important property: the minimal and maximal opinions $\min_ix_i(t)$ and $\max_ix_i(t)$ are, respectively,
non-decreasing and non-increasing. 
In particular, if the initial opinions are confined to some predefined interval, e.g., $x_i(0)\in [0,1]$, then one has $x_i(t)\in [0,1]$ for all $t\ge 0$. 
The additive noise leads to random drift of the opinion profile, thus destroying the latter important property. 
Since in practice bounded ranges of opinions are usually considered, the dynamics~\eqref{eq.classical-sde},~\eqref{eq.classical-sde1} are usually complemented by \emph{boundary conditions}~\cite{Pineda13}, preventing the opinions from escaping from the predefined range. 

A typical boundary condition is the \emph{periodic} condition, where the opinion domain $[0,1]$ is wrapped on a circle of circumference $1$ 
(formally, replacing a real opinion value $x\in\mathbb{R}$ by its fractional part $\{x\}=x-\lfloor x\rfloor=x\mod 1$). 
A disadvantage of the periodic boundary condition is that there is no distinction between the extreme opinions $0$ and $1$. 
In this paper, we address this issue by considering another type of boundary condition, which we call \emph{even $2$-periodic}. Precisely, a real
opinion $x\in\mathbb{R}$ is replaced by $f(x)$, where $f$ is an even $2$-periodic function, such that
$f(x)=x$ on $[0,1]$ (and hence $f(x)=-x$ for $x\in[-1,0]$, $f(x)=2-x$ for $x\in[1,2]$ and so on). 
In other words, we first \emph{evenly} extend the opinion domain $[0,1]$ into the interval $[-1,1]$ and then wrap it on a circle of circumference $2$ so that the extreme opinions $0$ and $1$ correspond to the antipodes of this circle.
We note that with this even $2$-periodic extension, the ``effective'' boundary condition experienced by the agents is an \emph{almost reflective} one, that is, when an agent leaves the opinion domain from one end, it is reflected back into the domain from the same end. 
This is different from the behavior under simple periodic boundary condition where the agents leaving the domain form one end, enter the domain from the other end. 
However, the introduced boundary condition is ``almost'' reflective since the even extension causes some \emph{boundary effects}: the influence of more extreme neighbors of opinion values in the $R$-neighborhood of extreme opinions~$0$ and $1$ is reinforced. 
This is due to the even extension which introduces more extreme ``artificial'' neighbors; see Fig.~\ref{boundary_effect}. 

\begin{figure}[t]
\centering
\includegraphics[width=3.5in]{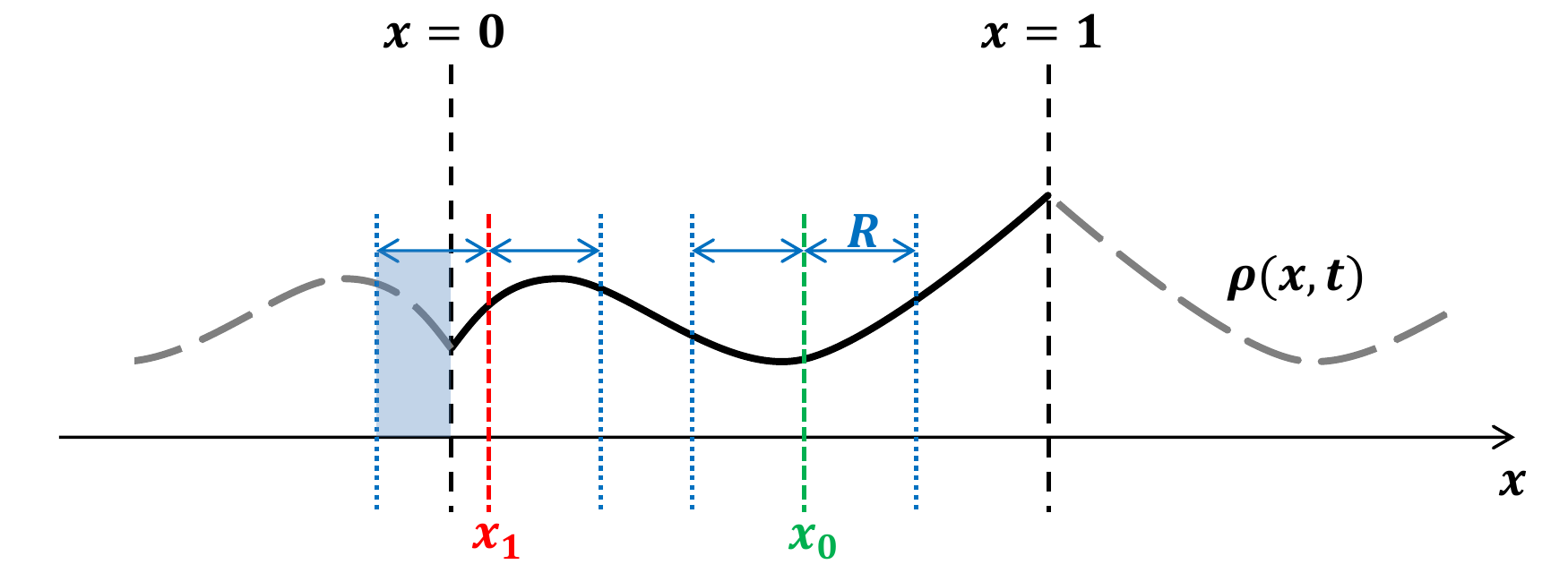}
\caption{The even $2$-periodic extension of the system. The opinion value $x_0 \in [R,1-R]$ effectively experiences a reflective boundary condition, while for the opinion value $x_1 \in [0,R]$ there is also a boundary effect due to the even extension. In particular, the influence of more extreme neighbors of $x_1$ is reinforced by introducing artificial ones (the shaded area in blue). The same boundary effect exists for opinion values in $[1-R, 1]$.}
\label{boundary_effect}
\end{figure}

As discussed in~\cite{Garnier17,Chazelle15,Wang17}, the FP equation~\eqref{pde1+} under the periodic conditions retains its validity, however, $\rho(t,x)$ is not a probability density on $\mathbb{R}$ but a $1$-periodic function $\rho(t,x+1)=\rho(t,x)\ge 0$, such that $\int_0^1\rho(t,x)\diff x=1$ (that is, $\rho(t,\cdot)$ serves as a density on the interval $[0,1]$). 
Similarly, for the even $2$-periodic boundary condition, the equation~\eqref{pde1+} retains its validity when we replace the probability density $\rho(t,x)$ with an even $2$-periodic function, that is, $\rho(t,-x)=\rho(t,x)$ and $\rho(t,x+2)=\rho(t,x)$. 
On the interval $[0,1]$, the function $\rho(t,\cdot)$ again serves as a probability density: $\int_0^1\rho(t,x)\diff x=1$. 
We also assume that the initial density $\rho_0(x)=\rho(0,x)$
and the density of radical opinions $\rho_r(x)$, defined on $[0,1]$, are extended (in the unique possible way) to even $2$-periodic functions on $\mathbb{R}$.

In this study, without loss of generality, we take $X = [0,1]$ and $\tilde{X} = [-1,1]$ to be the bounded opinion domain and its even extension, respectively. 
To summarize the discussion above, the macroscopic model for opinion dynamics considered in this study is fully described by the following PDE 
\begin{eqnarray} \label{pde2}
\left\{
\begin{array}{rcl}
\rho_t = (\rho \ G_{\rho})_x + \frac{\sigma^2}{2}\rho_{xx} \  &\text{in}& \  \tilde{X} \times (0,T) \\
\rho(\cdot + 2,t) = \rho(\cdot,t) \  &\text{on}& \  \partial \tilde{X} \times (0,T) \\
\rho(x,\cdot)=\rho_0(x) \  &\text{on}& \  \tilde{X} \times \{t=0\},
\end{array}
\right.
\end{eqnarray}
where
\begin{equation} \label{G}
G_{\rho}(x,t) := w(x) \star \big(\rho (x,t) + M \rho_r(x)\big).
\end{equation}
Note that in~\eqref{pde2}, we are considering the dynamics over a finite time horizon $T$ for the sake of analysis, however, $T$ can be chosen arbitrarily large. 
We again emphasize that the initial density $\rho_0$ and the radical density $\rho_r$ are the unique even $2$-periodic extensions of the corresponding densities from $X$ to $\tilde{X}$. 
In essence, we are considering the same dynamics as in~\cite{Chazelle15} with the extra requirement for $\rho_0$ (and the newly introduced density $\rho_r$) to be even. 
Finally, we note that~\cite{Carrillo18} also provides a detailed treatment of this dynamics (without radicals) for a class of interaction potentials on a torus in higher dimensions.

\subsection{Main Theoretical Results}
\label{sec:main results}

To recapitulate, we are interested in even $2$-periodic solutions of PDE~\eqref{pde2}, where $\rho_0$ and $\rho_r$ are even $2$-periodic. 
A natural question arises whether the model is well-posed in the sense that every (sufficiently smooth) initial condition $\rho_0$ and input $\rho_r$ correspond to a unique solution. The affirmative answer is given in the following theorem.

\begin{Thm}[Well-posedness of dynamics] \label{wellposed}
Let the initial density of normal opinions and the radical opinions density satisfy $\rho_0 \in H_{ep}^3(\tilde{X}) \cap \mathcal{P}_{e}(\tilde{X})$ and $\rho_r \in H_{ep}^2(\tilde{X}) \cap \mathcal{P}_{e}(\tilde{X})$, respectively. Then, PDE~\eqref{pde2} has a unique, even, strictly positive, classical solution $\rho \in C^1(0,\infty;C_{ep}^2(\tilde{X}))$ such that $\rho(t) \in \mathcal{P}_{e}(\tilde{X})$ for all $t > 0$.
\end{Thm} 

This result implies that $\rho(t):=\rho(t,\cdot)$ is a (strictly positive) probability density on $X = [0,1]$ for all $t > 0$, as required. 
For the  autonomous systems (without radicals),~\cite{Chazelle15, Carrillo18} provide a sufficient condition for exponential convergence of the dynamics towards uniform distribution $\rho = 1$ as an equilibrium of the system.
Unlike those studies, the uniform distribution is not an equilibrium of the model considered in this study. 
However, it is possible to extend this stability result to our model including the exogenous input, i.e., the radicals.
To this end, we first consider the stationary equation corresponding to PDE~\eqref{pde2} given by
\begin{equation} \label{pdes}
\frac{\sigma^2}{2} \rho_{xx}+(\rho \ G_{\rho})_x=0 .
\end{equation} 
We are particularly interested in even stationary solutions $\rho^s \in \mathcal{P}_{e}(\tilde{X})$ of~\eqref{pdes}. Our next result characterizes the stationary state of the system.
 
\begin{Thm} [Stationary behavior] \label{StatinarySolution}
Let $\rho_r \in H_{ep}^1(\tilde{X}) \cap \mathcal{P}_{e}(\tilde{X})$ be the radical opinions density. 
\begin{itemize}
\item \textbf{Existence}: the stationary equation~\eqref{pdes} has an even, strictly positive, classical solution $\rho^s \in C^{2}_{ep}(\tilde{X}) \cap \mathcal{P}_{e}(\tilde{X})$.
\item \textbf{Estimate}: for any $\eta > 0$, if $\sigma^2 > \sigma_b^2 + \eta c_b$, then $\Vert \rho^s - 1 \Vert_{L^2} \leq \frac{1}{\eta} \Vert \rho_r \Vert_{L^2}$, where
\begin{equation} \label{sigma_b and c_b}
\sigma_b^2 := \frac{4R}{\pi} \left( M+\frac{R}{\sqrt{3}}+2 \right) \text{ and } c_b := \frac{4R^2M}{\pi \sqrt{3}}.
\end{equation}
\end{itemize}
\end{Thm}

Notice how the global estimate in Theorem~\ref{StatinarySolution} bounds the difference between the stationary solution and the uniform distribution. This result shows that, even in presence of radical opinions, the stationary solution can be made arbitrarily close to the uniform distribution by increasing the noise level beyond a minimum level~$\sigma_b$. 
We note that the minimum noise level $\sigma_b$ is directly related to the confidence range $R$ and the relative mass $M$. 
Also, as the ``energy'' $M \Vert \rho_r \Vert_{L^2}$ of the radicals increases, in order to counteract their effect and keep the stationary profile in a somewhat uniform state, one must increase the noise level further beyond~$\sigma_b$.

With this result in hand, we can now consider the asymptotic stability of stationary state. 
The next result provides a sufficient condition for exponential convergence of the dynamics to stationary state for arbitrary (and sufficiently smooth) initial density $\rho_0$ and radical density $\rho_r$.
 
\begin{Thm} [Stability] \label{StabilityStatinarySolution}
Let $\rho_0 \in H_{ep}^3(\tilde{X}) \cap \mathcal{P}_{e}(\tilde{X})$ be the initial density of normal opinions and $\rho_r \in H_{ep}^2(\tilde{X}) \cap \mathcal{P}_{e}(\tilde{X})$ be the radical opinions density. Also, let $\rho \in C^1(0,\infty;C_{ep}^2(\tilde{X}))$ with $\rho(t) \in \mathcal{P}_{e}(\tilde{X})$ be the solution to the dynamic equation~\eqref{pde2}. Then, $\rho(t)$ converges to a stationary state $\rho^s \in C^{2}_{ep}(\tilde{X}) \cap \mathcal{P}_{e}(\tilde{X})$ exponentially in $L^2$ as $t \rightarrow \infty$ if $\sigma > \sigma_s$, where $\sigma_s > 0$ uniquely solves
\begin{equation} \label{sigma_s}
\sigma_s^2 = \frac{4R(3+M)}{\pi} + \frac{4R^2}{\pi \sqrt{3}} \ \exp \left( \frac{8R(1+M)}{\sigma_s^2} \right).
\end{equation}
\end{Thm} 
 
An immediate result of Theorems~\ref{StatinarySolution} and~\ref{StabilityStatinarySolution} is that for sufficiently large noises, the dynamics will converge to a stationary state that can be made arbitrarily close to uniform distribution by increasing the noise level.

\begin{Cor}[Input-output stability] \label{Long-term behavior}
For any $\eta > 0$, if 
$
\sigma^2 > \max \{ \sigma_b^2 + \eta c_b,\sigma_s^2 \},
$
where $\sigma_b$ and $c_b$ are defined in~\eqref{sigma_b and c_b} and $\sigma_s > 0$ uniquely solves~\eqref{sigma_s}, then it holds that 
\begin{align} \label{main result}
\Vert \rho (t) - 1 \Vert_{L^2} \leq \beta e^{-\lambda t} +  \frac{1}{\eta} \Vert \rho_r \Vert_{L^2},
\end{align}
where the constant $\beta> 0$ depends on $\rho_0$ and $\rho_r$ and the convergence rate $\lambda >0$ depends on $\sigma$, $R$, and $M$.
\end{Cor}

\begin{Rem}[Connection to existing works]\label{connection to stability result}
The stability result of Corollary~\ref{Long-term behavior} corresponds to the result reported in~\cite[Theorem 2.3]{Chazelle15} on global stability of uniform distribution $\rho = 1$ for sufficiently large noises in the autonomous system \emph{without radicals}. 
In particular, by setting $M=0$ in the estimate given in Theorem~\ref{StatinarySolution}, one has $c_b = 0$, hence  $\rho^s = 1$ is the unique stationary state of the system for $\sigma^2 > \sigma_b^2 = \frac{4R}{\pi} \left( 2+R/\sqrt{3} \right)$. 
We note that $\sigma_b$ is the same minimum noise level given in~\cite[Theorem 2.3]{Chazelle15}, taking into account a multiplicative factor of two due to the even extension considered in this study. 
However, direct application of Theorem~\ref{StabilityStatinarySolution} for stability of $\rho^s = 1$ leads to a sufficient minimum noise level $\sigma_s > \sigma_b$. This is due to the fact that this result is based on conservative estimates for $\rho^s$. 
Indeed, if one incorporates the fact that $\rho^s = 1$ and modifies some of the arguments provided in the proof of Theorem~\ref{StabilityStatinarySolution} in Section~\ref{sec:longterm}, then one can show that, in the absence of radical agents, the uniform distribution $\rho^s = 1$ is also globally exponentially stable for $\sigma > \sigma_b$, reproducing the result of \cite[Theorem 2.3]{Chazelle15}.
\end{Rem}

Finaly, we note that, based on the results provided in~\cite{Carrillo18}, the input-output stability result of Corollary~\ref{Long-term behavior} can be generalized to multi-dimensional first-order stochastic interacting particle systems for a particular class of interaction potentials. 

The next two sections are mainly concerned with the technical proofs of the theoretical results listed above. 

%===============================================================================
\section{Well-posedness of Dynamics}
\label{sec:wellposed}
%===============================================================================

This section is devoted to the proof of Theorem~\ref{wellposed} concerning the well-posedness of the dynamics~\eqref{pde2}. 
Throughout this section, all the norms are w.r.t. $\tilde{X} = [-1,1]$ (as opposed to $X = [0,1]$), unless indicated otherwise. 
We use $C, C_0, C_1, \ldots$ to represent a generic constant (depending on the model parameters) which actual values may change from line to line. 
In case these constants depend on a particular object of interest, say $\theta$, this dependence is explicitly indicated by $C[\theta]$.

Let us first note that because of periodicity, the mass is preserved in~\eqref{pde2}, that is, 
$$\int_{\tilde{X}} \rho(x,t) \ \diff x= \int_{\tilde{X}} \rho_0(x) \ \diff x = 2,$$
for all $t\geq 0$. 
In particular, we have 
$$
\Vert \rho (t) \Vert_{L^1} \geq \int_{\tilde{X}} \rho(x,t) \ \diff x = 2 > 0.
$$
We will be using this property in the sequel.

We start by presenting some useful estimates for the object $G_{\rho}$ defined in~\eqref{G} that make it possible to extend the results provided by~\cite{Chazelle15} to our model.
\begin{Lem}[Estimates for $G_{\rho}$] \label{Lem: G}
Let $G_{\rho}$ be the function defined in~\eqref{G} with $\rho_r \in \mathcal{P}_{e}(\tilde{X})$. 
If $\rho(t) \in L_{per}^1 (\tilde{X})$, then 
\begin{equation} \label{G1}
\Vert G_{\rho} \Vert_{L^{\infty}} \leq R \left( \Vert \rho (t) \Vert_{L^1} + 2M \right).
\end{equation}
If, moreover, $\Vert \rho (t) \Vert_{L^1} > 0$ , then
\begin{equation} \label{G2}
\Vert G_{\rho} \Vert_{L^{\infty}} \leq C \ \Vert \rho (t) \Vert_{L^1} \leq C \ \Vert \rho (t) \Vert_{L^2}.
\end{equation}
\end{Lem}
\begin{proof}
Notice
\begin{align*}
|G_{\rho}(x,t)| &= \left| \int (x-y) \ \textbf{1}_{|x-y|\leq R} \ (\rho(y,t) + M \rho_r(y)) \ \diff y  \right| \nonumber \\
&\leq \int |x-y| \ \textbf{1}_{|x-y|\leq R} \ |\rho(y,t) + M \rho_r(y)| \ \diff y \nonumber \\
&\leq R \int_{\tilde{X}} |\rho(y,t) + M \rho_r(y)| \ \diff y \nonumber \\
&\leq R \left( \int_{\tilde{X}} |\rho(y,t) | \ \diff y + 2M \right),
\end{align*}
from which we can conclude the inequality~\eqref{G1}. The first inequality in~\eqref{G2} then immediately follows from~\eqref{G1} and the assumption $\Vert \rho (t) \Vert_{L^1} > 0$. For the second inequality in~\eqref{G2} notice that since $\tilde{X}$ is of finite measure $\mu (\tilde{X}) = \mu ([-1,1]) = 2$, for any measurable function $v$ we have
\begin{equation} \label{finmeasure}
\Vert v \Vert_{L^p(\tilde{X})} \leq \mu (\tilde{X})^{\frac{1}{p} - \frac{1}{q}} \ \Vert v \Vert_{L^q(\tilde{X})},
\end{equation}
where $1 \leq p \leq q \leq \infty$. 
\end{proof}

Using the estimate~\eqref{G2} in Lemma~\ref{Lem: G}, one can follow similar arguments as in~\cite[Lemma 2.1]{Chazelle15} to show $\Vert \rho(t) \Vert_{L^1} = 2$ and $\rho(t) \geq 0$ for all $t \geq 0$; see also~\cite[Corollary 2.2]{Chazelle15}. 
Specifically, assuming PDE~\eqref{pde2} has a solution $\rho \in C^1(0,T;C_{per}^2(\tilde{X}))$, one can derive a priori estimate which in turn implies that the solution is non-negative so that $\rho(t)$ is a probability distribution on $X=[0,1]$ for all $t \geq 0$.

\begin{Lem}[Estimates for $\partial_x^k G_{\rho}$] \label{Lem: Gx}
Let $G_{\rho}$ be the function defined in~\eqref{G} with $\rho_r \in \mathcal{P}_{e}(\tilde{X})$. 
\begin{itemize}
\item[(i)] For $1 \leq p \leq \infty$, if $\rho(t), \rho_r  \in L_{per}^p (\tilde{X})$ with $\Vert \rho (t) \Vert_{L^1} > 0$, then
\begin{equation} \label{Gx1}
\Vert (G_{\rho})_x \Vert_{L^p} \leq C_1 \ \Vert \rho (t) \Vert_{L^p} + C_2 \ \Vert \rho_r \Vert_{L^p} \leq C[\Vert \rho_r \Vert_{L^p}] \ \Vert \rho (t) \Vert_{L^p}.
\end{equation}
\item[(ii)] For $k \geq 2$, if $\rho(t), \rho_r  \in H_{per}^{k-1} (\tilde{X})$ with $\Vert \rho (t) \Vert_{L^1} > 0$, then  
\begin{equation} \label{Gx2}
\Vert \partial_x^k G_{\rho}\Vert_{L^2} \leq C[\Vert \rho_r \Vert_{H^{k-1}}] \ \Vert \rho (t) \Vert_{H^{k-1}}.
\end{equation}
\end{itemize}
\end{Lem}
\begin{proof}
We have
\begin{align} \label{diffG}
\big(G_{\rho}(x,t)\big)_x = & \partial_x \left( \int (x-y) \ \textbf{1}_{|x-y|\leq R} \ \big(\rho(y,t) + M \rho_r(y)\big) \ \diff y \right) \nonumber \\
= &\partial_x \left( \int_{x-R}^{x+R} (x-y) \ \big(\rho(y,t) + M \rho_r(y)\big) \ \diff y \right) \nonumber \\
= &-R \big( \rho(x + R,t)+\rho(x- R,t) + M \rho_r(x + R) +  M \rho_r(x - R) \big) \nonumber \\ &+ \int_{x-R}^{x+R} \big(\rho(y,t) + M \rho_r(y,t)\big) \ \diff y,
\end{align}
which leads to the first inequality in~\eqref{Gx1}. Using the fact that $\Vert \rho (t) \Vert_{L^2} \geq  C \ \Vert \rho (t) \Vert_{L^1} > 0$ (see~\eqref{finmeasure}), we have the second inequality in~\eqref{Gx1}.
Computing the higher order derivatives w.r.t. $x$, we obtain for $k \geq 2$
\begin{align*}
\partial^k_x G_{\rho} = &
-R \left( \partial^{k-1}_x \rho(x+R,t) + \partial^{k-1}_x \rho(x-R,t) + M \partial^{k-1}_x \rho_r(x+R,t) + M \partial^{k-1}_x \rho_r(x-R,t) \right) \\
&+ \partial^{k-2}_x \rho(x+R,t) - \partial^{k-2}_x \rho(x-R,t) + M \partial^{k-2}_x \rho_r(x+R,t) - M \partial^{k-2}_x \rho_r(x-R,t). 
\end{align*}
Hence,
\begin{align*} \label{reg}
\Vert \partial^k_x G_{\rho} \Vert_{L^2} &\leq C \left( \Vert \partial^{k-1}_x \rho(t) \Vert_{L^2} + \Vert \partial^{k-2}_x \rho(t) \Vert_{L^2} + \Vert \partial^{k-1}_x \rho_r \Vert_{L^2} + \Vert \partial^{k-2}_x \rho_r \Vert_{L^2} \right) \nonumber \\
&\leq C \left( \Vert \rho(t) \Vert_{H^{k-1}} + \Vert \rho_r \Vert_{H^{k-1}} \right) \nonumber \\
&\leq C[\Vert \rho_r \Vert_{H^{k-1}}] \ \Vert \rho(t) \Vert_{H^{k-1}},
\end{align*}
where for the last inequality we used the fact that $ \Vert \rho(t) \Vert_{H^{k-1}} \geq \Vert \rho (t) \Vert_{L^2} \geq  C \ \Vert \rho (t) \Vert_{L^1} > 0$.
\end{proof}

\begin{Lem} [More estimates for $G_{\rho}$] \label{Lem: pG}
Let $\nu \in H^{k}_{per}(\tilde{X})$, $\rho_r \in H^{k-1}_{per}(\tilde{X}) \cap \mathcal{P}_e(X)$, and $\rho(t) \in H^{k-1}_{per}(\tilde{X})$ with $\Vert \rho (t) \Vert_{L^1} > 0$. Then for $k \geq 2$
\begin{equation}
\Vert \nu G_{\rho} \Vert_{H^k} \leq C[\Vert \rho_r \Vert_{H^{k-1}}] \ \Vert \nu \Vert_{H^k} \ \Vert \rho(t) \Vert_{H^{k-1}}.
\end{equation}
\end{Lem}
\begin{proof}
Notice
\begin{equation} \label{R1}
\Vert \nu G_{\rho} \Vert_{H^k} \leq C \left( \Vert \nu G_{\rho} \Vert_{L^2} +\Vert \partial^k_x (\nu G_{\rho}) \Vert_{L^2} \right).
\end{equation}
For the first term on the right hand side (r.h.s.) of~\eqref{R1}, we have
\begin{equation*}
\Vert \nu G_{\rho} \Vert_{L^2} \leq \Vert \nu \Vert_{L^2} \ \Vert G_{\rho} \Vert_{L^{\infty}} \leq C \ \Vert \nu \Vert_{L^2} \ \Vert \rho(t) \Vert_{L^2} \leq C \ \Vert \nu \Vert_{H^k} \ \Vert \rho(t) \Vert_{H^{k-1}},
\end{equation*}
where for the second inequality we used~\eqref{G2}. Also, using Leibniz rule, for the second term on the r.h.s. of~\eqref{R1} we can write 
\begin{align*}
\Vert \partial^k_x (\nu G_{\rho}) \Vert^2_{L^2} &= \left\Vert \sum_{i=0}^k C_i \ \partial^{k-i}_x \nu \ \partial^{i}_x G_{\rho} \right\Vert^2_{L^2} \nonumber \\
&\leq C_0 \ \Vert \partial^k_x \nu \Vert^2_{L^2} \ \Vert G_{\rho} \Vert^2_{L^{\infty}} + \sum_{i=1}^k C_i \ \Vert \partial^{k-i}_x \nu \Vert^2_{L^{\infty}} \ \Vert \partial^i_x G_{\rho} \Vert^2_{L^2} \nonumber \\
&\leq C_0 \ \Vert \nu \Vert^2_{H^k} \ \Vert \rho \Vert^2_{L^2} + \sum_{i=1}^k C_i \ \Vert \partial^{k-i}_x \nu \Vert^2_{H^1} \ \Vert \partial^i_x G_{\rho} \Vert^2_{L^2},
\end{align*}
where for the last inequality we used Morrey's inequality which implies
$$
\Vert \partial^{k-i}_x \nu \Vert_{L^{\infty}} \leq C \ \Vert \partial^{k-i}_x \nu \Vert_{H^1}.
$$
Now, from~\eqref{Gx1} we have for $i=1$
\begin{equation*}
\Vert \partial^i_x G_{\rho} \Vert^2_{L^2} \leq C[\Vert \rho_r \Vert_{L^2}] \ \Vert \rho (t) \Vert^2_{L^2},
\end{equation*}
and from~\eqref{Gx2} we have for $i \geq 2$
\begin{align*}
\Vert \partial_x^i G_{\rho}\Vert^2_{L^2} \leq C[\Vert \rho_r \Vert_{H^{i-1}}] \ \Vert \rho (t) \Vert^2_{H^{i-1}}. 
\end{align*}
Putting all these estimates together while keeping only the highest Sobolev norms, we obtain
\begin{align*}
\Vert \nu G_{\rho} \Vert_{H^k} &\leq C_1 \ \Vert \nu \Vert_{H^k} \ \Vert \rho(t) \Vert_{H^{k-1}} + C_2 \ \Vert \nu \Vert_{H^k} \ \Vert \rho(t) \Vert_{L^2} + C_3[\Vert \rho_r \Vert_{H^{k-1}}] \ \Vert \nu \Vert_{H^1} \ \Vert \rho(t) \Vert_{H^{k-1}} \nonumber \\
&\leq C[\Vert \rho_r \Vert_{H^{k-1}}] \ \Vert \nu \Vert_{H^k} \ \Vert \rho(t) \Vert_{H^{k-1}}.
\end{align*}
\end{proof}

\begin{Rem} [Connection to existing works]
Lemma~\ref{Lem: pG} is an extension of~\cite[Proposition 4.1]{Chazelle15}.
\end{Rem}

With these estimates in hand, we can follow the same arguments as in~\cite{Chazelle15} to show well-posedness of the dynamics described by PDE~\eqref{pde2}.

\begin{proof}[Sketch of proof of Theorem~\ref{wellposed}]
Consider the following sequence of PDEs 
\begin{eqnarray} \label{pdeseq}
\left\{
\begin{array}{rcl}
\partial_t \rho_n = \partial_x (\rho_n \ G_{\rho_{n-1}} ) + \frac{\sigma^2}{2} \partial_{xx} \rho_n \  &\text{in}& \  \tilde{X} \times (0,T)  \\
\rho_n(\cdot + 2,t) = \rho_n(\cdot,t) \  &\text{on}& \  \partial \tilde{X} \times (0,T)  \\
\rho_n(x,\cdot)=\rho_0(x) \  &\text{on}& \  \tilde{X} \times \{t=0\},
\end{array}
\right.
\end{eqnarray}
with smooth initial and radical distributions $\rho_0, \rho_r \in C_{per}^{\infty}(\tilde{X}) \cap \mathcal{P}_{e}(\tilde{X})$ for now. 
By standard results on linear parabolic PDEs~\cite[Chapter 7]{Evans}, there exists a sequence $\{\rho_n : \ n\geq 0 \}$ in $C^{\infty}(0,T;C_{per}^{\infty}(\tilde{X}))$ that satisfies~\eqref{pdeseq}. 
Furthermore, using the estimate~\eqref{G2} in Lemma~\ref{Lem: G}, one can follow the same procedure provided in~\cite[Proposition 3.1]{Chazelle15} to show $\Vert \rho_n (t) \Vert_{L^1} = \Vert \rho_n (0) \Vert_{L^1} = 2$, and hence, $\rho_n(t) \geq 0$ for all $n \geq 1$ and $t \geq 0$; see also~\cite[Corollary 3.2]{Chazelle15}.

\begin{Rem}[Evenness of $\rho_n$] 
One can use the evenness of $\rho_0$ and $\rho_r$ to show that the unique solutions $\rho_n$ to PDEs~\eqref{pdeseq} are also even in $x$ for all $t \geq 0$. 
However, since this property will not be used for existence, uniqueness and regularity results provided below, we will postpone this argument to later when we deal with the evenness of the unique solution to PDE~\eqref{pde2}. 
\end{Rem}

\textit{Existence with smooth data.} Using Lemmas~\ref{Lem: G} and~\ref{Lem: Gx} and following a similar idea as in~\cite[Lemmas 3.5 and 3.7]{Chazelle15}, we can obtain the following convergence results for a limiting object $\br$
\begin{eqnarray} \label{converg}
\rho_n \rightarrow \br \ &\text{in}& \  L^1(0,T;L^1_{per}(\tilde{X})), \label{converg1}  \\
\rho_{n_k} \rightharpoonup \br  \ &\text{in}& \  L^2(0,T;H^1_{per}(\tilde{X})), \label{converg2} \\
\partial_t \rho_{n_k} \rightharpoonup \br_t \ &\text{in}& \  L^2(0,T;H^{-1}_{per}(\tilde{X})), \label{converg3}
\end{eqnarray}
where $n_k$ denotes a subsequence. Moreover, we have the following estimate for $\{\rho_n: \ n \geq 1\}$ and $\br$
\begin{equation} \label{estim}
\Vert \rho \Vert_{L^{\infty}(0,T;L^2)} + \Vert \rho \Vert_{L^2(0,T;H^1)} + \Vert \rho_t \Vert_{L^2(0,T;H^{-1})} \leq C[T] \ \Vert \rho_0 \Vert_{L^2}.
\end{equation}
We claim that $\br$ is the unique weak solution to~\eqref{pde2}. 
That is, $\br$ solves the weak formulation of~\eqref{pde2} defined as
\begin{equation} \label{pdeweak}
\int_0^T \langle \eta, \rho_t \rangle \ \diff t + \int_0^T \int_{\tilde{X}} \left( \frac{\sigma^2}{2} \rho_x + \rho \ G_{\rho} \right) \ \eta_x \ \diff x \diff t = 0,
\end{equation}
for any $\eta \in L^2(0,T;H^1_{per}(\tilde{X}))$. To show this, we multiply~\eqref{pdeseq} by $\eta$ with $n = n_k$ and integrate to obtain
\begin{equation} \label{E1}
\int_0^T \langle \eta , \partial_t \rho_{n_k} \rangle \ \diff t + \frac{\sigma^2}{2} \int_0^T \int_{\tilde{X}} \partial_x \rho_{n_k} \ \eta_x \ \diff x \diff t + \int_0^T \int_{\tilde{X}} \rho_{n_k} \ G_{\rho_{n_k-1}} \ \eta_x \ \diff x \diff t = 0.
\end{equation}
For the first two terms in~\eqref{E1}, using convergence results~\eqref{converg3} and~\eqref{converg2}, we have as $k \rightarrow \infty$
\begin{align*}
\int_0^T \langle \eta , \partial_t \rho_{n_k} \rangle \ \diff t &\rightarrow \int_0^T  \langle \eta, \br_t \rangle \ \diff t,
\end{align*}
and
\begin{align*}
\int_0^T \int_{\tilde{X}} \partial_x \rho_{n_k} \ \eta_x \ \diff x \diff t \rightarrow  \int_0^T \int_{\tilde{X}} \br_x \eta_x \ \diff x \diff t.
\end{align*}
Also, the last term in~\eqref{E1} can be written as
\begin{align} \label{E2}
&\int_0^T \int_{\tilde{X}} (\rho_{n_k} - \br) \ G_{\rho_{n_k-1}} \ \eta_x \ \diff x \diff t + \int_0^T \int_{\tilde{X}} \br \ (w \star ( \rho_{n_k-1} -\br)) \ \eta_x \ \diff x \diff t + \int_0^T \int_{\tilde{X}} \br \ G_{\br} \ \eta_x \ \diff x \diff t.
\end{align}
The limit of the first integral in~\eqref{E2} is zero as $k \rightarrow \infty $. 
Indeed, we know $G_{\rho_{n_k-1}}$ is bounded by the inequality~\eqref{G1} in Lemma~\ref{Lem: G}, hence, $\eta_xG_{\rho_{n_k-1}} \in L^2(0,T;L^2_{per}(\tilde{X}))$. Also,~\eqref{converg2} implies $\rho_{n_k} \rightharpoonup \br$ in $L^2(0,T;L^2_{per}(\tilde{X}))$. 
The limit of the second integral in~\eqref{E2} is also zero as $k \rightarrow \infty$. For this integral, we have
\begin{align*}
\int_0^T \int_{\tilde{X}} \br \ (w \star ( \rho_{n_k-1} -\br)) \ \eta_x \ \diff x \diff t &\leq \Vert \br \Vert_{L^{\infty}(0,T;L^2)} \ \Vert \eta_x \Vert_{L^2(0,T;L^2)} \ \Vert w \star ( \rho_{n_k-1} -\br) \Vert_{L^2(0,T;L^2)} \nonumber \\
&\leq C[T] \ \Vert \rho_0 \Vert_{L^2} \ \Vert \eta \Vert_{L^2(0,T;H^1)} \ \left( \int_0^T \Vert \rho_{n_k-1} -\br \Vert^2_{L^1} \ \diff t \right)^{\frac{1}{2}},
\end{align*} 
where for the second inequality we used~\eqref{estim} and the fact that 
$|w \star ( \rho_{n_k-1} -\br)| \leq C \ \Vert \rho_{n_k-1} -\br \Vert_{L^1}$  
by Lemma~\ref{G} (set $M=0$ in~\eqref{G1}). Now, notice 
$\Vert \rho_{n_k-1} -\br \Vert_{L^1(\tilde{X})} \leq \Vert \rho_{n_k-1} \Vert_{L^1(\tilde{X})} + \Vert \br \Vert_{L^1(\tilde{X})} \leq 4$. 
Hence,
\begin{align*}
\int_0^T \Vert \rho_{n_k-1} -\br \Vert^2_{L^1} \ \diff t &\leq 4 \ \int_0^T \Vert \rho_{n_k-1} -\br \Vert_{L^1} \ \diff t = 4 \ \Vert \rho_{n_k-1} -\br \Vert_{L^1(0,T;L^1)} \rightarrow 0,
\end{align*}
as $k \rightarrow \infty$ by the strong convergence~\eqref{converg1}. Putting all these results together, we see that $\br$ indeed satisfies the weak formulation~\eqref{pdeweak}.

To complete the existence result, we have to show $\br(x,0) = \rho_0(x)$. This condition makes sense since $\br \in C(0,T;L^2_{per}(\tilde{X}))$ by~\cite[Theorem 3.8]{Chazelle15} and the convergence results~\eqref{converg2} and~\eqref{converg3}.
Pick some $\eta \in C^1(0,T;H^1_{per}(\tilde{X}))$ with $\eta(T) = 0$ and rewrite the weak formulation~\eqref{pdeweak} as
\begin{equation} \label{E4}
-\int_0^T \langle \br, \eta_t \rangle \ \diff t + \int_0^T \int_{\tilde{X}} \left( \frac{\sigma^2}{2} \br_x + \br \ G_{\br}\right) \ \eta_x \ \diff x \diff t = \int_{\tilde{X}} \br (x,0) \ \eta (x,0) \ \diff x. 
\end{equation}
Similarly, since $\rho_{n_k}(x,0) = \rho_0(x)$, we have
\begin{equation} \label{E3}
-\int_0^T \langle \rho_{n_k}, \eta_t \rangle \ \diff t + \int_0^T \int_{\tilde{X}} \left( \frac{\sigma^2}{2} \partial_x \rho_{n_k} + \rho_{n_k} G_{\rho_{n_k}} \right) \ \eta_x \ \diff x \diff t = \int_{\tilde{X}} \rho_0 (x) \ \eta (x,0) \ \diff x. 
\end{equation}
Let $k \rightarrow \infty$ in~\eqref{E3}, so for arbitrary $\eta (x,0)$ we obtain from~\eqref{E3} and~\eqref{E4} that 
\begin{equation*}
\int_{\tilde{X}} \br (x,0) \ \eta (x,0) \ \diff x  = \int_{\tilde{X}} \rho_0 (x) \ \eta (x,0) \ \diff x,
\end{equation*}
which implies $\br (x,0) = \rho_0 (x)$.

\textit{Relaxed regularity on data.} 
In order to relax regularity assumption on data to $\rho_0, \rho_r \in L_{per}^2(\tilde{X}) \cap \mathcal{P}_{e}(\tilde{X})$, we can use the mollified version of the distributions $\rho^{\epsilon}_0 = \phi_{\epsilon} \star \rho_0$ and $\rho^{\epsilon}_r = \phi_{\epsilon} \star \rho_r$ with the standard positive mollifier $\phi_{\epsilon}$, follow the same procedure and take the limit $\epsilon \rightarrow 0$ at the end. 
See also~\cite[Theorem 3.12]{Chazelle15} for the details of this process.

\textit{Uniquness.} Let $\xi = \br_1 - \br_2$ where $\br_1$ and $\br_2$ are two weak solutions to~\eqref{pde2} with $\rho_0, \rho_r \in L_{per}^2(\tilde{X}) \cap \mathcal{P}_{e}(\tilde{X})$. Then, for every $\eta \in L^2(0,T;H^1_{per}(\tilde{X}))$ we have
\begin{align*}
\int_0^T \langle \eta , \xi_t \rangle \ \diff t &+ \frac{\sigma^2}{2} \int_0^T \int_{\tilde{X}} \xi_x \ \eta_x \ \diff x \diff t + \int_0^T \int_{\tilde{X}} ( \br_1 \ G_{\br_1} - \br_2 \ G_{\br_2} ) \ \eta_x \ \diff x \diff t = 0.
\end{align*} 
We can rewrite the last integrand as
\begin{align*}
\br_1 \ G_{\br_1} - \br_2 \ G_{\br_2} & = \br_1 (w \star ( \br_1 + M \rho_r)) -  \br_2 (w \star ( \br_2 + M \rho_r)) \\
& = (\br_1 - \br_2)(w \star ( \br_1 + M \rho_r)) + \br_2 (w \star (\br_1 - \br_2)) \\
&=  \xi \ G_{\br_1} + \br_2 \ (w \star \xi),
\end{align*}
to obtain
\begin{equation} \label{U1}
\int_0^T \langle \eta , \xi_t \rangle \ \diff t + \frac{\sigma^2}{2} \int_0^T \int_{\tilde{X}} \xi_x \ \eta_x \ \diff x \diff t = - \int_0^T \int_{\tilde{X}} \xi \ G_{\br_1} \ \eta_x \ \diff x \diff t - \int_0^T \int_{\tilde{X}} \br_2 \ (w \star \xi) \ \eta_x \ \diff x \diff t .
\end{equation}
Now, for the first integral on the r.h.s. of~\eqref{U1}, we have
\begin{align} \label{U2}
\left| \int_0^T \int_{\tilde{X}} \xi \ G_{\br_1} \ \eta_x \ \diff x \diff t \right| &\leq 2R(1+M) \ \Vert \xi \Vert_{L^2(0,T;L^2)} \ \Vert \eta_x \Vert_{L^2(0,T;L^2)} \nonumber \\
&\leq \frac{\sigma^2}{4} \ \Vert \eta_x \Vert^2_{L^2(0,T;L^2)} + C_1 \ \Vert \xi \Vert^2_{L^2(0,T;L^2)},
\end{align}
where for the first inequality we used~\eqref{G1} in Lemma~\ref{Lem: G} and Cauchy-Schwarz inequality, and for the second inequality we used Young's inequality. 
Similarly, for the the second integral on the r.h.s. of~\eqref{U1}, we have
\begin{align} \label{U3}
\left| \int_0^T \int_{\tilde{X}} \br_2 \ (w \star \xi) \ \eta_x \ \diff x \diff t \right| &\leq \Vert \br_2 \Vert_{L^{\infty}(0,T;L^2)} \ \Vert \eta_x \Vert_{L^2(0,T;L^2)} \ \Vert w \star \xi \Vert_{L^2(0,T;L^2)} \nonumber \\
&\leq C_2[T] \ \Vert \rho_0 \Vert_{L^2} \ \Vert \eta_x \Vert_{L^2(0,T;L^2)} \ \Vert \xi \Vert_{L^2(0,T;L^2)} \nonumber \\
&\leq \frac{\sigma^2}{4} \ \Vert \eta_x \Vert^2_{L^2(0,T;L^2)} + C_2[T] \ \Vert \rho_0 \Vert^2_{L^2} \ \Vert \xi \Vert^2_{L^2(0,T;L^2)},
\end{align}
where for the second inequality we used~\eqref{estim} and Lemma~\ref{Lem: G} (see~\eqref{G2} and~\eqref{finmeasure}). 
Using~\eqref{U2} and~\eqref{U3} for~\eqref{U1} and setting $\eta = \xi$, we obtain
\begin{equation*}
\int_0^T \langle \xi , \xi_t \rangle \ \diff t \leq \left( C_1+C_2[T] \ \Vert \rho_0 \Vert^2_{L^2} \right)  \ \Vert \xi \Vert^2_{L^2(0,T;L^2)}.
\end{equation*}
By~\cite[Theorem 3.8]{Chazelle15}, we know 
$$ \langle \xi , \xi_t \rangle = \frac{1}{2} \frac{\diff}{\diff t} \Vert \xi (t) \Vert^2_{L^2}.$$
Thus, for all $T$, we have
\begin{equation*}
\frac{1}{2} \int_0^T \frac{\diff}{\diff t} \Vert \xi (t) \Vert^2_{L^2} \ \diff t \leq \left( C_1+C_2[T] \ \Vert \rho_0 \Vert^2_{L^2} \right)  \ \int_0^T \Vert \xi (t) \Vert^2_{L^2} \diff t.
\end{equation*}
This implies, for a.e. $t \in [0,T]$
\begin{equation*}
\frac{\diff}{\diff t} \Vert \xi (t) \Vert^2_{L^2} \leq C[T, \rho_0] \ \Vert \xi (t) \Vert^2_{L^2}.
\end{equation*}
Hence, by  Gr\"{o}nwall's inequality, 
\begin{equation*} \label{U4}
\Vert \xi (t) \Vert^2_{L^2} \leq C[T, \rho_0] \ \Vert \xi (0) \Vert^2_{L^2}.
\end{equation*}
This implies $\Vert \xi (t) \Vert_{L^2} = \Vert \br_1(t) - \br_2(t) \Vert_{L^2} = 0$ since $\xi (0) = \rho_0 - \rho_0 = 0$. 
Then, from continuity of $\br_1$ and $\br_2$ in time (by~\cite[Theorem 3.8]{Chazelle15}), we obtain uniqueness. That is, $\br_1 = \br_2$ for all $t \in [0,T]$.

\textit{Regularity.} Here, we first mollify the problem data $\rho_0$ and $\rho_r$ with the standard positive mollifier $\phi_{\epsilon}$ so that the solutions $\{\rho_n : \ n\geq 0 \}$ to~\eqref{pdeseq} are all smooth. This allows us to take derivatives of~\eqref{pdeseq} to any order. We then take the limit $\epsilon \rightarrow \infty$ at the end. For simplicity, we omit the arguments for this last step and drop the subscript $\epsilon$. 

Employing Lemma~\ref{Lem: pG}, we can extend the improved regularity results in space in~\cite{Chazelle15} (Theorem 4.2). 
That is, for $\rho_0 \in H_{per}^k(\tilde{X}) \cap \mathcal{P}_{e}(\tilde{X})$ and $\rho_r \in H_{per}^{k-1}(\tilde{X}) \cap \mathcal{P}_{e}(\tilde{X})$, we have
\begin{equation} \label{R2}
\br \in L^2(0,T;H^{k+1}_{per}(\tilde{X})) \cap L^{\infty}(0,T;H^{k}_{per}(\tilde{X})).
\end{equation}
Moreover, since $\rho_r$ is constant in time, we can also employ the results on improved regularity in time provided by~\cite[Theorem 4.3]{Chazelle15} for our model. This means, for $\rho_0 \in H_{per}^{2k}(\tilde{X}) \cap \mathcal{P}_{e}(\tilde{X})$ and $\rho_r \in L_{per}^{2}(\tilde{X}) \cap \mathcal{P}_{e}(\tilde{X})$, we have for $i \leq k$
\begin{equation} \label{R3}
\partial^i_t \br \in L^2(0,T;H^{2k-2i+1}_{per}(\tilde{X})) \cap L^{\infty}(0,T;H^{2k-2i}_{per}(\tilde{X})),
\end{equation}
and
\begin{equation} \label{R4}
\partial^{k+1}_t \br \in L^2(0,T;H^{-1}_{per}(\tilde{X})).
\end{equation}
With these regularity results in space and time, we can derive the required regularity on the solution as stated in Theorem~\ref{wellposed}. 
Let $\rho_0 \in H_{per}^3(\tilde{X}) \cap \mathcal{P}_{e}(\tilde{X})$ and $\rho_r \in H_{per}^{2}(\tilde{X}) \cap \mathcal{P}_{e}(\tilde{X})$ and  $\br$ be the unique weak solution to PDE~\eqref{pde2}. 
By~\eqref{R2}, we have $\br \in L^{\infty}(0,T;H^{3}_{per}(\tilde{X}))$. 
Hence, by Sobolev embedding theorem~\cite[Section 4.12]{SobolevAdams}, we have $\br(t) \in C_{per}^2(\bar{\tilde{X}})$ (after possibly being redefined on a set of measure zero). This gives the required regularity in space. 
Also,~\eqref{R3} and~\eqref{R4} imply that $\br_t \in L^2(0,T;H^{1}_{per}(\tilde{X}))$ and $\br_{tt} \in L^2(0,T;H^{-1}_{per}(\tilde{X}))$. 
Hence, by~\cite[Theorem 3.8]{Chazelle15}, we have $\br_t \in C(0,T;L^2_{per}(\tilde{X}))$ (after possibly being redefined on a set of measure zero). This gives the required regularity in time.
Putting these results together, we have $\br \in C^1(0,T;C_{per}^2(\tilde{X}))$.

\textit{Evenness.} The evenness imposed on $\rho_0$ and $\rho_r$ implies that if $\rho(x,t)$ is a solution of~\eqref{pde2}, then $\rho(-x,t)$ is also a solution. Indeed, from~\eqref{pde2} we obtain
\begin{align*}
\partial_t \rho (-x,t) - \frac{\sigma^2}{2} \partial^2_{x} \rho (-x,t) &= \partial_x \left( \rho (-x,t) \int w(-x-y) \ (\rho(y,t) + M \rho_r(y)) \ \diff y \right) \\
&= \partial_x \left( \rho (-x,t) \int w(-x+y) \ (\rho(-y,t) + M \rho_r(-y)) \ (-\diff y) \right) \\
&= \partial_x \left( \rho (-x,t) \int -w(-x+y) \ (\rho(-y,t) + M \rho_r(y)) \ \diff y \right) \\
&= \partial_x \left( \rho (-x,t) \int w(x-y) \ (\rho(-y,t) + M \rho_r(y)) \ \diff y \right),
\end{align*}
where for that last equality we used the fact that $w$ is an odd function.
Then, assuming $\rho_0 \in H_{ep}^3(\tilde{X}) \cap \mathcal{P}_{e}(\tilde{X})$ and $\rho_r \in H_{ep}^{2}(\tilde{X}) \cap \mathcal{P}_{e}(\tilde{X})$ (notice that $H_{ep}^{k}(\tilde{X}) \subset H_{per}^{k}(\tilde{X})$), the \emph{uniqueness} of the solution $\br \in C^1(0,T;C_{per}^2(\tilde{X}))$ to PDE~\eqref{pde2} implies that the solution is even, that is, $\br \in C^1(0,T;C_{ep}^2(\tilde{X}))$.

\textit{Positivity.} Using the same approach as in~\cite{Carrillo18}, we consider the following version of~\eqref{pde2} in the unknown function $\rho$ with $\br$ being the non-negative weak solution
\begin{equation*}
\rho_t = (\rho \ G_{\br})_x+\frac{\sigma^2}{2} \rho_{xx}.
\end{equation*}
This is a linear parabolic PDE with smooth and bounded coefficients (by Lemmas~\ref{Lem: G} and~\ref{Lem: Gx}) for which $\br$ is a classic non-negative solution. 
Thus, by parabolic Harnack inequality~\cite[Section 7.1.4, Theorem 10]{Evans}, we have
$$\sup_{x \in \tilde{X}} \br(x,t_1) \leq c \inf_{x \in \tilde{X}} \br(x,t_2),$$
for $0 < t_1 < t_2 < \infty$ and some positive constant $c$.
Non-negativity of $\br(x,t)$ implies that $\inf_{x \in \tilde{X}} \br (x,t)$ and hence $\br(x,t)$ is strictly positive for all $t>0$.
\end{proof}
%===============================================================================
\section{Stationary Behavior and Stability} 
\label{sec:stationary and stability}
%===============================================================================

\subsection{Existence of Stationary Solution} 
\label{sec:stationary existence}

This section mainly concerns the proof of \emph{existence} result in Theorem~\ref{StatinarySolution} for stationary equation~\eqref{pdes}. 
All the norms in this section are w.r.t. $X = [0,1]$  (as opposed to $\tilde{X} = [-1,1]$), unless indicated otherwise. We note that norms on the even $2$-periodic spaces computed w.r.t. to $X$ and $\tilde{X}$
differ by a multiplicative constant, e.g., $\Vert u \Vert_{L^p(\tilde{X})} = 2^{\frac{1}{p}} \Vert u \Vert_{L^p(X)}$. 
We again use $C, C_0, C_1, \ldots$ to represent a generic constant (depending on the model parameters) which actual values may change from line to line. 
In case these constants depend on a particular object of interest, say $\theta$, this dependence is explicitly indicated by $C[\theta]$.

Let us begin with providing a fixed point characterization of the solution to stationary equation~\eqref{pdes}. We note that, corresponding to the solution to dynamic equation~\eqref{pde2}, we are particularly interested in \emph{even} solutions $\rho^s \in \mathcal{P}_{e}(\tilde{X})$ of stationary equation~\eqref{pdes}.

\begin{Lem} [Fixed point characterization] \label{Lem: fixed point}
$\rho^s \in C^{2}_{ep}(\tilde{X}) \cap \mathcal{P}_{e}(\tilde{X})$ is a solution of stationary equation~\eqref{pdes} if and only if $\rho^s$ is a fixed point of the operator $\mathcal{T}: \mathcal{P}_{e}(\tilde{X}) \rightarrow \mathcal{P}_{e}(\tilde{X})$ defined by
\begin{equation} \label{operator}
\mathcal{T} \rho := \frac{1}{K} \exp \left( -\frac{2}{\sigma^2}\int_{0}^{x} G_{\rho}(z) \ \diff z \right),
\end{equation}
where the constant $K$ is determined by the normalizing condition
$$
K = \int_{0}^{1} \exp \left(-\frac{2}{\sigma^2}\int_{0}^{x} G_{\rho}(z) \ \diff z \right) \diff x.
$$
\end{Lem}
\begin{proof}
The ``if'' part is clear since any fixed point $\rho^s \in C^{2}_{ep}(\tilde{X})$ satisfies the stationary equation~\eqref{pdes}. For the ``only if'' part, note that integrating~\eqref{pdes} once, we have
\begin{equation} \label{temp1}
\frac{\sigma^2}{2}\rho_x+\rho \ G_{\rho}=C.
\end{equation}
Now notice that we can set $C=0$ since we are interested in \emph{even} solutions to~\eqref{temp1}. 
Indeed, from~\eqref{temp1} we have
\begin{equation*}
\frac{\sigma^2}{2}\rho_x (-x)+\rho (-x) [w(-x) \star (\rho(-x) + M \rho_r(-x))]=C.
\end{equation*}
Hence, for an even solution, we obtain
\begin{equation*}
-\frac{\sigma^2}{2}\rho_x (x)-\rho (x) [w(x) \star (\rho(x) + M \rho_r(x))]=C,
\end{equation*}
where we used the fact that $w$ is an odd function. This implies $C=0$.
Rearranging and integrating~\eqref{temp1} once again, we have
\begin{equation}\label{StaSol}
\rho(x)= \frac{1}{K} \exp \left(-\frac{2}{\sigma^2}\int_{0}^{x} G_{\rho}(z) \ \diff z\right),
\end{equation}
where the normalizing condition gives the constant $K$ as
$$
K = \int_{0}^{1} \exp \left(-\frac{2}{\sigma^2}\int_{0}^{x} G_{\rho}(z) \ \diff z\right)\diff x.
$$
This completes the proof.
\end{proof}

This characterization allows us to use tools from operator theory, in particular, Schauder fixed point theorem to derive existence result for the stationary solution. Before that, we present some preliminary results for the operator $\mathcal{T}$.

\begin{Lem} [Estimates for $\mathcal{T}$] \label{Lem: T}
Let $\mathcal{T}$ be the operator on $\mathcal{P}_{e}(\tilde{X})$ defined by~\eqref{operator}. 
\begin{itemize}
\item If $\rho, \rho_r \in \mathcal{P}_{e}(\tilde{X})$, then 
\begin{equation} \label{T1}
\Vert \mathcal{T} \rho \Vert_{L^{\infty}} \leq \exp \left( \frac{8R(1+M)}{\sigma^2} \right),
\end{equation}
and
\begin{equation} \label{T2}
\Vert \partial_x \mathcal{T} \rho \Vert_{L^{\infty}} \leq \frac{4R(1+M)}{\sigma^2} \exp \left( \frac{8R(1+M)}{\sigma^2} \right).
\end{equation}
\item If $\rho, \rho_r \in L_{ep}^2(\tilde{X}) \cap \mathcal{P}_{e}(\tilde{X})$, then
\begin{equation} \label{T3}
\Vert \mathcal{T} \rho \Vert_{H^{2}} \leq C[\Vert \rho_r \Vert_{L^{2}}] \ \Vert \rho \Vert_{L^{2}}.
\end{equation}
\item If $\rho, \rho_r \in H_{ep}^{k-2}(\tilde{X}) \cap \mathcal{P}_{e}(\tilde{X})$, then for $k \geq 3$
\begin{equation} \label{T4}
\Vert \mathcal{T} \rho \Vert_{H^{k}} \leq \sum_{i=1}^{k-1} C_i[\Vert \rho_r \Vert_{H^{k-2}}] \ \Vert \rho \Vert^i_{H^{k-2}}.
\end{equation}
\end{itemize}
\end{Lem}
\begin{proof}
From the definition~\eqref{operator} and inequality~\eqref{G1} in Lemma~\ref{Lem: G} we obtain
\begin{align*} 
|\mathcal{T} \rho| &= \frac{\exp \left\{-\frac{2}{\sigma^2}\int_{0}^{x} G_{\rho}(z) \ \diff z\right\}}{\int_{0}^{1}\exp \left\{-\frac{2}{\sigma^2}\int_{0}^{x} G_{\rho}(z) \ \diff z\right\}\diff x}  \leq \frac{\exp \left\{ \frac{4R(1+M)}{\sigma^2} \right\} }{ \exp \left\{ -\frac{4R(1+M)}{\sigma^2} \right\} } =  \exp \left\{ \frac{8R(1+M)}{\sigma^2} \right\},
\end{align*}
which gives the estimate~\eqref{T1}.

Now, observe
\begin{align*} 
\Vert \partial_x \mathcal{T} \rho \Vert_{L^{\infty}} &=  \left\Vert -\frac{2}{\sigma^2} \ G_{\rho} \ \mathcal{T} \rho \right\Vert_{L^{\infty}} 
 \leq \frac{2}{\sigma^2} \ \Vert G_{\rho} \Vert_{L^{\infty}} \ \Vert \mathcal{T} \rho \Vert_{L^{\infty}}.
\end{align*}
Using~\eqref{G1} in Lemma~\ref{Lem: G} and~\eqref{T1}, we obtain the inequality~\eqref{T2}.

For the inequality~\eqref{T3}, first notice
\begin{align} \label{tempT1}
\Vert \mathcal{T} \rho \Vert_{H^{2}} &\leq C \ \left( \Vert \mathcal{T} \rho \Vert_{L^{2}} +\Vert \partial_x^2 \mathcal{T} \rho \Vert_{L^{2}} \right) \leq C_1 + C_2 \ \Vert \partial_x^2 \mathcal{T} \rho \Vert_{L^{2}},
\end{align}
where for the second inequality we used the fact that $\Vert \mathcal{T} \rho \Vert_{L^{2}} \leq C \ \Vert \mathcal{T} \rho \Vert_{L^{\infty}}$ is bounded by~\eqref{T1}. Also, we have
\begin{align*}
\Vert \partial_x^2 \mathcal{T} \rho \Vert_{L^{2}} &= \bigg\Vert -\frac{2}{\sigma^2} \ \left( \mathcal{T} \rho \ \partial_x G_{\rho} +  G_{\rho} \ \partial_x \mathcal{T} \rho \right) \bigg\Vert_{L^{2}}  \nonumber \\
&\leq C \ \left( \Vert \mathcal{T} \rho \Vert_{L^{\infty}} \  \Vert \partial_x G_{\rho} \Vert_{L^{2}} + \Vert G_{\rho} \Vert_{L^{\infty}} \  \Vert \partial_x \mathcal{T} \rho \Vert_{L^{2}}  \right) \nonumber \\
&\leq C[\Vert \rho_r \Vert_{L^{2}}] \ \Vert \rho \Vert_{L^{2}} + C_2 \nonumber \\
&\leq C[\Vert \rho_r \Vert_{L^{2}}] \ \Vert \rho \Vert_{L^{2}},
\end{align*}
where for the second inequality we used~\eqref{Gx1} in Lemma~\ref{Lem: Gx} and the last inequality follows from the fact that  
$\Vert \rho \Vert_{L^{2}} \geq \Vert  \rho \Vert_{L^{1}} > 0$ (see~\eqref{finmeasure}). Inserting this result in~\eqref{tempT1}, we obtain inequality~\eqref{T3}.

Similarly, for $k \geq 3$, we have (see~\eqref{tempT1})
\begin{align} \label{tempT2}
\Vert \mathcal{T} \rho \Vert_{H^{k}} \leq C_1 + C_2 \ \Vert \partial_x^k \mathcal{T} \rho \Vert_{L^{2}}.
\end{align}
Now, notice
\begin{align*} 
\Vert \partial_x^k \mathcal{T} \rho \Vert_{L^{2}} &= \Vert \partial_x^{k-1} \partial_x  \mathcal{T} \rho \Vert_{L^{2}} = \bigg\Vert \partial_x^{k-1} \left( -\frac{2}{\sigma^2} \ G_{\rho} \ \mathcal{T} \rho \right) \bigg\Vert_{L^{2}} = \frac{2}{\sigma^2} \ \Vert \partial_x^{k-1} \left( \mathcal{T} \rho \  G_{\rho}   \right) \Vert_{L^{2}} \\
&\leq C \ \Vert  \mathcal{T} \rho \  G_{\rho} \Vert_{H^{k-1}} \leq C[\Vert \rho_r \Vert_{H^{k-2}}] \ \Vert \rho  \Vert_{H^{k-2}} \ \Vert \mathcal{T} \rho  \Vert_{H^{k-1}},
\end{align*}
where for the last inequality we used Lemma~\ref{Lem: pG}. Combining this result with~\eqref{tempT2}, we derive a recursive inequality. Performing the recursive computations while keeping the highest Sobolev norms, we obtain
\begin{equation*}
\Vert \mathcal{T} \rho \Vert_{H^{k}} \leq C_0 + \sum_{i=1}^{k-1} C_i[\Vert \rho_r \Vert_{H^{k-2}}] \ \Vert \rho \Vert^i_{H^{k-2}}.
\end{equation*}
Then, since $ \Vert \rho \Vert^i_{H^{k-2}} \geq \Vert \rho \Vert_{L^{2}} \geq C \ \Vert  \rho \Vert_{L^{1}} > 0$, we can remove the constant $C_0$ and consider its effect in constants $C_i$. This gives the desired inequality~\eqref{T4}.
\end{proof}

\begin{Prop} [Lipschitz continuity of $\mathcal{T}$] \label{Prop: LipT}
Let $\mathcal{T}$ be the operator on $\mathcal{P}_{e}(\tilde{X})$ defined by~\eqref{operator} with $\rho_r \in \mathcal{P}_{e}(\tilde{X})$. Then $\mathcal{T}$ is Lipschitz continuous in $L^p$ for $1 \leq p < \infty$ with Lipschitz constant
\begin{equation} \label{LipConst}
L_{\mathcal{T}} = \frac{1}{2} \ \exp \left\{ \left( \frac{8R(1+M)}{\sigma^2} \right) \left(1-\frac{1}{p} \right) \right\} \ \left( \exp \left\{ \frac{16R}{\sigma^2} \right\} - 1 \right).
\end{equation}
\end{Prop}
\begin{proof}
We use a similar argument to the one provided by~\cite{Nordio18}. 
Let $\rho_1, \rho_2 \in \mathcal{P}_{e}(\tilde{X})$. Using the estimate~\eqref{T1} in Lemma~\ref{Lem: T}, we have for $1 \leq p < \infty$ 
\begin{align} \label{Lip1}
\Vert \mathcal{T} \rho_2 - \mathcal{T} \rho_1 \Vert_{L^p} 
&= \bigg\Vert \mathcal{T} \rho_1 \left( \frac{\mathcal{T} \rho_2}{\mathcal{T} \rho_1}  - 1 \right)  \bigg\Vert_{L^p} \leq \Vert \mathcal{T} \rho_1 \Vert_{L^p} \ \bigg\Vert \frac{\mathcal{T} \rho_2}{\mathcal{T} \rho_1}  - 1 \bigg\Vert_{L^{\infty}} \nonumber \\
&\leq \ \Vert \mathcal{T} \rho_1 \Vert_{L^{\infty}}^{1-\frac{1}{p}} \  \bigg\Vert \frac{K_1}{K_2} \ \exp \left\{-\frac{2}{\sigma^2}\int_{0}^{x} w \star (\rho_2 - \rho_1) \ \diff z\right\}  - 1 \bigg\Vert_{L^{\infty}},
\end{align}
where for the last inequality we used $\Vert \mathcal{T} \rho \Vert_{L^1(X)} = 1$. Now, define 
$$\Gamma(\rho_1 - \rho_2) := \frac{2}{\sigma^2} \int_{0}^{x} w \star (\rho_1 - \rho_2) \ \diff z,$$
and observe
\begin{align} \label{Lip2}
\left| \Gamma(\rho_2 - \rho_1) \right| &= \frac{2}{\sigma^2} \left| \int_{0}^{x} \int (z-y) \ \textbf{1}_{|z-y|\leq R} \ (\rho_2(y)-\rho_1(y)) \ \diff y \diff z \right| \nonumber \\
&\leq \frac{2}{\sigma^2} \int_{0}^{x} \int |(z-y)| \ \textbf{1}_{|z-y|\leq R} \ |\rho_2(y)-\rho_1(y)| \ \diff y\diff z \nonumber \\
&\leq \frac{2R}{\sigma^2} \int_{0}^{x} \int_{\tilde{X}} |\rho_2(y)-\rho_1(y)| \ \diff y\diff z 
\leq \frac{4R}{\sigma^2} \ \Vert \rho_2 - \rho_1 \Vert_{L^1}.
\end{align}
Similarly, we can write the normalizing constant $K_1$ as
\begin{align*}
K_1 &= \int_{0}^{1} \exp \left(-\frac{2}{\sigma^2}\int_{0}^{x} G_{\rho_1} \ \diff z\right) \diff x \nonumber = \int_{0}^{1} \exp \left(-\frac{2}{\sigma^2}\int_{0}^{x}G_{\rho_2} \ \diff z\right) \exp\left\{-\Gamma(\rho_1 - \rho_2) \right\} \ \diff x \nonumber .
\end{align*}
From~\eqref{Lip2}, it follows
\begin{align*}
K_1 &\leq \int_{0}^{1} \exp\left(-\frac{2}{\sigma^2}\int_{0}^{x} G_{\rho_2} \ \diff z\right) \ \exp \left(\frac{4R}{\sigma^2} \ \Vert \rho_2 - \rho_1 \Vert_{L^1} \right) \ \diff x = K_2 \exp \left(\frac{4R}{\sigma^2} \ \Vert \rho_2 - \rho_1 \Vert_{L^1} \right),
\end{align*}
and
\begin{align*}
K_1 &\geq \int_{0}^{1} \exp\left(-\frac{2}{\sigma^2}\int_{0}^{x} G_{\rho_2} \ \diff z\right) \ \exp \left(-\frac{4R}{\sigma^2} \ \Vert \rho_2 - \rho_1 \Vert_{L^1} \right) \ \diff x = K_2 \exp \left(-\frac{4R}{\sigma^2} \ \Vert \rho_2 - \rho_1 \Vert_{L^1} \right) .
\end{align*}
Hence,
\begin{equation} \label{Lip3}
\exp \left(-\frac{4R}{\sigma^2} \ \Vert \rho_2 - \rho_1 \Vert_{L^1} \right) \leq \frac{K_1}{K_2} \leq \exp \left(\frac{4R}{\sigma^2} \ \Vert \rho_2 - \rho_1 \Vert_{L^1} \right).
\end{equation}
Using~\eqref{Lip2} and~\eqref{Lip3}, we can rewrite~\eqref{Lip1} as
\begin{align*}
\Vert \mathcal{T} \rho_2 - \mathcal{T} &\rho_1 \Vert_{L^p} \leq \Vert \mathcal{T} \rho_1 \Vert_{L^{\infty}}^{1-\frac{1}{p}} \ \max \left\{ \exp \left(\frac{8R}{\sigma^2} \Vert \rho_2 - \rho_1 \Vert_{L^1} \right)  - 1, \ 1 - \exp \left(-\frac{8R}{\sigma^2} \Vert \rho_2 - \rho_1 \Vert_{L^1} \right)  \right\}.
\end{align*}
Hence,
\begin{align} \label{Lip4}
\Vert \mathcal{T} \rho_2 - \mathcal{T} \rho_1 \Vert_{L^p} \leq \Vert \mathcal{T} \rho_1 \Vert_{L^{\infty}}^{1-\frac{1}{p}} \ \left( \exp \left(\frac{8R}{\sigma^2} \Vert \rho_2 - \rho_1 \Vert_{L^1} \right)  - 1 \right).
\end{align}
Now, notice that
$$\Vert \rho_2 - \rho_1 \Vert_{L^1} \leq \Vert \rho_2 \Vert_{L^1} + \Vert \rho_1 \Vert_{L^1} = 2,$$ 
(recall that norms are defined over $X$) and for $a > 0$
$$\e^{ax} - 1 \leq \frac{1}{2} (\e^{2a} - 1)x, \quad \forall x \in [0,2].$$
Thus, we have
\begin{equation} \label{Lip5}
\exp \left( \frac{8R}{\sigma^2} \Vert \rho_2 - \rho_1 \Vert_{L^1} \right)  - 1  \leq \frac{1}{2} \left( \e^{\frac{16R}{\sigma^2}} - 1 \right) \Vert \rho_2 - \rho_1 \Vert_{L^1} .
\end{equation}
Combining~\eqref{Lip4} and~\eqref{Lip5}, we obtain
\begin{align*}
\Vert \mathcal{T} \rho_2 - \mathcal{T} \rho_1 \Vert_{L^p} &\leq \frac{1}{2} \ \Vert \mathcal{T} \rho_1 \Vert_{L^{\infty}}^{1-\frac{1}{p}} \ \left( \e^{\frac{16R}{\sigma^2}} - 1 \right) \Vert \rho_2 - \rho_1 \Vert_{L^1}.
\end{align*}
Finally, using~\eqref{T1} in Lemma~\ref{Lem: T} and the inequality~\eqref{finmeasure} which relates norms over domains of finite measure, we have  
$$\Vert \mathcal{T} \rho_2 - \mathcal{T} \rho_1 \Vert_{L^p} \leq L_{\mathcal{T}} \ \Vert \rho_2 - \rho_1 \Vert_{L^p},$$
where the constant $L_{\mathcal{T}}$ is given by~\eqref{LipConst}.
\end{proof} 

With this preliminary results in hand, we next move on to the proof of existence result in Theorem~\ref{StatinarySolution}.

\begin{proof}[Proof of Theorem~\ref{StatinarySolution} (Existence)]

Following essentially the same argument as in~\cite[Theorem 2.3]{Carrillo18} and using Lemma~\ref{Lem: fixed point}, we can present the existence result for the stationary solution as the fixed point of the operator~$\mathcal{T}$. 
First note that using the estimate~\eqref{T1} in Lemma~\ref{Lem: T}, we have 
$\Vert \mathcal{T} \rho \Vert_{L^2} \leq C \ \Vert \mathcal{T} \rho \Vert_{L^{\infty}} \leq c$
for some positive constant $c$.
Thus, for the purpose of finding the fixed points of $\mathcal{T}$, we can restrict $\mathcal{T}$ to act on the closed and convex set 
$E := \{ \rho \in L^2_{ep}(\tilde{X}) \cap \mathcal{P}_{e}(\tilde{X}) : \Vert \rho \Vert_{L^2} \leq c \}$. 
Now, notice that, using inequalities~\eqref{T1} and~\eqref{T2} in Lemma~\ref{Lem: T}, we have for any $\rho \in E$
\begin{equation} \label{Scha}
\Vert \mathcal{T} \rho \Vert_{H^1}^2 \leq \Vert \mathcal{T} \rho \Vert_{L^{2}}^2 + \Vert \partial_x \mathcal{T} \rho \Vert_{L^{2}}^2 \leq C_1 \Vert \mathcal{T} \rho \Vert_{L^{\infty}}^2 + C_2 \Vert \partial_x \mathcal{T} \rho \Vert_{L^{\infty}}^2 \leq c',
\end{equation}
for some constant $c'> 0$. 
That is, $\mathcal{T} (E) \subset E$ is uniformly bounded in $H^1_{ep}(\tilde{X})$. 
Thus, by the Rellich-Kondrachov compactness theorem~\cite[Section 5.7, Theorem 1]{Evans}, $\mathcal{T} (E)$ is precompact in $L^2_{ep}(\tilde{X})$.
Since $E \subset L^2_{ep}(\tilde{X})$ is closed, this implies $\mathcal{T} (E)$ is also precompact in $E$. 
Also, $\mathcal{T}$ is Lipschitz continuous by Proposition~\ref{Prop: LipT}. 
Hence, by Schauder fixed point theorem~\cite[Section 9.2.2, Theorem 3]{Evans}, it has a fixed point $\rho^s \in E$ which by~\eqref{Scha} belongs to $H^1_{ep}(\tilde{X})$.

\textit{Regularity.} 
The estimate~\eqref{T4} in Lemma~\eqref{Lem: T} implies that if $\rho_r  \in H^{k-2}_{ep} (\tilde{X})$, then the fixed point $\rho^s = \mathcal{T} \rho^s \in H^{k}_{ep} (\tilde{X})$. In particular, if $\rho_r  \in H^{1}_{ep} (\tilde{X})$, then $\rho^s \in H^{3}_{ep} (\tilde{X})$. Hence, by Sobolev embedding theorem~\cite[Section 4.12]{SobolevAdams}, $\rho \in C^{2}_{ep}(\ol{\tilde{X}})$ (after possibly being redefined on a set of measure zero).

\textit{Positivity.} The positivity of the fixed point directly follows from the representation~\eqref{operator}.
\end{proof}

\begin{Rem} [Uniqueness]
By Proposition~\ref{Prop: LipT}, $\mathcal{T}$ is Lipschitz continuous in $L^p$ with Lipschitz constant $L_{\mathcal{T}}$ given by~\eqref{LipConst}, and thus, is a contraction for $L_{\mathcal{T}} < 1$. Hence, by Banach fixed-point theorem~\cite[Section 9.2.1, Theorem 1]{Evans}, $\mathcal{T}$ has a unique fixed point for $L_{\mathcal{T}} < 1$. Setting $p=1$ in~\eqref{LipConst} gives the sufficient condition $\sigma^2 > \frac{16R}{\ln 3}$ for \emph{uniqueness} of stationary solution. This result corresponds to the sufficient condition provided in~\cite[Theorem 2]{Nordio18}.
\end{Rem}

\begin{Rem} [Semi-Gaussian clusters] \label{Rem: approx}
For a highly concentrated radical opinion distribution with average opinion $A = \int_X x \ \rho_r (x) \ \diff x$, we can provide an approximate solution to the stationary equation~\eqref{pdes} as follows
\begin{equation} \label{StaSol4}
\rho^s(x) = \frac{1}{K} \exp \left\{ -\frac{M+1}{\sigma^2} \min \left\{(x - A)^2, \ R^2 \right\} \right\},
\end{equation}
where $K$ is the normalizing constant (see Appendix~\ref{Append:approx} for the details). 
This result is an extension of the approximate solution provided by~\cite[Section 5.2]{Wang17}. In particular, one can reproduce the same result by setting $M=0$ and $A=0$. 
Equation~\eqref{StaSol4} shows that for highly concentrated radicals the possible accumulation of normals around the average radical opinion $A$ in the stationary state is semi-Gaussian with variance $\frac{\sigma^2}{2(M+1)}$.
Note that, as argued in~\cite{Wang17}, other clusters centered at opinion values other than $x=A$ may also exist. 
As long as these clusters are well-separated so that inter-cluster influences can be ignored, one can use the same approximation to derive a semi-Gaussian profile for the shape of these clusters (set $M=0$ and $A = x_0$ in~\eqref{StaSol4} where $x_0$ denotes the center of the corresponding cluster). 
This analysis shows that $M$ affects the shape of the possible cluster formed at the average radical opinion $A$ in stationary state. 
\end{Rem}

\subsection{Global Estimate for Stationary Solution} 
\label{sec:stationary characterization}

This section is devoted to the proof of the \emph{estimate} given in Theorem~\ref{StatinarySolution}. 
In this section, all the norms are w.r.t. the domain $\tilde{X}=[-1,1]$, unless indicated otherwise. 
 
\begin{proof}[Proof of Theorem~\ref{StatinarySolution} (Estimate)] 
Let $\psi = \rho^s - 1$ so that $\int_X \psi (x) \ \diff x = 0$. From the stationary equation~\eqref{pdes} we obtain
\begin{align*}
- \frac{\sigma^2}{2} \psi_{xx} &=  \left[ (\psi+1) \ G_{\psi+1}\right]_x  =  \left[ (\psi+1) \ ( w \star 1 + G_{\psi} ) \right]_x = \left[ (\psi+1) \ G_{\psi} \right]_x =  \left[ \psi \ G_{\psi} \right]_x + [G_{\psi}]_x,
\end{align*}
where we used the fact that $w \star 1 = 0$. Next, we multiply this last equation by $\psi$ and integrate by part over $\tilde{X}$ to derive
\begin{align*} 
\frac{\sigma^2}{2} \Vert \psi_x \Vert_{L^2}^2 = - \int_{\tilde{X}} \psi_x \ \psi \ G_{\psi} \ \diff x  - \int_{\tilde{X}} \psi_x \ G_{\psi} \ \diff x.
\end{align*}
The extra terms are zero due to periodicity. Thus,
\begin{align} \label{boundPsi}
\frac{\sigma^2}{2} \Vert \psi_x \Vert_{L^2}^2 &\leq \left| \int_{\tilde{X}} \psi_x \ \psi \ G_{\psi} \ \diff x  \right| +  \left| \int_{\tilde{X}} \psi_x \ G_{\psi} \ \diff x  \right| \leq \Vert G_{\psi} \Vert_{L^{\infty}} \ \Vert \psi_x \Vert_{L^2} \  \Vert \psi \Vert_{L^2} + \Vert \psi_x \Vert_{L^2} \  \Vert G_{\psi} \Vert_{L^2}.
\end{align}
Now, using inequality~\eqref{G1} in Lemma~\ref{Lem: G}, we obtain
\begin{align} \label{boundG1}
\Vert G_{\psi} \Vert_{L^{\infty}} &\leq 2R \left( \Vert \psi \Vert_{L^1(X)} + M \right)  = 2R \left( \Vert \rho - 1 \Vert_{L^1(X)} + M \right) \nonumber \\
&\leq 2R \left( \Vert \rho \Vert_{L^1(X)} + 1 + M \right) \leq 2R ( M+2 ).
\end{align}
Also, we have
\begin{align} \label{boundG21}
| G_{\psi} (x) |^2 &= \left( \int w (x-y) \ (\psi (y) + M \rho_r (y)) \ \diff y \right)^2 \nonumber \\
&= \left( \int_{x-R}^{x+R} (x-y) \ (\psi (y) + M \rho_r (y)) \ \diff y \right)^2 \nonumber \\
&\leq   \int_{x-R}^{x+R} (x-y)^2 \ \diff y  \ \int_{x-R}^{x+R} (\psi (y) + M \rho_r (y))^2 \ \diff y \nonumber \\
&\leq \frac{2}{3} R^3 \int_{x-R}^{x+R} (\psi (y) + M \rho_r (y))^2 \ \diff y.
\end{align}
Hence,
\begin{align} \label{boundG22}
\Vert G_{\psi} \Vert_{L^2}^2 &\leq  \frac{2}{3} R^3 \int_{\tilde{X}} \int_{x-R}^{x+R} (\psi (y) + M \rho_r (y))^2 \ \diff y \diff x \nonumber \\
&=  \frac{2}{3} R^3 \int_{\tilde{X}} \int_{-R}^{R} (\psi (x+y) + M \rho_r (x+y))^2 \ \diff y \diff x \nonumber \\
&=  \frac{2}{3} R^3  \int_{-R}^{R} \int_{\tilde{X}} (\psi (x+y) + M \rho_r (x+y))^2 \ \diff x \diff y \nonumber \\
&=  \frac{4}{3} R^4  \Vert \psi + M \rho_r \Vert_{L^2}^2.
\end{align}
Using estimates~\eqref{boundG1} and~\eqref{boundG22}, we can obtain form~\eqref{boundPsi} (recall that uniform distribution is not an equilibrium of the system and hence $\Vert \psi_x \Vert_{L^2} \neq 0$)
\begin{align} \label{boundPsi2}
\frac{\sigma^2}{2} \Vert \psi_x \Vert_{L^2} &\leq 2R ( M+2 ) \Vert \psi \Vert_{L^2} + \frac{2R^2}{\sqrt{3}} \ \Vert \psi + M \rho_r \Vert_{L^2} \nonumber \\
&\leq 2R ( M+2 ) \Vert \psi \Vert_{L^2} + \frac{2R^2}{\sqrt{3}} \left( \Vert \psi \Vert_{L^2} + M \Vert \rho_r \Vert_{L^2} \right) \nonumber \\
&= 2R \left( M+\frac{R}{\sqrt{3}}+2 \right) \Vert \psi \Vert_{L^2} + \frac{2R^2M}{\sqrt{3}}\Vert \rho_r \Vert_{L^2}. 
\end{align}
Now, since  $\int_X \psi (x) \ \diff x = 0$, we can employ Poincar\'e inequality~\cite[Section 5.8.1, Theorem 1]{Evans} to obtain $\Vert \psi \Vert_{L^2} \leq C \ \Vert \psi_x \Vert_{L^2}$. 
The optimal value for the Poincar\'e constant for $\tilde{X} = [-1,1]$ is $C = \frac{1}{\pi}$.
Combining this result with inequality~\eqref{boundPsi2}, we have
\begin{equation}
\left( \sigma^2 - \frac{4R}{\pi} \left( M+\frac{R}{\sqrt{3}}+2 \right) \right) \Vert \psi \Vert_{L^2} \leq \frac{4R^2M}{\pi \sqrt{3}} \Vert \rho_r \Vert_{L^2}.
\end{equation}
Defining $\sigma_b$ and $c_b$ as in~\eqref{sigma_b and c_b} gives the desired inequality
$\Vert \psi \Vert_{L^2} \leq \frac{1}{\eta} \Vert \rho_r \Vert_{L^2}$,
where $\eta = (\sigma^2 - \sigma_b^2)/c_b$.
\end{proof}

%===============================================================================
\subsection{Stability of Stationary State} 
\label{sec:longterm}
%===============================================================================

This section is devoted to the proof of Theorem~\ref{StabilityStatinarySolution} concerning stability of stationary state. 
All the norms in this subsection are w.r.t. the domain $\tilde{X}=[-1,1]$ (as opposed to $X = [0,1]$), unless indicated otherwise.

\begin{proof}[Proof of Theorem~\ref{StabilityStatinarySolution}]
We follow similar arguments as the ones in~\cite{Chazelle15}, except we consider a general stationary state $\rho^s$ (instead of the uniform distribution considered in~\cite{Chazelle15}). 
Let $\psi = \rho - \rho^s$ so that $\int_X \psi (x) \ \diff x = 0$. From the dynamic equation~\eqref{pde2}, we obtain
\begin{align} \label{TempS1}
\psi_t &=  \left[ (\psi+\rho^s) \ G_{\psi+\rho^s}\right]_x + \frac{\sigma^2}{2} \ [\psi+\rho^s]_{xx} \nonumber \\
&=  \left[ (\psi+\rho^s) \ ( w \star \psi + G_{\rho^s} ) \right]_x + \frac{\sigma^2}{2} \ [\psi+\rho^s]_{xx} \nonumber \\
&=  \left[ \psi \ ( w \star \psi + G_{\rho^s} ) \right]_x + \left[ \rho^s \ ( w \star \psi ) \right]_x + \left[ \rho^s \ G_{\rho^s} \right]_x +  \frac{\sigma^2}{2} \ \psi_{xx} + \frac{\sigma^2}{2} \ \rho_{s_{xx}} \nonumber \\
&=  \left[ \psi \ ( w \star \psi + G_{\rho^s} ) \right]_x + \left[ \rho^s \ ( w \star \psi ) \right]_x + \frac{\sigma^2}{2} \ \psi_{xx},
\end{align}
where for the last equality we used the fact that $\rho^s$ is a solution to the stationary equation~\eqref{pdes}, that is, 
$$\left[ \rho^s \ G_{\rho^s} \right]_x +  \frac{\sigma^2}{2} \ \rho^s_{xx} = 0.$$
Multiplying~\eqref{TempS1} by $\psi$ and integrating by part over $\tilde{X}$ we obtain (the extra terms are zero due to periodicity)
\begin{align} \label{TempS2}
\frac{1}{2}\frac{\diff}{\diff t} \Vert \psi \Vert^2_{L^2} + \frac{\sigma^2}{2} \Vert \psi_x \Vert_{L^2}^2 &\leq \left| \int_{\tilde{X}} \psi_x \ \psi \ ( w \star \psi + G_{\rho^s} ) \ \diff x  \right| +  \left| \int_{\tilde{X}} \psi_x \ \rho^s \ ( w \star \psi ) \ \diff x  \right| \nonumber \\
&\leq \left( \Vert w \star \psi \Vert_{L^{\infty}} + \Vert G_{\rho^s} \Vert_{L^{\infty}} \right) \ \Vert \psi_x \Vert_{L^2} \  \Vert \psi \Vert_{L^2}  +  \Vert \rho^s \Vert_{L^{\infty}} \ \Vert \psi_x \Vert_{L^2} \ \Vert w \star \psi \Vert_{L^2},
\end{align}
Now, from inequality~\eqref{G1} in Lemma~\ref{Lem: G}, we have
\begin{align*} 
\Vert w \star \psi \Vert_{L^{\infty}} &\leq 2R \ \Vert \psi \Vert_{L^1(X)} = 2R \ \Vert \rho - \rho^s \Vert_{L^1(X)} \leq 2R \ \left( \Vert \rho \Vert_{L^1(X)} + \Vert \rho^s \Vert_{L^1(X)} \right) = 4R, 
\end{align*}
and
$$ \Vert G_{\rho^s} \Vert_{L^{\infty}} \leq  2R \left( \Vert \rho^s \Vert_{L^1(X)} + M \right) = 2R(1+M).$$
Also, following a similar procedure as in~\eqref{boundG21} and~\eqref{boundG22} with $M=0$, we obtain $ \Vert w \star \psi \Vert_{L^2} \leq 2R^2 \Vert \psi \Vert_{L^2}/\sqrt{3}$. 
Finally, from~\eqref{T1} in Lemma~\ref{Lem: T}, we have $ \Vert \rho^s \Vert_{L^{\infty}} \leq  \exp \left(8R(1+M)/\sigma^2\right)$. 
Using these estimates and the Young's inequality we can rewrite~\eqref{TempS2} as
\begin{align*}
\frac{1}{2}\frac{\diff}{\diff t} \Vert \psi \Vert^2_{L^2} + \frac{\sigma^2}{2} \Vert \psi_x \Vert_{L^2}^2 &\leq \left( 2R(3+M)+ \frac{2R^2}{\sqrt{3}} \ \exp \left( \frac{8R(1+M)}{\sigma^2} \right) \right) \ \Vert \psi_x \Vert_{L^2} \  \Vert \psi \Vert_{L^2} \nonumber \\
&\leq \frac{1}{\sigma^2} \left( 2R(3+M)+ \frac{2R^2}{\sqrt{3}} \ \exp \left( \frac{8R(1+M)}{\sigma^2} \right) \right)^2 \Vert \psi \Vert^2_{L^2} + \frac{\sigma^2}{4} \ \Vert \psi_x \Vert^2_{L^2}.
\end{align*}
Hence,
\begin{align*}
\frac{1}{2}\frac{\diff}{\diff t} \Vert \psi \Vert^2_{L^2} \leq \frac{1}{\sigma^2} \left( 2R(3+M)+ \frac{2R^2}{\sqrt{3}} \ \exp \left( \frac{8R(1+M)}{\sigma^2} \right) \right)^2 \Vert \psi \Vert^2_{L^2} -  \frac{\sigma^2}{4} \ \Vert \psi_x \Vert^2_{L^2}.
\end{align*}
Once again, since  $\int_X \psi (x) \ \diff x = 0$, we can employ the Poincar\'e inequality~\cite{Evans} (Section 5.8.1, Theorem 1) $\Vert \psi \Vert_{L^2} \leq C \ \Vert \psi_x \Vert_{L^2}$ with optimal Poincar\'e constant $C = \frac{1}{\pi}$ to obtain
\begin{align*}
\frac{\diff}{\diff t} \Vert \psi \Vert^2_{L^2} \leq \left\{ \frac{2}{\sigma^2} \left( 2R(3+M)+ \frac{2R^2}{\sqrt{3}} \ \exp \left( \frac{8R(1+M)}{\sigma^2} \right) \right)^2 - \frac{\pi^2 \sigma^2}{2} \right\} \Vert \psi \Vert^2_{L^2}.
\end{align*}
Then, by Gr\"{o}nwall's inequality, we have
\begin{align*}
\Vert \psi(t) \Vert^2_{L^2} \leq \Vert \psi(0) \Vert^2_{L^2} \exp \left[ \left\{ \frac{2}{\sigma^2}  \left( 2R(3+M)+ \frac{2R^2}{\sqrt{3}} \ \exp \left( \frac{8R(1+M)}{\sigma^2} \right) \right)^2 - \frac{\pi^2 \sigma^2}{2} \right\} t \right].
\end{align*}
Now, notice that $ \Vert \psi(0) \Vert_{L^2} \leq \Vert \rho_0 \Vert_{L^2} + \Vert \rho^s \Vert_{L^2}$ is finite. 
Thus, if the constant factor in the exponential is negative, then $\Vert \psi(t) \Vert^2_{L^2} \rightarrow 0$ exponentially fast as $t \rightarrow \infty$. Negativity of the this constant factor corresponds to the condition $\sigma > \sigma_s$, where $\sigma_s$ solves~\eqref{sigma_s}.
\end{proof}

%===============================================================================
\section{Characterization of Solution: Fourier Analysis} 
\label{sec: Fourier}
%===============================================================================

In this section, we exploit the periodic nature of the system and use Fourier analysis to study the behavior of the solution to the PDE~\eqref{pde2} with \emph{uniform initial condition} $\rho_0 = 1$. 
To this end, we derive a system of ordinary differential equations (ODEs) describing the evolution of Fourier coefficients of the normal opinion density~$\rho$. 
Then, these ODEs are used for identification of the so-called \emph{order-disorder transition}. 
In particular, a numerical scheme is presented for approximating the critical noise level at which this transition occurs. 
Moreover, we use these ODEs to provide another approximation scheme for characterizing the initial clustering behavior of the system including the number and the timing of possible clusters. 
These numerical schemes are in essence similar to the linear stability analysis previously employed by~\cite{Pineda09,Pineda11,Pineda13,Garnier17,Wang17} for analysis of noisy bounded confidence models without radicals. 

\subsection{Fourier ODEs for Macroscopic Model}

Notice that the set $\{ \cos\left(\pi nx\right) \}^{\infty}_{n=0}$ is an orthogonal basis for the space $L^2_{ep}(\tilde{X})$ containing even $2$-periodic functions on $\tilde{X}=[-1,1]$. 
Then, the even $2$-periodic extension of the probability densities in the model allows us to consider the Fourier expansions of $\rho$ and $\rho_r$ in the form of
\begin{eqnarray} \label{F}
\rho(x,t) = \sum_{n=0}^{\infty} \ p_n(t) \ \cos\left(\pi nx\right) \ \ \text{and} \ \ \rho_r(x) = \sum_{n=0}^{\infty} \ q_n \ \cos\left(\pi nx\right).
\end{eqnarray}

By inserting the expansions~\eqref{F} into~\eqref{pde2} and setting the inner product of the residual with elements of the basis to zero (in other words, taking inverse Fourier transform), we can obtain a system of quadratic ODEs describing the evolution of Fourier coefficients $p_n(t)$. Considering the first frequency components $n=1,\ldots, N_f$ , these ODEs are expressed as 
\begin{equation} \label{ODE}
\dot{p}_n = c_n + b_n^T p + p^T Q_n p,
\end{equation}
where  $p = (p_1, p_2, \ldots, p_{N_f})^T$. 
Note that for $n=0$, i.e., the constant term in the Fourier expansion, we obtain $\dot{p_0} =\ 0$. 
This is due to the periodic nature of the system that preserves the zeroth moment. 
The coefficients in~\eqref{ODE} are given by

\begin{eqnarray} \label{ODE coeff}
\begin{array}{rcl}
c_n & = &2MR \ f_n \ q_n, \\ \\

(b_n)_k & = &\left\{ \begin{array}{lc}
2R \ f_n+ \frac{MR}{2} \ f_{2n} \ q_{2n} -\frac{\pi^2 \sigma^2 n^2}{2}, & k=n \\
nMR \left\{ \frac{q_{n+k} \ f_{n+k}}{n+k} + \frac{q_{|n-k|} \ f_{n-k}}{n-k} \right\}, & k \neq n,
\end{array}
\right. \\ \\

(Q_n)_{k,l} & = &\left\{ \begin{array}{lc}
nR\frac{f_k}{k}, & l=n-k>1 \\
nR\left\{ \frac{f_k}{k} + \frac{f_{n-k}}{n-k} \right\}, & l=k-n>1\\
0, & \text{otherwise,}
\end{array}
\right.
\end{array}
\end{eqnarray}
where 
\begin{align} \label{fn}
f_n := -\cos \left( \pi n  R\right)+\sinc\left( \pi n  R \right),
\end{align}
with $\sinc x = \frac{\sin x}{x}$. 
Also, recall that $q_n, n \in \N$ are the Fourier coefficients of $\rho_r$.

Interestingly, one notices that the interaction between different frequency components in the quadratic terms is limited to those that are in a sense complements of each other. 
That is, each frequency $n$ of $\rho$ is affected by the frequency pairs $(n_1, n_2)$ such that either $n_1+n_2 = n$ or $|n_1-n_2| = n$. 
This, in turn, leads to a particular structure for the matrix $Q_n$ in the quadratic terms.
As expected, a similar behavior is seen in the linear terms: the effect of each frequency $k$ of $\rho$ on a given frequency $n$ of $\rho$ is modulated by the frequency components $n+k$ and $|n-k|$ of $\rho_r$.

\subsection{Order-disorder Transition} \label{sec: order-disorder approx}

A common behavior in noisy interactive particle systems is the order-disorder transition.
For large values of $\sigma$, the effect of diffusion process can overcome the attracting forces among agents preventing the system from forming any cluster. 
This behavior has been analyzed and observed in several noisy bounded confidence models for opinion dynamics.
Pineda et. al. used linear stability analysis in~\cite{Pineda09, Pineda11} to compute the critical noise level above which the clustering behavior diappears for a modified version of Defuant model~\cite{Deffuant00}. 
The same behavior was also reported in~\cite{Grauwin12} for Defuant model.
The same technique of linear stability analysis was used in~\cite{Wang17,Garnier17} to compute the critical noise level for noisy HK system similar to our model, except without radicals.

Here, we provide a method for approximating the critical noise level $\sigma_c$ at which the transition occurs. 
To this end, we linearize the systems at $t=0$ to obtain a system of linear ODEs expressed as
\begin{equation} \label{ODE linear}
\dot{p} = c + B p.
\end{equation}
The vector $c \in \R^{N_f}$ and the matrix $B \in \R^{N_f\times N_f}$ are defined accordingly using the objects $c_n$ and $b_n$ in~\eqref{ODE coeff}. 
We emphasize that the linearization~\eqref{ODE linear} is for a uniform initial condition, i.e., $p_n(0) = 0$ for $n = 1, \ldots, N_f$. 

Looking at coefficients $c_n$ and $b_n$ in~\eqref{ODE coeff}, we notice that the noise level $\sigma$ only appears in the diagonal entries of $B$ such that by increasing $\sigma$, these diagonal entries decrease. 
That is, for a large enough $\sigma$, all eigenvalues of $B$ are negative and the linearized system~\eqref{ODE linear} is stable. 
This will be our first criterion for determining the critical noise level $\sigma_c$: the noise level above which all eigenvalues of $B$ are negative.

In order to consider the effect of the constant linear growth rates $c$ in~\eqref{ODE linear}, we further require the stationary values $\bar{p}_n, n = 1, \ldots, N_f$ of the linearized system~\eqref{ODE linear} (i.e., the solution to the equation $c + B \bar{p} =0$) to be relatively small. 
In other words, taking the equilibrium of the linearized system $1+ \sum_{n=0}^{N_f} \bar{p}_n \cos\left(\pi nx\right)$ as an approximation of the stationary state $\rho^s$, we require $\rho^s$ to be close to uniform distribution $\rho = 1$ (representing disorder). 
Similar to the theoretical estimate of Theorem~\ref{StatinarySolution}, we quantify this criterion by using Parseval's identity and setting 
\begin{eqnarray}\label{Pareseval}
\Vert \rho^s - 1 \Vert^2_{L^2} \approx \Vert \bar{p} \Vert^2_2 < \gamma,
\end{eqnarray} 
where the constant $\gamma > 0$ determines the level of similarity between $\rho^s$ and uniform distribution. 
To sum up, we solve numerically for the minimum level of noise for which $B$ is Hurwitz and the inequality~\eqref{Pareseval} holds.

\subsection{Initial Clustering Behavior} \label{sec: initial cluster approx}

For noises smaller than the critical noise level $\sigma_c$, we expect to see a clustering bahvior. In order to characterize the initial clustering behavior, we make use of the \emph{exponential growth rate} $\gamma_n \Let (b_n)_n$ and \emph{linear growth rate} $c_n$ given in~\eqref{ODE coeff}. 
The proposed numerical method is as follows. 
We ignore the interactions between different frequencies in~\eqref{ODE}, that is, for each frequency $n =1,\ldots, N_f$, we consider the equation $\dot{p}_n = c_n + \gamma_n p_n$ with $p_n(0) = 0$ (corresponding to uniform initial distribution) for initial evolution of the Fourier coefficient $p_n$. 
Then, for a given set of model parameters $(\sigma,R, M)$ and radical opinions density $\rho_r$, we numerically compute the \emph{dominant wave-number} $n^*:= \argmax_{n\in \N} \gamma_n$ with $\gamma_{n^*} > 0$, that is, the unstable mode with the largest exponential growth rate. 
We speculate that the corresponding trigonometric term $p_{n^*} \cos (\pi n^*  x)$ is the dominant component of the initial clustering behavior.
The sign of $p_{n^*}$ depends on the linear growth rate $c_{n^*}$: $p_{n^*} > 0$ if $c_{n^*} > 0$, and $p_{n^*} < 0$ otherwise. 

Considering the even $2$-periodic extension of the model, the dominant wave-form must be interpreted on the interval $\tilde{X}=[-1,1]$. 
Then, the \emph{number of initial clusters} $n_{\text{clu}}$ in the interval $X=[0,1]$ resulting from the wave-form $1+p_{n^*} \cos (\pi n^*  x)$ is given by
\begin{eqnarray} \label{n_clu}
n_{\text{clu}} := \left\{ \begin{array}{ll}
\lfloor \frac{n^*}{2} \rfloor +1, & c_{n^*} > 0 \\
\lceil \frac{n^*}{2} \rceil, & c_{n^*} < 0.
\end{array}
\right.
\end{eqnarray}
We also expect that the timing of this initial clustering behavior to be inversely related to the corresponding exponential growth rate $\gamma_{n^*}$. 
Indeed, by solving for the time for which the solution to the equation $\dot{p}_n = c_n + \gamma_n p_n$ is equal to $\pm 1$, we can approximate the \emph{time to initial clustering} $t_{\text{clu}}$ as
\begin{equation} \label{t_clu}
t_{\text{clu}} := \frac{1}{\gamma_{n^*}} \ln \left( 1+\frac{\gamma_{n^*}}{|c_{n^*}|} \right).
\end{equation} 
A similar approximation has been used in~\cite{Garnier17} in order to derive the time to the initial clustering using fluctuation theory.

%===============================================================================
\section{Numerical Study} 
\label{sec:NumSim}
%===============================================================================

In this section, we provide a numerical study of the model at hand for a particular distribution of radical agents/opinions through simulations of the corresponding discrete- and continuum-agent models. 
Furthermore, we validate the result of Fourier analysis for identification of order-disorder transition (Section~\ref{sec: order-disorder approx}) and  characterization of initial clustering behavior (Section~\ref{sec: initial cluster approx}). 

The particular radical distribution considered in this section is a triangular distribution with average $A$ and width $2S$
\begin{eqnarray} \label{TriangStub}
\rho_r(x) = 
\left\{
\begin{array}{ll}
\frac{1}{S^2} (S - |x-A|),  &|x-A| \leq S \\
0, &\text{otherwise.}
\end{array}
\right.
\end{eqnarray}
Although this choice may seem specific, it is rich enough for our purposes. 
In particular, with this choice, the zeroth, first and second moments of the radical opinions density are simply captured by the parameters $M$, $A$ and $S$, respectively. Moreover, we assume that the radicals are concentrated around their average opinion, that is, we consider small values of $S$ (w.r.t. the confidence range $R$).  

For the discrete-agent model, the SDEs~\eqref{eq.classical-sde1} are solved numerically using the Euler-Maruyama method for $N=500$ normal agents with time step $\Delta t = 0.01$. 
In particular, for the radical agents, we produce a random sample of size $N_r = MN$ from the triangular distribution~\eqref{TriangStub}. 
The initial distribution of normal agents is taken to be uniform, that is, the initial opinions are randomly sampled from a uniform distribution on the interval $X=[0,1]$.
For complete correspondence between the discrete- and continuum- agent models, we also consider the effect of even $2$-periodic extension in the simulations of the discrete-agent model. See Appendix~\ref{Append: Euler-Maruyama} for the details of the numerical scheme. 
The details of numerical scheme for simulation of the continuum-agent model will be discussed in Section~\ref{sec: sim continuum}.

In the sequel, we make use of the \emph{order parameter}  
\begin{align*}
Q_d(t) &=\frac{1}{N^2} \sum_{i,j=1}^{N} \textbf{1}_{|x_i(t) - x_j(t)| \leq R},
\end{align*}
introduced by~\cite{Wang17} and its continuum counterpart 
\begin{align*}
Q_c(t)&=\int_{X^2} \rho(x,t) \ \rho(y,t) \ \textbf{1}_{|x-y|\leq R} \ \diff x\diff y,
\end{align*}
to quantify orderedness in the clustering behavior of the model. 
In words, the order parameter $Q$ is the (normalized) number/mass of agents that are in $R$-neighborhood of and hence interacting with each other. 
In particular, in the continuum case, $Q_c = 2R$ for a uniform distribution of opinions (absolute disorder), while $Q_c = 1$ for a single-cluster distribution with all agents residing in an interval of width $R$ or less (complete order).
In case of a clustered behavior, roughly speaking, the inverse of the order parameter is equal to the \emph{number of clusters}. 
We also use the evolution of order parameter to characterize the timing of the clustering behavior.

In all the simulation results reported in this section the width of radicals distribution and the confidence range are fixed at $S = 0.1$ and $R = 0.1$, respectively.

\subsection{Simulation of the Continuum-agent Model} \label{sec: sim continuum}

In order to solve the continumm-agent model described by PDF~\eqref{pde2} numerically, we use Fourier ODEs~\eqref{ODE} to compute the coefficients of Fourier expansion of normal opinion density $\rho$ using the first $N_f$ terms of the expansion.
However, regarding the radical opinion density, one notices that the considered triangular distribution does not satisfy the conditions of Theorem~\ref{wellposed} for well-posedness of PDE~\eqref{pde2}, that is, $\rho_r \notin H^2_{ep}(\tilde{X})$. 
This will not be an issue since we will be working with the projection of the proposed $\rho_r$ in the Hilbert space $L^2_{ep}(\tilde{X})$. 
That is, we use the Fourier coefficients of $\rho_r$ in~\eqref{ODE} which for the triangular distribution~\eqref{TriangStub} are given by
\begin{equation} \label{Fst}
q_n=2 \cos \left(n\pi A\right) \sinc^2 (n\pi S/2).
\end{equation}
To be precise, we need the Fourier coefficients $q_n$ of $\rho_r$ for $1 \leq n \leq 2N_f$, that is, twice the length of Fourier expansion of $\rho$; see the linear terms of~\eqref{ODE}. 
For the initial condition, we again consider uniform distribution $\rho_0=1$, which corresponds to $p_0 = 1$ and $p_n (0) = 0$ for the Fourier coefficients.
 
It is also possible to employ a semi-explicit pseudo-spectral method, similar to the one provided by~\cite{Wang17}, for numerically solving~\eqref{pde2}; see Appendix~\ref{Append: Psuedo-Spectral} for details of this method for our model. 
The main difference is that the pseudo-spectral method solves the PDE for a set of discrete points in the space ($x \in X$) while solving the Fourier ODEs gives an approximation of the solution in terms of a finite basis for the corresponding Hilbert space.

These two methods (if both converge) result in the same solution. 
Fig.~\ref{PSvsFC} compares the result of numerical simulations of the model using these two methods for a particular combination of system data. 
Note that, in these simulations, the number of points for the spacial discretization in pseudo-spectral method is twice the $N_f$ for Fourier ODEs so that the methods are compatible, i.e., both include the same set of frequency components.
The left panel of Fig.~\ref{PSvsFC} shows a similar result using these two methods for $N_f = 32$ frequencies. However, as the number of frequencies considered in the simulations are decreased, we see that that the pseudo-spectral method starts to diverge while Fourier ODEs are still stable. 

In the remainder of this section, we use the Fourier ODEs~\eqref{ODE} with $N_f = 128$ for numerical simulation of the continuum-agent model since they are computationally more efficient.
\begin{figure}
\centering
\includegraphics[clip, trim=2cm 0cm 2cm 0cm,width=1\linewidth]{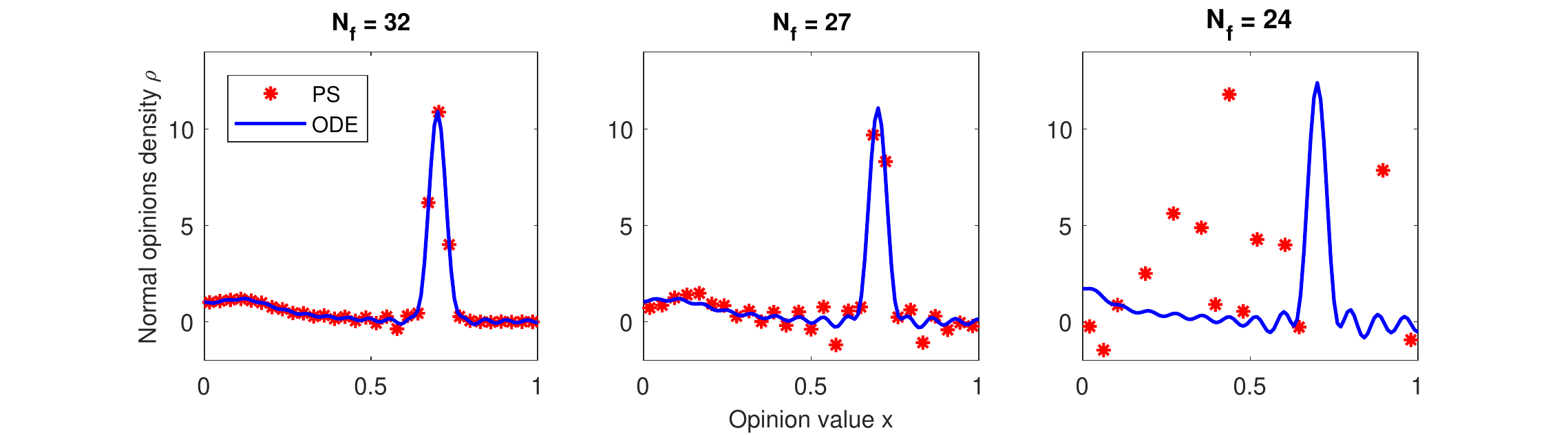} 
\caption{Comparison of the pseudo-spectral method (\emph{PS}) described in Appendix~\ref{Append: Psuedo-Spectral} with $\Delta t = 0.01$ and the Fourier ODEs (\emph{ODE}) given in~\eqref{ODE} for numerical simulation of the continuum-agent model~\eqref{pde2}. The results are for $t = 400$ with system data $(\sigma, M, A) = (0.03, 0.1, 0.7)$. In the right panel, some of the points in the solution of the pseudo-spectral method are outside the limits of the vertical axis.}
\label{PSvsFC}
\end{figure}

\subsection{Order-disorder Transition} \label{sec: sim transition}

In this section, we numerically study the order-disorder transition in the model. 
In particular, we consider the effect of the relative mass of radicals $M$ on the critical noise level $\sigma_c$ at which this transition occurs. 
Furthermore, we use our simulation results to examine the approximation scheme presented in Section~\ref{sec: order-disorder approx}.
In this regard, we note that the interplay between the confidence range $R$ and the critical noise level $\sigma_c$ have been studied in~\cite{Wang17}.
There, the authors showed that as $R$ increases, the critical noise level $\sigma_c$ also increases in such way that for small values of $R$, we observe a first-order transition.

\subsubsection{An Illustrative Example}
Our model exhibits the same order-disorder transition previously reported for similar noisy HK systems~\cite{Pineda13, Garnier17, Wang17}. 
Fig.~\ref{noise_level} shows this effect for a particular combination of system data in the discrete- and continuum-agent models.
Notice that for $\sigma$ larger than a critical level the clustering behavior almost disappears (see the lower panel corresponding to $\sigma = 0.05$ in Fig.~\ref{noise_distr}).
To be more precise, a higher level of noise decreases the life-time of clustering behaviors with larger number of clusters.

This effect can be particularly seen in the evolution of order parameter in Fig.~\ref{noise_ordpar}.
In this regard, notice that for noises smaller than the critical noise level (here $\sigma < 0.05$) the flat areas in the order parameter in Fig.~\ref{noise_ordpar} correspond to a clustered behavior, where the number of clusters is equal to the inverse of the order parameter. 
To illustrate, observe that for $\sigma = 0.03$ and $\sigma = 0.04$, the system reaches a single-cluster profile around the average radical opinion $A = 0.7$.
Notice, however, for $\sigma = 0.03$ the system first goes through a 2-cluster profile corresponding to the flat area in the blue solid line at height $0.5$ in Fig.~\ref{noise_ordpar}.
On the other hand, for $\sigma = 0.02$, we observe a 2-cluster profile at $t=10^4$ in Fig.~\ref{noise_distr}. Notice, however, how the system goes through 4-cluster and 3-cluster profiles as depicted in Fig.~\ref{noise_ordpar} (the flat areas in the order parameter). 
Finally, for $\sigma = 0.01$, we observe a very fast emergence of a 4-cluster profile (Fig.~\ref{noise_ordpar}) that has survived until $t=10^4$ as shown in Fig.~\ref{noise_distr}.
Here, we also notice that exact position of clusters in the discrete- and continuum-agent models differ. 
This particular difference between mean-field and agent-based models has been also mentioned by~\cite{Pineda09, Pineda11}. 
Indeed, our numerical simulations show that even the number of clusters resulting from mean-field and agent-based models may differ; this also has been reported and explained previously in~\cite{Wang17}. 
Finally, we note that for $M=0.1$, the approximation scheme explained in Section~\ref{sec: order-disorder approx} results in $\sigma_c = 0.043$ for $\gamma = 1$ and $\sigma_c = 0.051$ for $\gamma = 0.1$ (see~\eqref{Pareseval} for influence of $\gamma$). 

\begin{figure}
\begin{subfigure}{.5\textwidth}
  \centering
  \includegraphics[clip, trim=0cm 0.7cm 0cm 1cm,width=.9\linewidth]{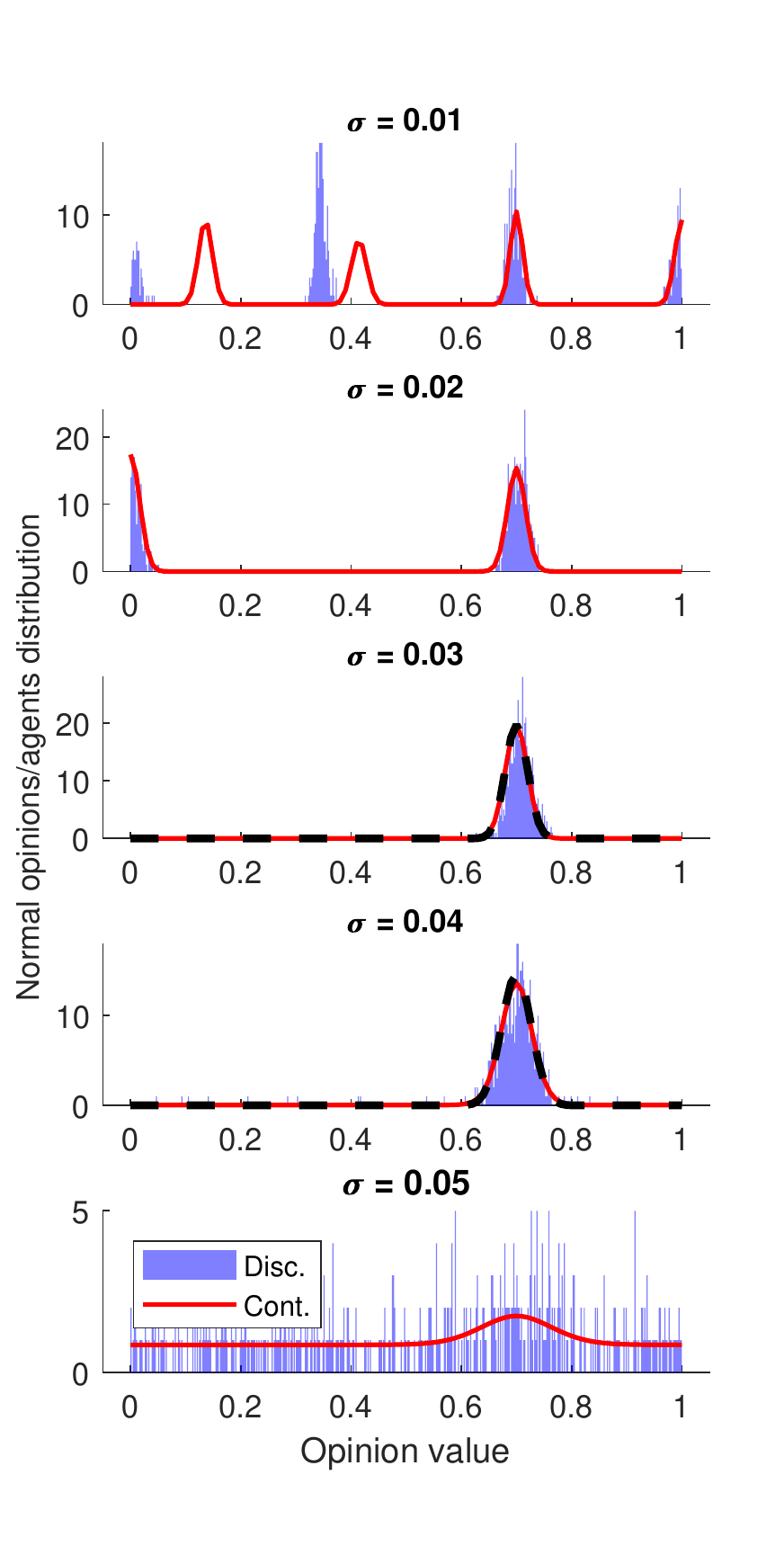}
  \caption{Distribution of opinions/agents at $t = 10^4$}
  \label{noise_distr}
\end{subfigure}%
\begin{subfigure}{.5\textwidth}
  \centering
  \includegraphics[clip, trim=0cm 0.7cm 0cm 1cm,width=.9\linewidth]{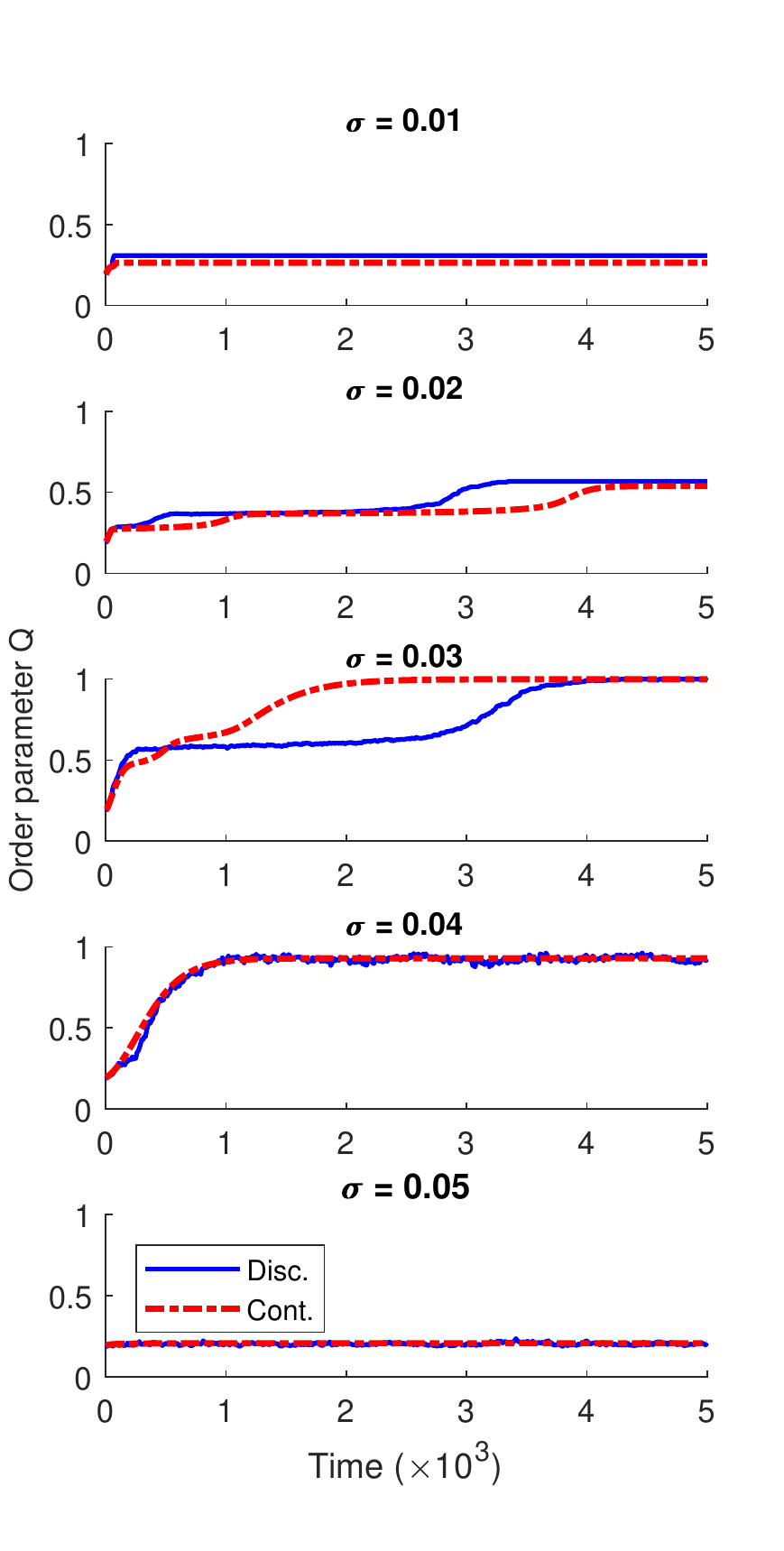}
  \caption{Evolution of order parameter}
  \label{noise_ordpar}
\end{subfigure}
\caption{Numerical simulation of the discrete-agent model (\emph{Disc.}) and continuum-agent model (\emph{Cont.}) for different values of noise $\sigma$ with system data $(M, A) = (0.1, 0.7)$. As noise increases the number of clusters decreases so that for a large enough noise the clustering behavior disappears (see Section~\ref{sec: sim transition}). The black dashed lines in left panels for $\sigma = 0.03, 0.04$ are the approximate stationary solutions~\eqref{StaSol4} (see Remark~\ref{Rem: approx}).}
\label{noise_level}
\end{figure}

\subsubsection{Effect of $M$ on $\sigma_c$}
Fig.~\ref{mass_noise} shows the order parameter derived numerically by simulating  the continuum- and discrete-agent models. 
Notice how for each $M$, as noise increases, the system experiences a transition form order (clustered phase with $Q \approx 1$ in the yellow strip) to disorder (with $Q \approx 0.2$ in the dark blue area in the upper part of the plots). 
Also, we note that the blue strip in the lower part of plots in Fig.~\ref{mass_noise} represents clustering behaviors with larger number of clusters (similar to the behavior seen for $\sigma = 0.01$ in Fig.~\ref{noise_level}).  

This result shows that as the relative mass of radicals $M$ increases, the corresponding critical noise level $\sigma_c$, above which the system is in disordered state, also increases. 
The dependence of $\sigma_c$ on $M$ is in the form of a concave function. 
Furthermore, for small values of $M$, the transition seems to be discrete, signaling a first-order transition. However, for large values of $M$ the transition becomes blurry. 
This phenomenon was also reported in~\cite{Wang17} for the dependence of the critical noise level on the confidence range $R$. 
Notice that as $M$ increases, the required noise level for disordered behavior also increases. 
This increase in the noise level leads to wider clusters, which, in turn, makes it difficult to differentiate order from disorder; see, e.g., the panels corresponding to $\sigma = 0.04, 0.05$ in Fig.~\ref{noise_level}.      

Also shown in Fig.~\ref{mass_noise} (red lines) is the result of scheme provided in Section~\ref{sec: order-disorder approx} for approximating the critical noise level. As can be seen, the scheme indeed provides a good approximation of the critical noise level. 
In particular, the dashed red line (for $\gamma = 1$) almost perfectly separates the two phases of order and disorder.  

\begin{figure}
\begin{subfigure}{.5\textwidth}
  \centering
  \includegraphics[clip, trim=.1cm .1cm .9cm .2cm,width=0.9\linewidth]{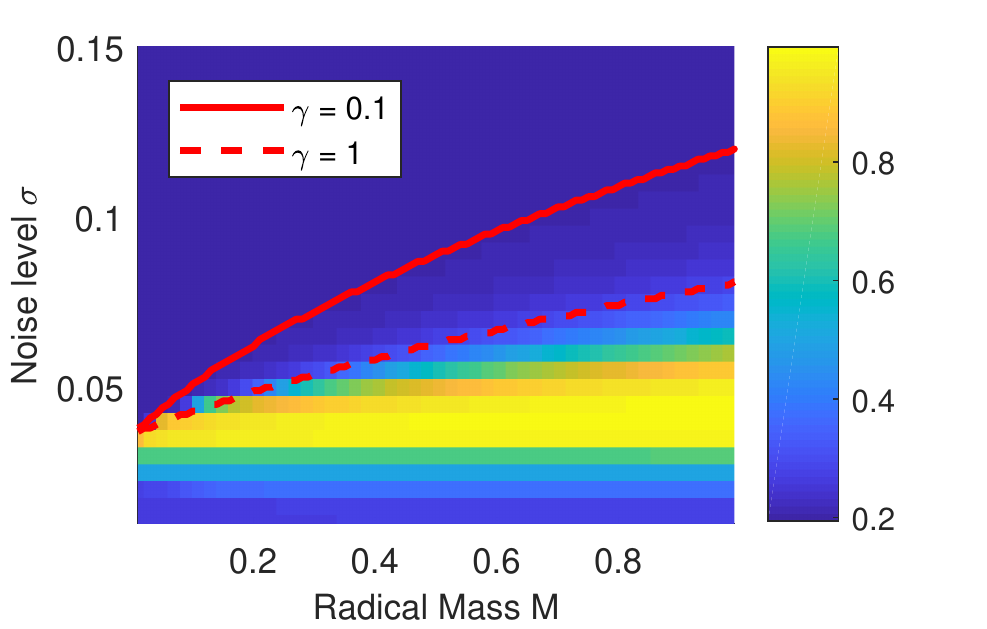}
  \caption{Continuum-agent model}
  \label{mass_noise_PDE}
\end{subfigure}%
\begin{subfigure}{.5\textwidth}
  \centering
  \includegraphics[clip, trim=.1cm .1cm .9cm .2cm,width=0.9\linewidth]{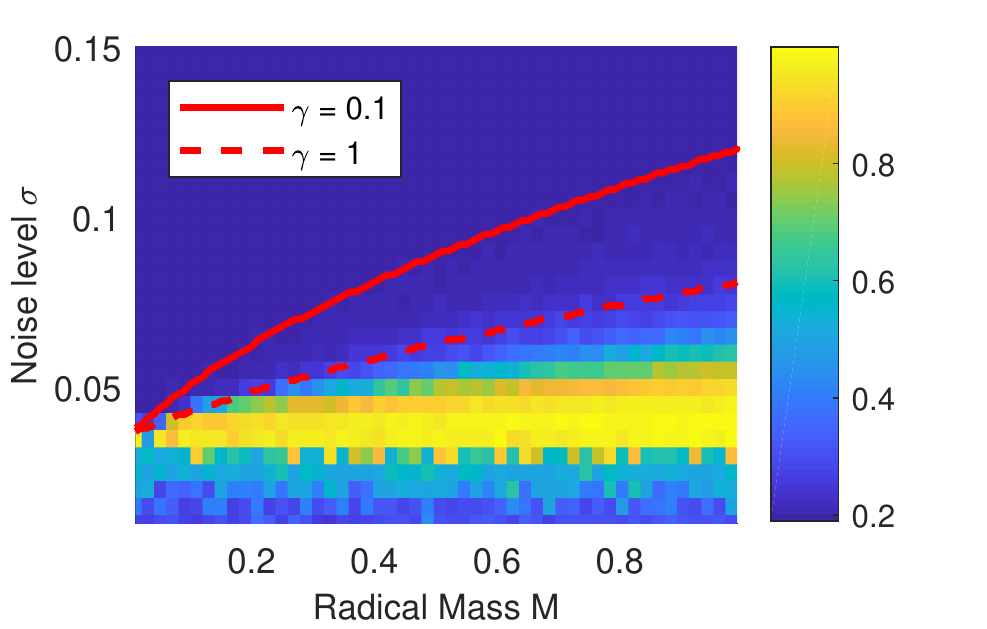}
  \caption{Discrete-agent model}
  \label{mass_noise_SDE}
\end{subfigure}
\caption{The order parameter at $t = 10^3$ from numerical simulation of the continuum- and discrete-agent models starting form uniform initial distribution. For the discrete-agent model, the average of order parameter over  the time  window $[900, 1000]$ is reported. The plot covers the region $\sigma \times M \in [0.01, 0.15] \times [0.01, 1]$ with step sizes $\Delta \sigma = 0.005$ and $\Delta M = 0.02$. The red lines show the result of the numerical scheme described in Section~\ref{sec: order-disorder approx} for approximating the critical noise level for different values of $\gamma$ w.r.t. the second criterion~\eqref{Pareseval}. See Section~\ref{sec: sim transition} for details.}
\label{mass_noise}
\end{figure} 

\subsection{Initial Clustering Behavior} \label{sec: sim initial}

For noises smaller than the critical noise level, agents start to form clusters; see Fig.~\ref{noise_level}. 
In particular, we observe a cluster of normal agents around the average radical opinion $A$ due to the force field generated by the radicals.
Generally, three types of clusters may form: (1) the cluster at the average radical opinion $A$, (2) the cluster(s) at the extreme opinions $x = 0$ and/or $x = 1$, and (3) the cluster(s) around opinion values other than $x = 0, 1, A$.  
The third type of clusters are expected to perform a random walk with their center of mass moving like a Brownian motion (assuming clusters do not interact). The effective diffusivity of these Brownian motions is inversely related to the size of the cluster, i.e., the number of agents in the cluster.
This will result in a process of consecutive merging between these clusters until complete disappearance of them. 
Detailed descriptions of this process are provided in~\cite{Garnier17,Wang17}.
Notice however that this description does not apply to cluster(s) formed at $x = A$ and $x = 0,1$. 
These clusters are affected by forces other than the normal attractions among the agents within the cluster.
The cluster formed at $x = A$ is under influence of radicals and the possible clusters at the extreme opinions $x = 0,1$ are reinforced due to the even $2$-periodic extension considered in our model.
The behavior of these clusters (survival or dissolution) depends on their size, the exogenous force acting on them, and the effect of other clusters in their neighborhood.

In this section, we use the analysis scheme provided in Section~\ref{sec: initial cluster approx} to investigate the effect of the zeroth and first moment of radicals ($M$ and $A$, respectively) on the initial clustering behavior of the model for noises smaller than the critical level. 
In particular, we investigate the effect of $M$ and $A$ on the number, position and timing of initial clusters for different values of $\sigma$.
We again emphasize that we are considering a concentrated triangular distribution for radical agents and a uniform initial distribution for normal agents. 
Let us begin with illustrating how the objects introduced in Section~\ref{sec: initial cluster approx}, namely, exponential and linear growth rates and the dominant wave-number, can be used to characterize the initial clustering behavior. 

\subsubsection{An Illustrative Example} 
Consider the system data $(\sigma, M, A) = (0.01, 0.1, 0.7)$. 
Fig.~\ref{initial_case} depicts the values of the exponential growth rate $\gamma_n$ and the linear growth rate $c_n$ for different frequencies. 
In Fig.~\ref{initial_case_Exp_terms} we observe that the unstable mode with the maximum exponential growth rate is $n^* = 8$ with $\gamma_{n^*} = 0.177$. 
Fig.~\ref{initial_case_Lin_terms} shows that the linear coefficient corresponding to this frequency is $c_{n^*} = 0.007 > 0$.  
Then,~\eqref{n_clu} implies that the initial clustering behavior is expected to have $n_{\text{clu}} = 5$  clusters. 
Also, using~\eqref{t_clu}, we obtain $t_{\text{clu}} = 18.16$ for the time to initial clustering. 

\begin{figure}
\begin{subfigure}{.5\textwidth}
  \centering
  \includegraphics[width=1\linewidth]{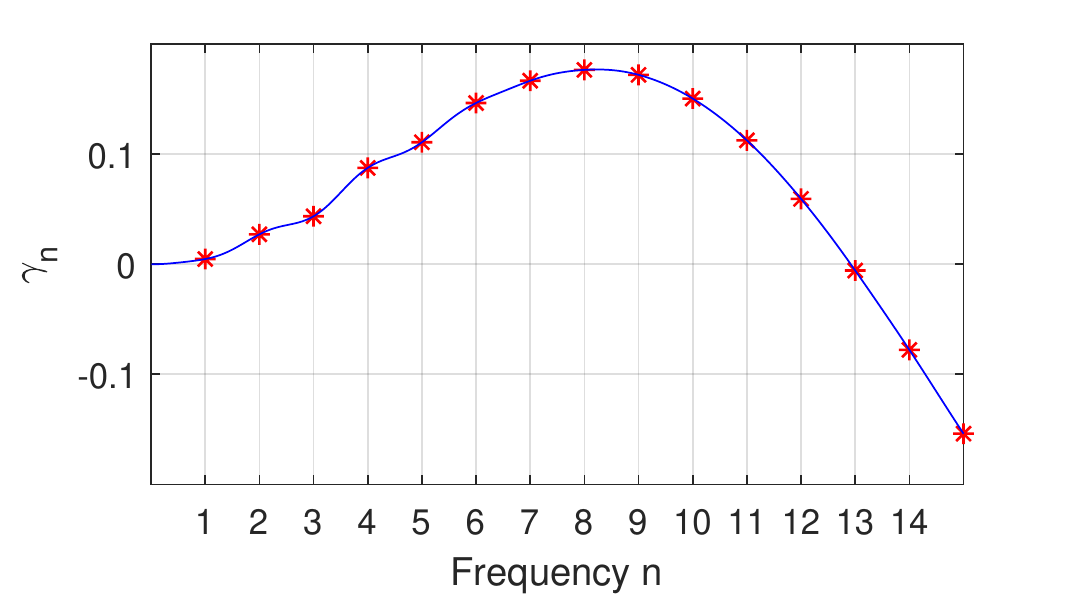}
  \caption{Exponential growth rate}
  \label{initial_case_Exp_terms}
\end{subfigure}%
\begin{subfigure}{.5\textwidth}
  \centering
  \includegraphics[width=1\linewidth]{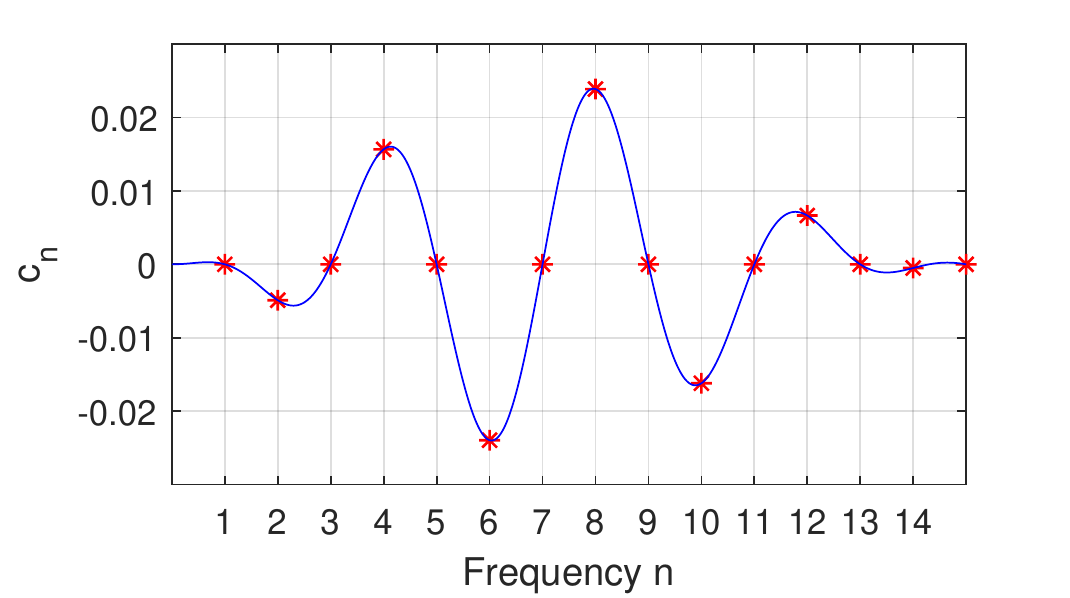}
  \caption{Linear growth rate}
  \label{initial_case_Lin_terms}
\end{subfigure}
\caption{Exponential and linear growth rates for system data $(\sigma, M, A) = (0.01, 0.1, 0.7)$ for different frequencies. On the left panel we see the maximum exponential growth corresponds to $n^* = 8$ with $\gamma_{n^*} = 0.177$. On the right panel we see $c_{n^*} = 0.007 > 0$. This implies that the waveform $p_8 \cos (8 \pi x)$ with $p_8 > 0$ is the dominant component of the initial clustering behavior.}
\label{initial_case}
\end{figure}

Fig.~\ref{initial_case_1} shows the time evolution of distribution of normal opinions/agents for the system data corresponding to Fig.~\ref{initial_case}. 
For the continuum-agent model, we can see a 5-cluster profile corresponding to the speculated waveform as depicted in Fig.~\ref{case_PDE_dist_1}. 
A similar clustering behavior is observed in the Monte Carlo simulation of the discrete-agent model in Fig.~\ref{case_SDE_dist_1}. 
Here, we observe three clear clusters: the cluster at average radical opinion $A = 0.7$ and the two clusters at extreme opinions $x = 0,1$. 
However, we observe an almost uniform distribution of normal agents in the opinion range $[0.1, 0.5]$. 
This is due to the fact that the exact position of the corresponding clusters formed in the discrete-agent model varies within this range. 
Individual realizations of the discrete model show one, two or three clusters in this range with two clusters being the most frequent behavior as expected.
This effect has been also reported by~\cite{Pineda09} in Monte Carlo simulations of a noisy Defuant model.
Furthermore, we notice that the timing object $t^* = 18.16$ also gives a good approximation for the onset of the corresponding clustering behavior for both continuum- and discrete-agent systems.  

\begin{figure}
\begin{subfigure}{.5\textwidth}
  \centering
  \includegraphics[clip, trim=0cm 0.6cm 0cm 1cm,width=0.9\linewidth]{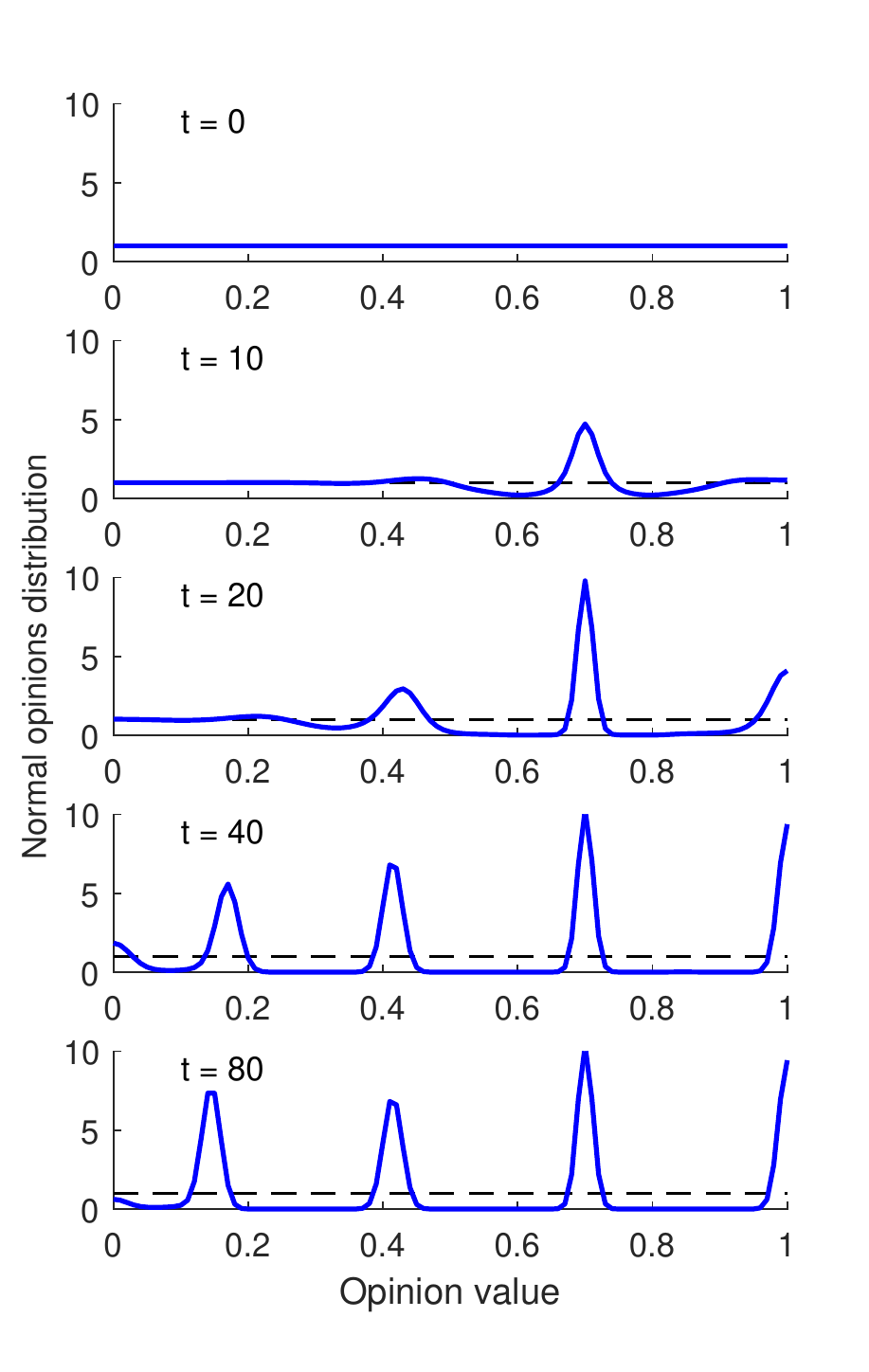}
  \caption{Continuum-agent model}
  \label{case_PDE_dist_1}
\end{subfigure}%
\begin{subfigure}{.5\textwidth}
  \centering
  \includegraphics[clip, trim=0cm 0.6cm 0cm 1cm,width=0.9\linewidth]{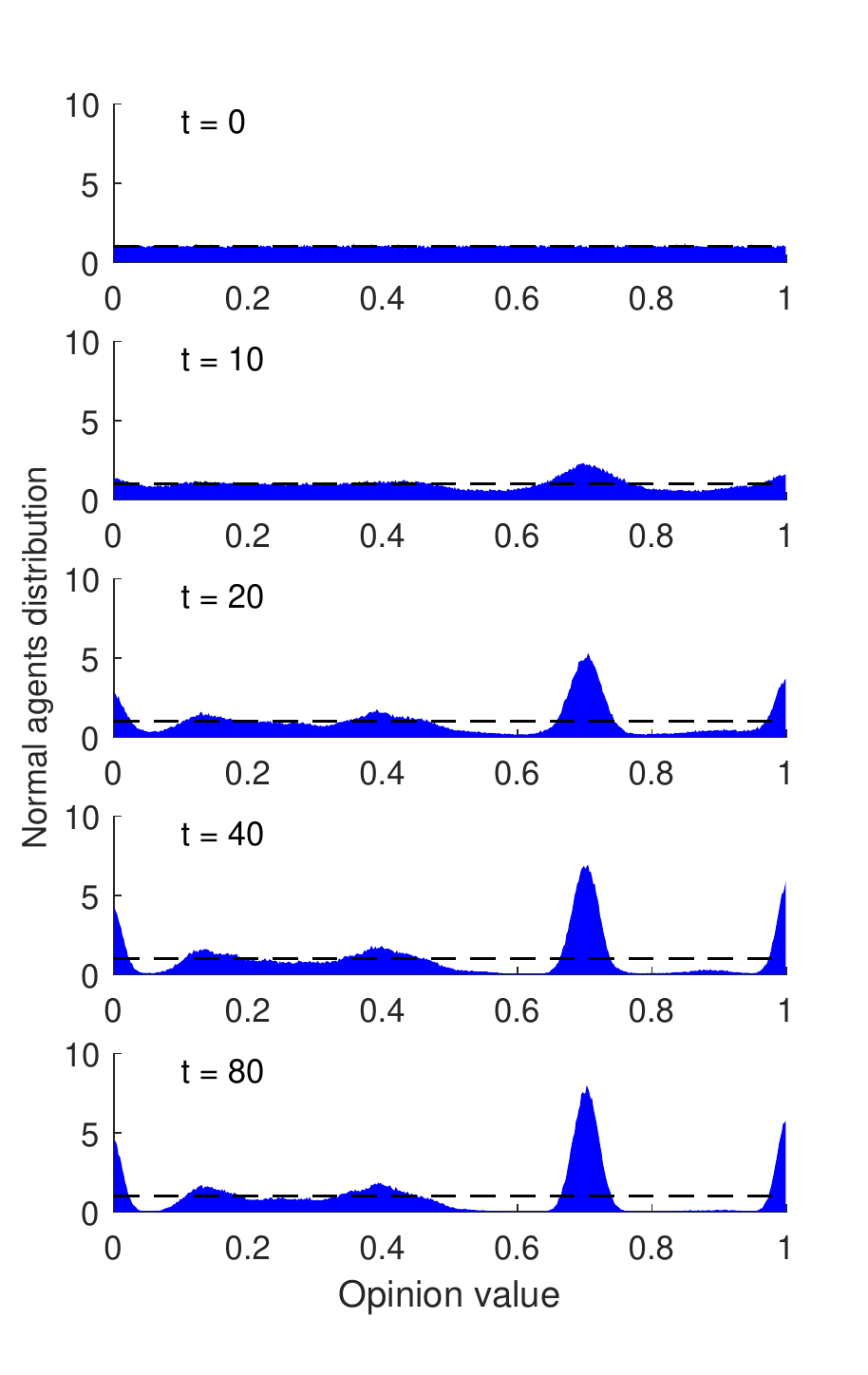}
  \caption{Discrete-agent model}
  \label{case_SDE_dist_1}
\end{subfigure}
\caption{Evolution of distribution of normal opinions/agents during the initial clustering behavior for system data $(\sigma, M, A) = (0.01, 0.1, 0.7)$ corresponding to Fig.~\ref{initial_case}. The distributions shown for the discrete-agent model are the average profiles of 300 realizations. The onset of a 5-cluster behavior is observed from approximately $t = 20$ corresponding to the waveform $1+\cos (8 \pi x)$ speculated for the initial clustering behavior with $t_{\text{clu}} = 18.16$.}
\label{initial_case_1}
\end{figure}

\subsubsection{Effect of $M$ and $A$ on Initial Clustering} 
Performing a similar analysis to the one provided in the example above, we can compute the dominant wave-number ($n^*$), number of initial clusters ($n_{\text{clu}}$) and time to initial clustering ($t_{\text{clu}}$) for a general combination of system data.
Fig.~\ref{initial_clustering} shows the result of this analysis for different values of $M$ and $A$ at three different noise levels $\sigma$.
Here, we only considered the values $A < 1-R = 0.9$ since for $1-R < A < 1$ the boundary effect due to even $2$-periodic extension comes into play.

\begin{figure}
\begin{subfigure}{1\textwidth}
  \centering
  \includegraphics[clip, trim=2cm 0cm 1cm 0cm,width=1\linewidth]{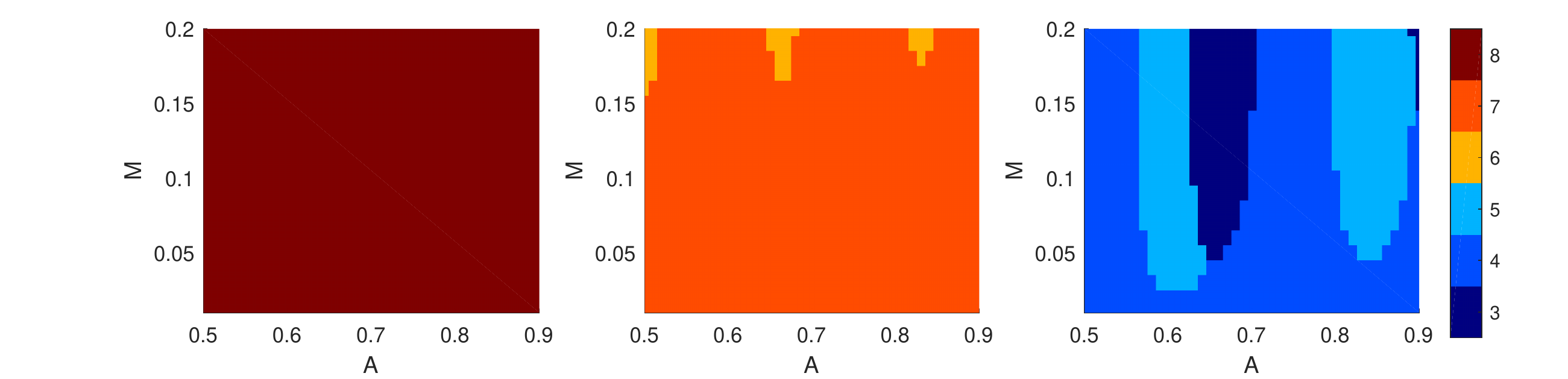}
  \caption{Dominant wave-number: $n^*$}
  \label{initial_n_star}
\end{subfigure}
\begin{subfigure}{1\textwidth}
  \centering
  \includegraphics[clip, trim=2cm 0cm 1cm 0cm,width=1\linewidth]{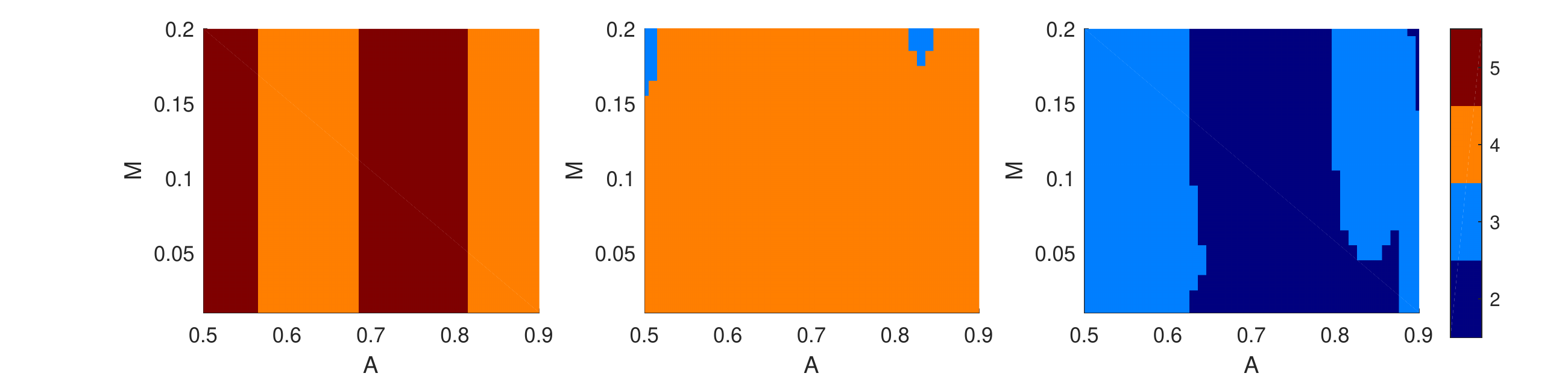}
  \caption{Number of initial clusters: $n_{\text{clu}}$}
  \label{initial_n_cluster}
\end{subfigure}
\begin{subfigure}{1\textwidth}
  \centering
  \includegraphics[clip, trim=2cm 0cm 1cm 0cm,width=1\linewidth]{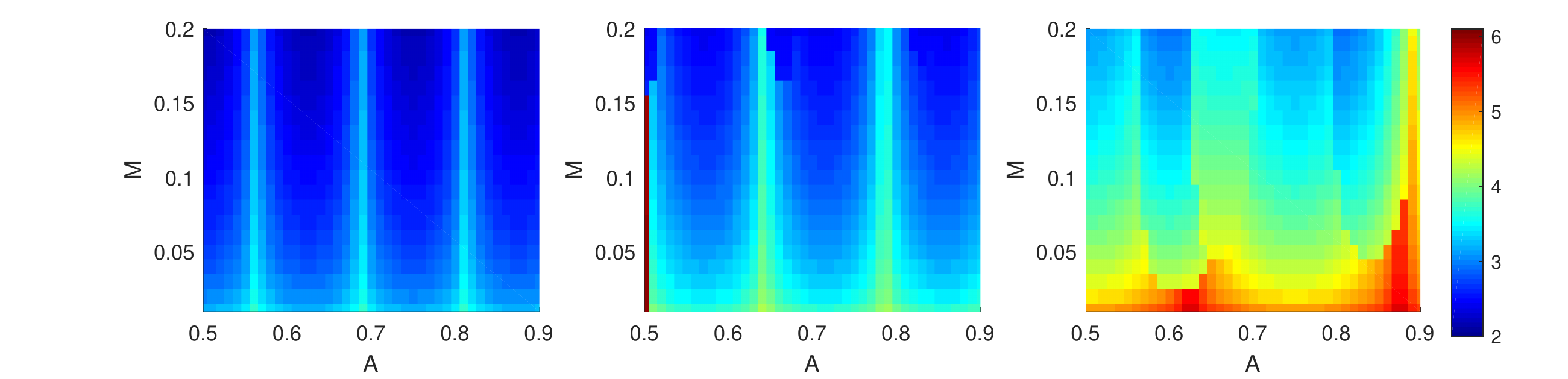}
  \caption{Time to initial clustering: $\ln (t_{\text{clu}})$}
  \label{initial_t_cluster}
\end{subfigure}
\caption{Characterization of the initial clustering behavior based on the dominant wave-number in the Fourier expansion of the continuum-agent model for different values of $M$ and $A$ with noise levels $\sigma = 0.01$ (\textit{left}), $\sigma = 0.02$ (\textit{middle}), and $\sigma = 0.03$ (\textit{right}).}
\label{initial_clustering}
\end{figure}

Comparing the left, middle and right panels of Fig.~\ref{initial_clustering} corresponding to different levels of noise, we observe that as the level of noise increases, the number of clusters in the possible clustering behavior of the system decreases (see Fig.~\ref{initial_n_cluster}), while the timing experiences a general increase (see Fig.~\ref{initial_t_cluster}). 
This effect has been already shown in Fig.~\ref{noise_level}. 
In particular, with respect to the timing, we notice that as the level of noise decreases, the \emph{initial} clustered profile emerges faster; see Fig.~\ref{noise_ordpar}.

For low levels of noise, e.g., $\sigma = 0.01$ (see the left panels in Fig.~\ref{initial_clustering}), the dominant wave-number does not depend on the $M$ or $A$. 
In this case, the most important effect of the first moment of radical opinions density $A$ is on the position of clusters. 
That is, the clustered profile emerges in a way that we observe a particular cluster formed at the average radical opinion $A$. 
The parameter $A$ also affects the timing of the clustering behavior in a periodic fashion. 
On the other hand, the zeroth moment of radical opinions density $M$ only affects the timing of the clustering behavior: as $M$ increases, $t_{\text{clu}}$ decreases.
Fig.~\ref{initial_comp1} shows the simulation results for $\sigma = 0.01$ and compares the evolution of opinions for different values of $M$ and $A$. 
For the continuum model in the the top panels of Fig.~\ref{initial_comp1} we observe that indeed a 4-cluster profile has emerged in all systems. 
Comparing Figs.~\ref{initial_compare_noise1_PDE1} and~\ref{initial_compare_noise1_PDE2} shows that $M$ only affects the timing of clustering behavior. This effect is better seen in Fig.~\ref{initial_comp1_ordpar_cont} where we observe a faster convergence of order parameter for $\mathcal{S}_2$ with larger $M$. On the other hand, comparing Figs.~\ref{initial_compare_noise1_PDE2} and~\ref{initial_compare_noise1_PDE3} corresponding to $A=0.85$ and $A=0.7$, respectively, we observe a change in the positioning of the clusters.
Monte Carlo simulations of the discrete-agent model reveals that the same general description also holds for this system. 
This is particularly seen in the time evolution of the order parameter in the discrete-agent model as depicted in Fig.~\ref{initial_comp1_ordpar_disc}. 
However, we once again note that there are differences between the behavior of the continuum- and discrete-agent models. 
In particular, the evolution of order parameter in Fig.~\ref{initial_comp1_ordpar_cont} shows that the continuum-agent model has seemingly converged to steady-state with four clusters, while this is clearly not the case for the discrete-agent model as can be seen in Fig.~\ref{initial_comp1_ordpar_disc}. 
Indeed, in the discrete-agent model, as described in the beginning of this section, all the possible clusters formed around opinion values other than $x = 0, 1, A$ will necessarily disappear in the steady state profile, where the time required for their disappearance depends on the noise level and particularly the size of these clusters. 
Hence, unlike the discrete-agent model, for the continuum-agent model (in the limit $N \rightarrow \infty$), the system may require \emph{infinite time} for this merging of the clusters to occure. 
This, in turn, can lead to differet behaviors in the discrete- and continuum-agent models over exponentially large times scales~\cite{Wang17}; see also the evolution of order parameter in Fig.~\ref{noise_ordpar}. 

\begin{figure}

\begin{subfigure}{.33\textwidth}
  \centering
  \includegraphics[width=1\linewidth]{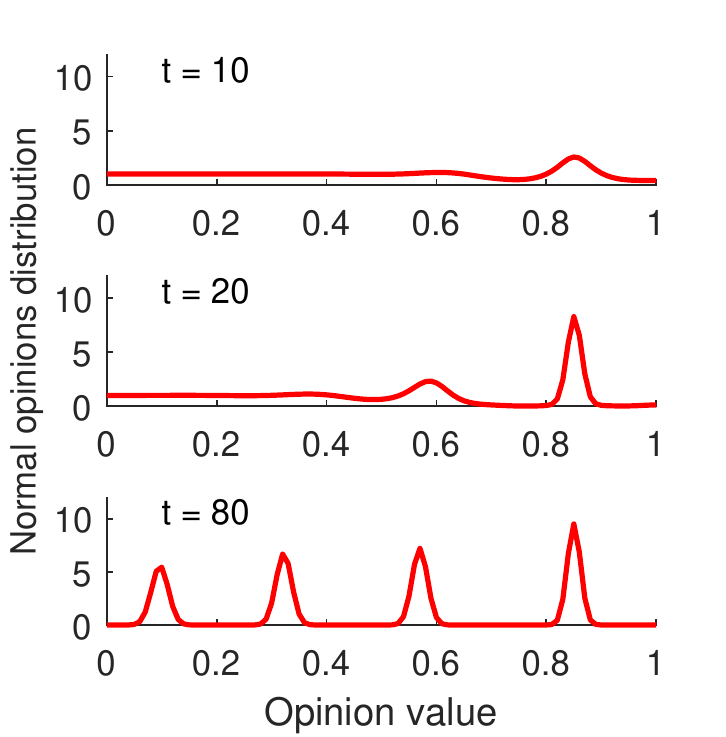}
  \caption{Continuum - $\mathcal{S}_1:(0.05,0.85)$}
  \label{initial_compare_noise1_PDE1}
\end{subfigure}%
\begin{subfigure}{.33\textwidth}
  \centering
  \includegraphics[width=1\linewidth]{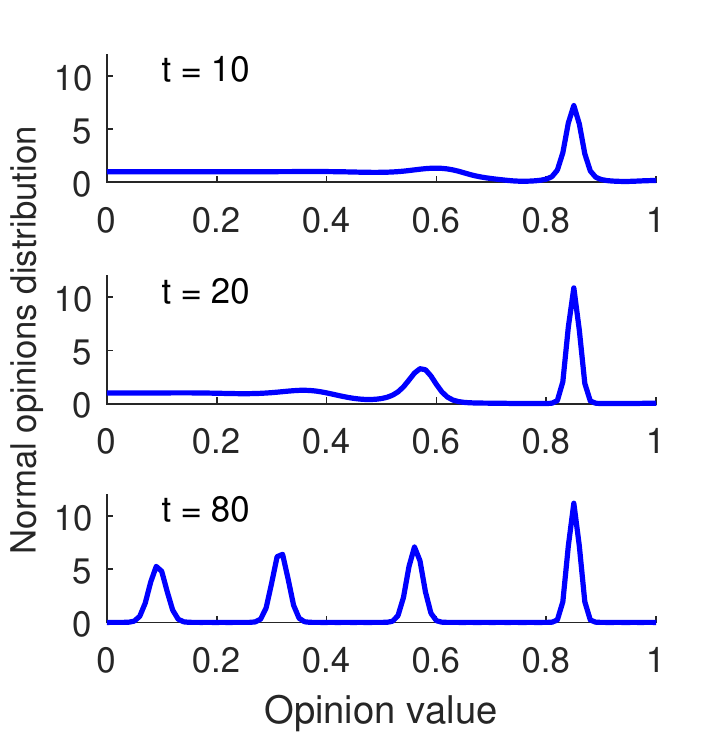}
  \caption{Continuum - $\mathcal{S}_2:(0.15,0.85)$}
  \label{initial_compare_noise1_PDE2}
\end{subfigure}%
\begin{subfigure}{.33\textwidth}
  \centering
  \includegraphics[width=1\linewidth]{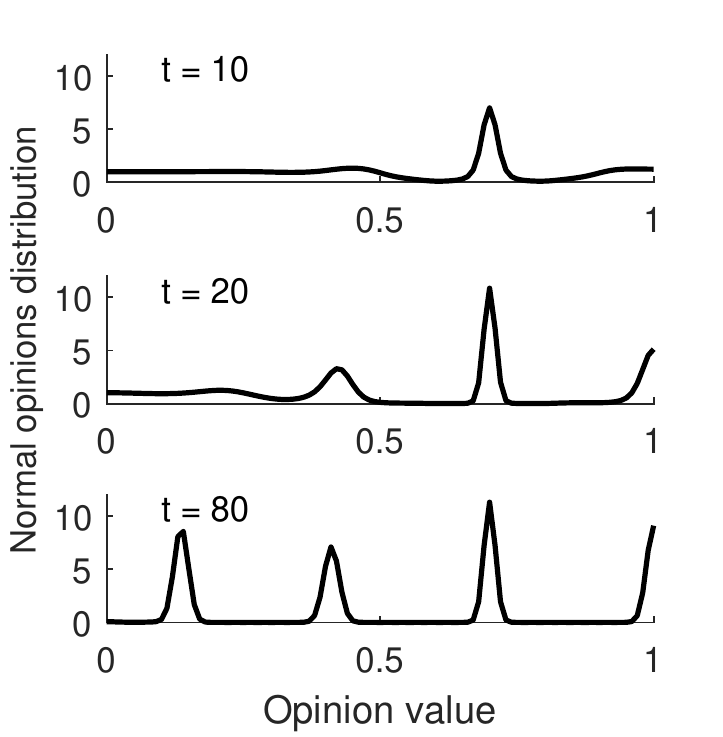}
  \caption{Continuum - $\mathcal{S}_3:(0.15,0.7)$}
  \label{initial_compare_noise1_PDE3}
\end{subfigure}

\begin{subfigure}{.33\textwidth}
  \centering
  \includegraphics[width=1\linewidth]{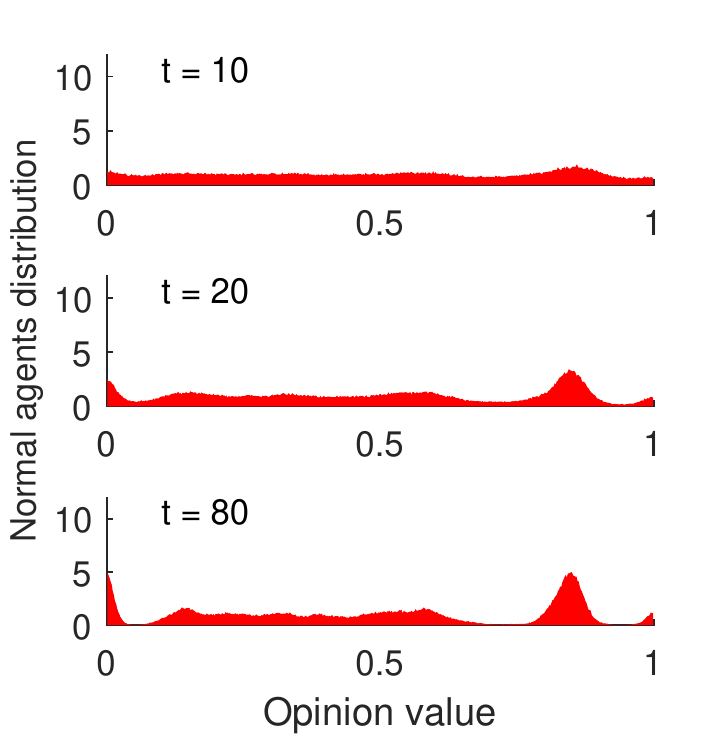}
  \caption{Discrete - $\mathcal{S}_1:(0.05,0.85)$}
  \label{initial_compare_noise1_SDE1}
\end{subfigure}%
\begin{subfigure}{.33\textwidth}
  \centering
  \includegraphics[width=1\linewidth]{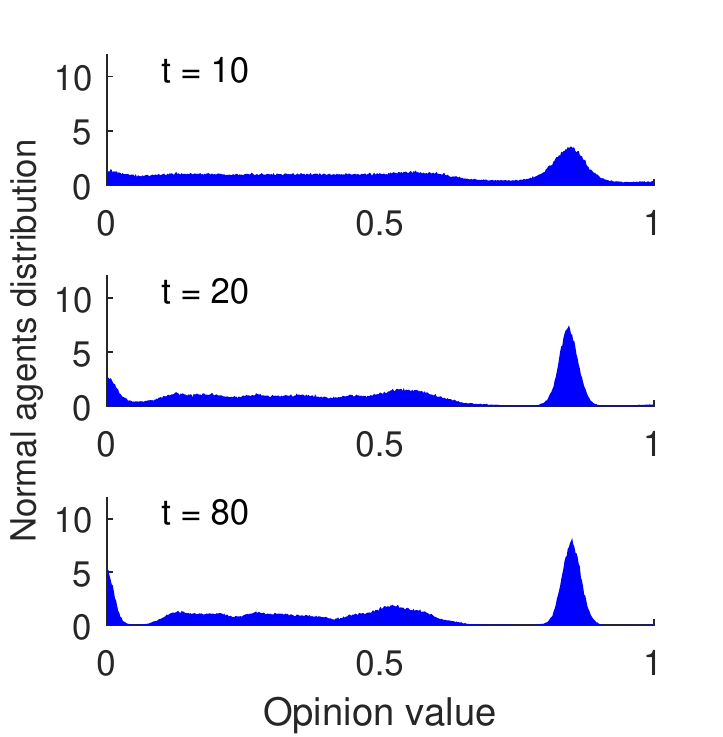}
  \caption{Discrete - $\mathcal{S}_2:(0.15,0.85)$}
  \label{initial_compare_noise1_SDE2}
\end{subfigure}%
\begin{subfigure}{.33\textwidth}
  \centering
  \includegraphics[width=1\linewidth]{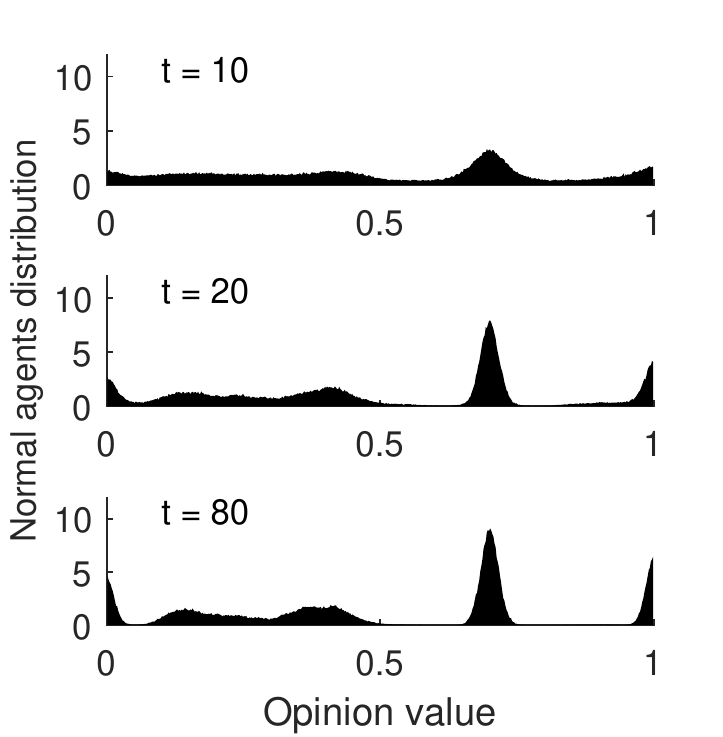}
  \caption{Discrete - $\mathcal{S}_3:(0.15,0.7)$}
  \label{initial_compare_noise1_SDE3}
\end{subfigure}

\begin{subfigure}{.5\textwidth}
  \centering
  \includegraphics[width=.9\linewidth]{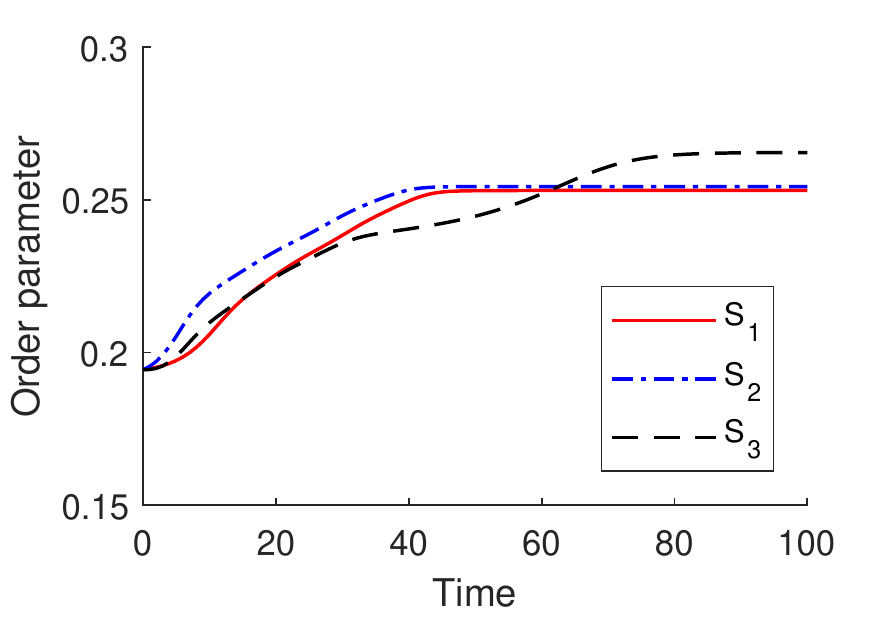}
  \caption{Continuum-agent model}
  \label{initial_comp1_ordpar_cont}
\end{subfigure}%
\begin{subfigure}{.5\textwidth}
  \centering
  \includegraphics[width=.9\linewidth]{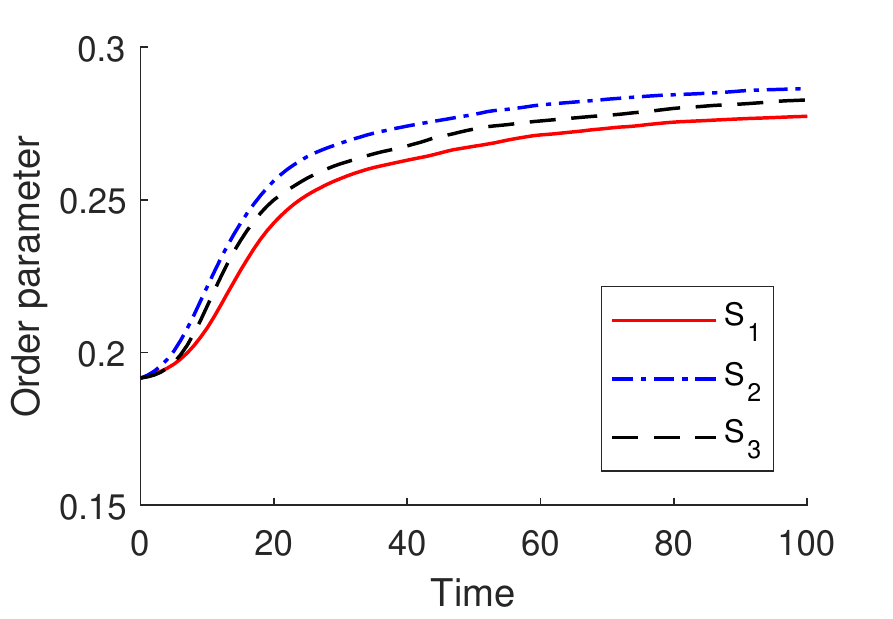}
  \caption{Discrete-agent model}
  \label{initial_comp1_ordpar_disc}
\end{subfigure}

\caption{Numerical simulation of the model with $\sigma=0.01$ for different values of $(M,A)$, namely, $\mathcal{S}_1:(0.05,0.85)$,  $\mathcal{S}_2:(0.15,0.85)$, and  $\mathcal{S}_3:(0.15,0.7)$. The upper panels (A, B, and C) show the opinion distribution for continuum-agent model. The middle panels (D, E, and F) show the the result of Monte Carlo simulation (average of 300 realizations) of discrete-agent model. The lower panels (G and H) show the evolution of order parameter for these systems.}
\label{initial_comp1}
\end{figure}

As shown in Fig.~\ref{initial_clustering}, for higher levels of noise, e.g., $\sigma = 0.03$, we observe nonlinear effects. 
That is, $M$ and $A$ start to affect the dominant wave-number (see the middle and right panels of Fig.~\ref{initial_n_star}). Nevertheless, these effects are limited as the number of clusters is still 3 or 4 for $\sigma = 0.02$, and 2 or 3 for $\sigma = 0.03$. 
Besides, we still observe a general increase in the timing of the clustering behavior as $M$ decreases.
Fig.~\ref{initial_comp3} shows the evolution of normal opinions/agents distribution and the corresponding order parameter for three different combinations of $M$ and $A$ at the noise level $\sigma = 0.03$.
Once again, in the continuum-agent model we observe a 2-cluster profile for all combinations as shown in the top panels of Fig.~\ref{initial_comp3}. 
For the discrete-agent model, we observe a 3-cluster behavior in which the cluster formed between the two clusters at $x=0$ and $x=A$ has already disappeared for $\mathcal{S}_3$ in Fig.~\ref{initial_compare_noise3_SDE3} at $t=400$. 
Indeed, our simulations for $\sigma = 0.03$ reveals a single-cluster profile around the average radical opinion $x=A$ after a large enough time; see Fig.~\ref{noise_level}. 

\begin{figure}

\begin{subfigure}{.33\textwidth}
  \centering
  \includegraphics[width=1\linewidth]{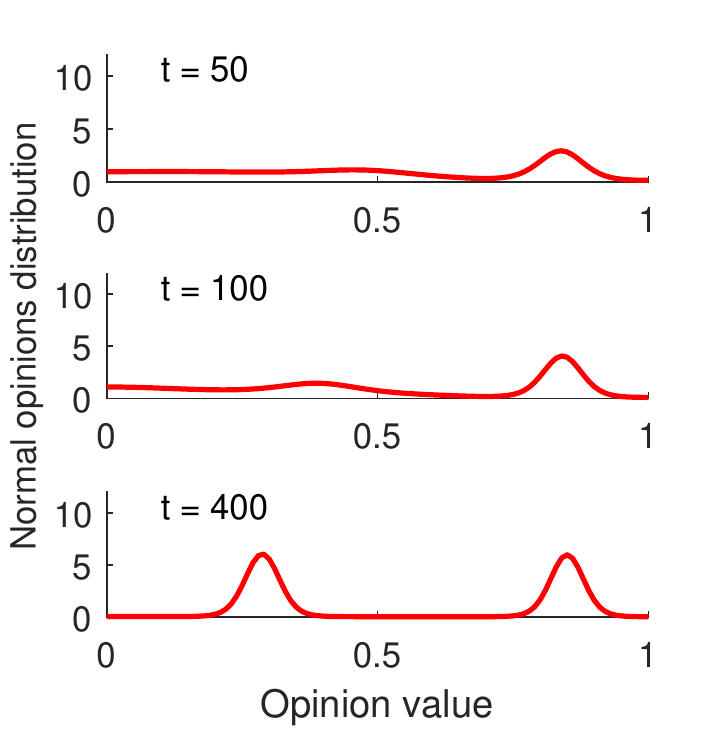}
  \caption{Continuum - $\mathcal{S}_1:(0.05,0.85)$}
  \label{initial_compare_noise3_PDE1}
\end{subfigure}%
\begin{subfigure}{.33\textwidth}
  \centering
  \includegraphics[width=1\linewidth]{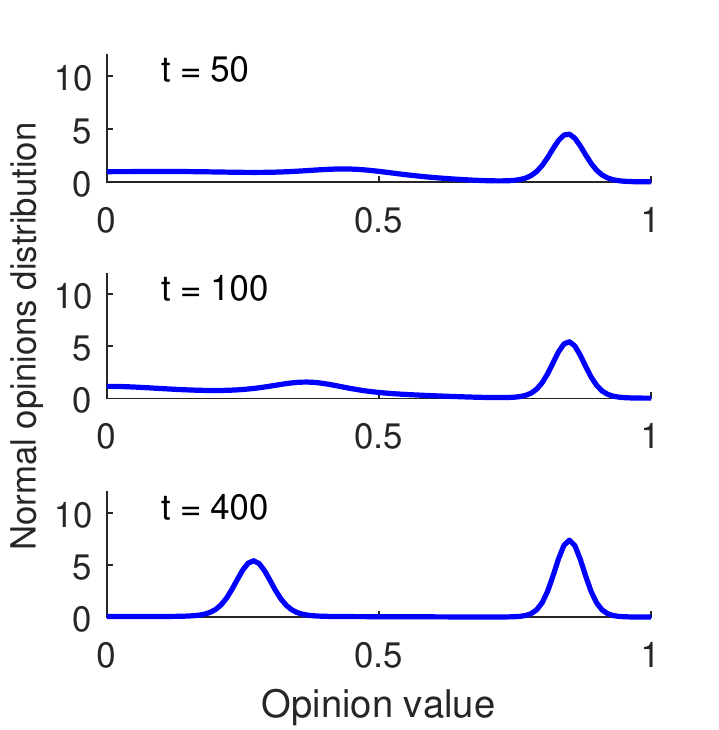}
  \caption{Continuum - $\mathcal{S}_2:(0.15,0.85)$}
  \label{initial_compare_noise3_PDE2}
\end{subfigure}%
\begin{subfigure}{.33\textwidth}
  \centering
  \includegraphics[width=1\linewidth]{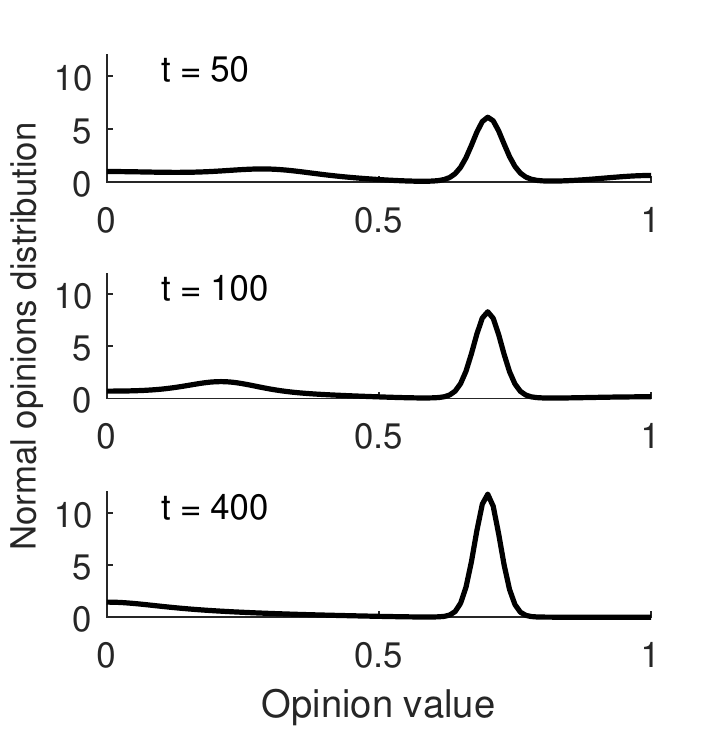}
  \caption{Continuum - $\mathcal{S}_3:(0.15,0.7)$}
  \label{initial_compare_noise3_PDE3}
\end{subfigure}

\begin{subfigure}{.33\textwidth}
  \centering
  \includegraphics[width=1\linewidth]{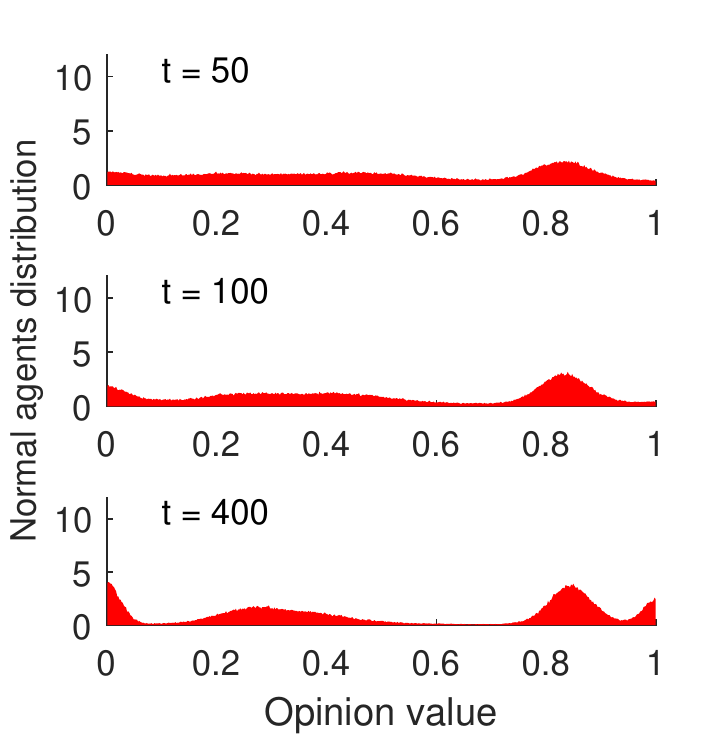}
  \caption{Discrete - $\mathcal{S}_1:(0.05,0.85)$}
  \label{initial_compare_noise3_SDE1}
\end{subfigure}%
\begin{subfigure}{.33\textwidth}
  \centering
  \includegraphics[width=1\linewidth]{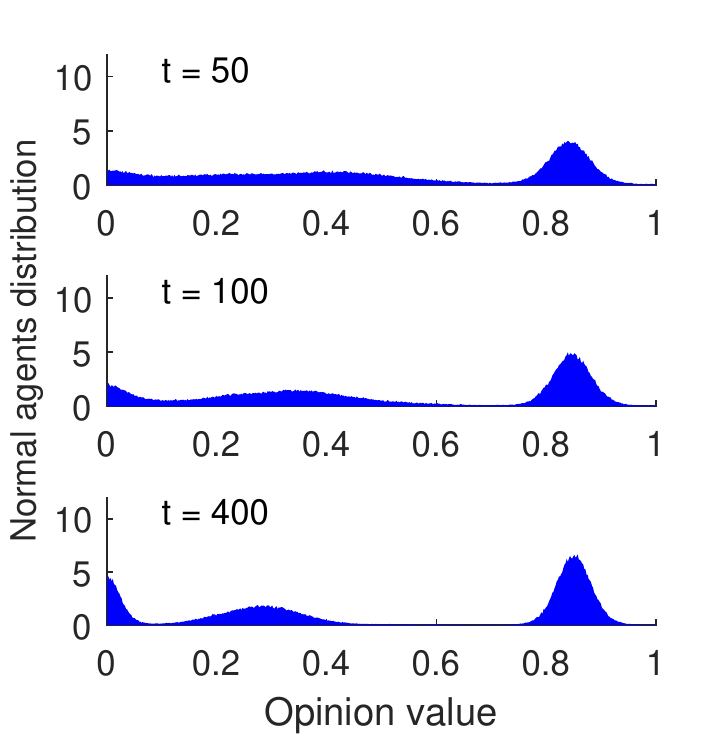}
  \caption{Discrete - $\mathcal{S}_2:(0.15,0.85)$}
  \label{initial_compare_noise3_SDE2}
\end{subfigure}%
\begin{subfigure}{.33\textwidth}
  \centering
  \includegraphics[width=1\linewidth]{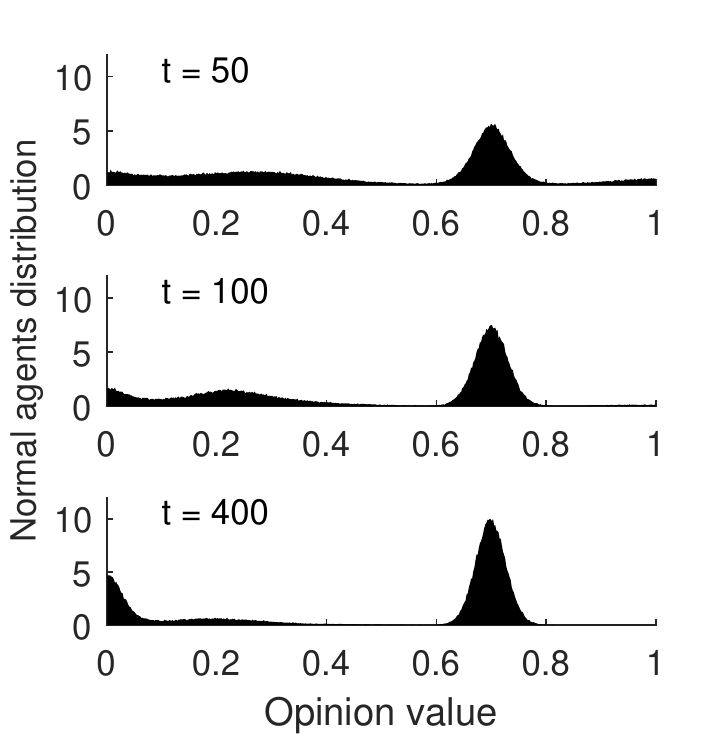}
  \caption{Discrete - $\mathcal{S}_3:(0.15,0.7)$}
  \label{initial_compare_noise3_SDE3}
\end{subfigure}

\begin{subfigure}{.5\textwidth}
  \centering
  \includegraphics[width=.9\linewidth]{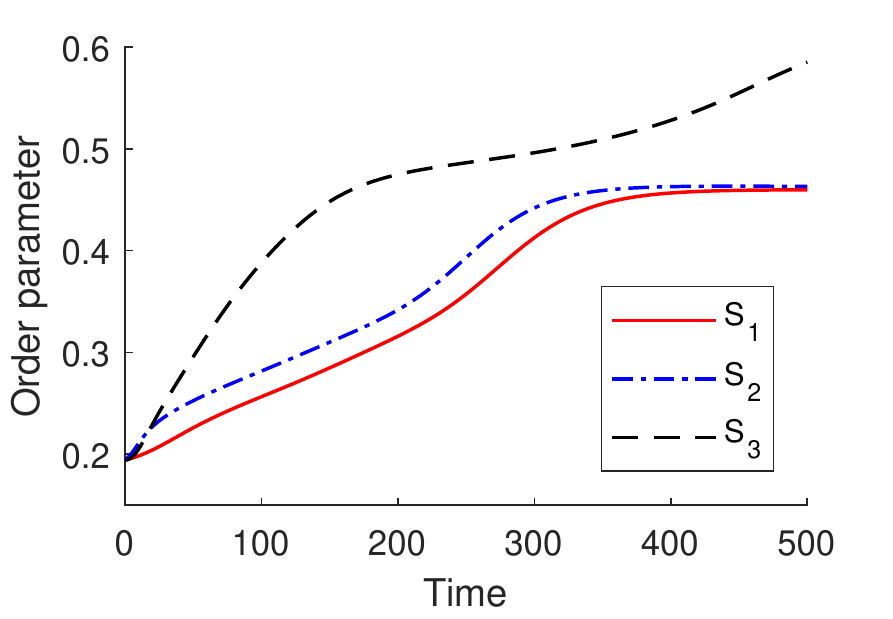}
  \caption{Continuum-agent model}
  \label{initial_comp3_ordpar_cont}
\end{subfigure}%
\begin{subfigure}{.5\textwidth}
  \centering
  \includegraphics[width=.9\linewidth]{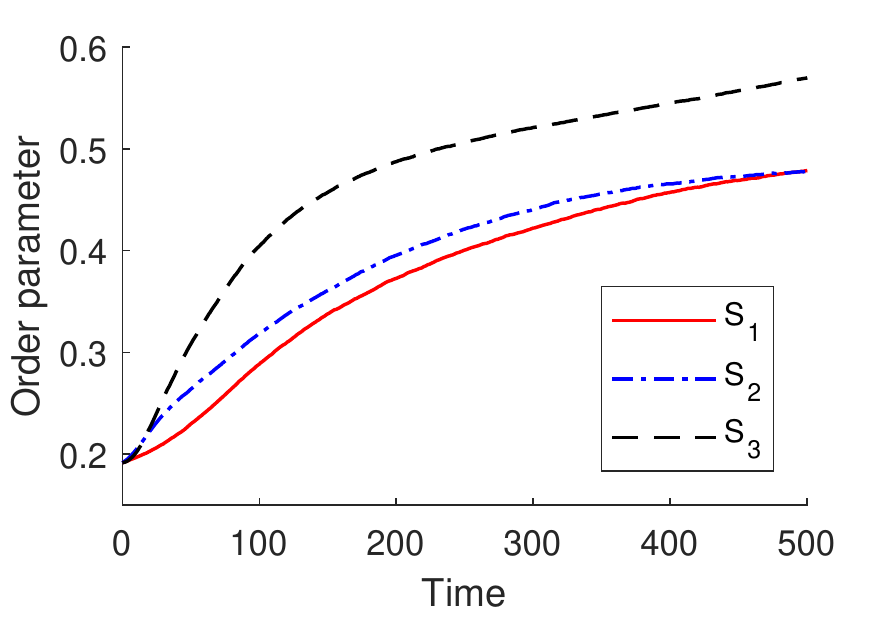}
  \caption{Discrete-agent model}
  \label{initial_comp3_ordpar_disc}
\end{subfigure}

\caption{Numerical simulation of the model with $\sigma = 0.03$ for different values of $(M,A)$, namely, $\mathcal{S}_1:(0.05,0.85)$,  $\mathcal{S}_2:(0.15,0.85)$, and  $\mathcal{S}_3:(0.15,0.7)$. The upper panels (A, B, and C) show the opinion distribution for continuum-agent model. The middle panels (D, E, and F) show the the result of Monte Carlo simulation (average of 300 realizations) of discrete-agent model. The lower panels (G and H) show the evolution of order parameter for these systems.}
\label{initial_comp3}
\end{figure}

To summarize the discussions above, for concentrated distribution of radicals, the main effect of the zeroth and first moments of radical distribution is on the timing and positioning of the possible clustering behavior, respectively. The number of clusters (to be precise, the life-time of possible transient clustered profiles) is mainly determined by the noise level of the system. This is particularly the case for lower levels of noise.  

%===============================================================================
\section{Final Remarks} 
\label{sec:conclusion}
%===============================================================================

In this paper, we considered a macroscopic model for bounded confidence opinion dynamics with environmental noise. 
In particular, we studied the effect of exogenous influence by adding a mass of radical (continuum) agents to the original population of the normal agents. 
The well-posedness of the continuum dynamics expressed as a nonlinear Fokker-Planck equation was established under some assumptions on the initial density of the normal opinions and the density of radical opinions. 
The long-term behavior of the model was also discussed by considering the corresponding stationary equation. 
In this regard, we provided a sufficient condition based on the noise level that guarantees exponential convergence of the dynamics towards the stationary state that can be made arbitrarily close to uniform distribution. 
In the context of opinion dynamics, we derived a theoretical bound on the minimum noise level required to counteract the effect of radical agents and keep the system in a somewhat uniform state. 

Exploiting the periodicity of the considered continuum-agent model, we used Fourier analysis to provide a general framework for characterization of the clustering behavior of the system with uniform initial distribution.
We then applied this framework for a particular distribution of radical opinions, namely, a relatively concentrated triangular distribution. 
In particular, we studied the effect of the relative mass of the radicals on the critical noise level for order-disorder transition. 
As expected, the analysis showed that for a larger number of radical agents, the critical noise level  increases. 
We note that this result corresponds to the theoretical result on the global estimate for stationary state. 
However, comparing the theoretical lower bound on the noise level for the global estimate with its counterpart derived numerically, we find that the theoretical bound is quite conservative, which was expected considering its theoretical nature. 
We also considered the effect of relative mass and average opinion of radicals on the number, timing and positioning of the clusters for noises smaller than the critical noise level. Here, the noise level was shown to be the main factor in determining the number of clusters. 
Meanwhile, the relative mass of the radicals mainly affects the timing of the clustering behavior, that is, for larger masses of radicals, the clustering behavior is expected to emerge faster. 
On the other hand, the main effect of the average opinion of the radicals is on the positioning of the clusters; the clusters are positioned in a way that we see a cluster formed around the average opinions of radicals. 
The numerical simulations of the continuum-agent model and the corresponding discrete-agent model were in agreement with these results.

%==================================== APPENDICES ===============================
\appendix 

%===============================================================================
\section{Preliminaries on Function Spaces} 
\label{Append: prelimin}
%===============================================================================

The definitions provided here are mostly borrowed from~\cite{Evans}. 
Let $\{f_k)\}_{k=1}^{\infty}$ be a sequence in a Banach space $B$ with norm $\Vert \cdot \Vert_B$. 
The \emph{strong} convergence $f_k \rightarrow f$ implies $\Vert f_k - f \Vert_B \rightarrow 0$, while the \emph{weak} convergence $f_k \rightharpoonup f$ implies $ g (f_k) \rightarrow g(f)$ for all bounded linear functionals $g: B \rightarrow \R$.

Let $f:\tilde{X} \rightarrow \R$ be a measurable function on $\tilde{X} = (-1,1)$. The $L^p$-norm of $f$ is defined as follows
\begin{eqnarray*}
\Vert f \Vert_{L^p (\tilde{X})} = \left\{\begin{array}{ll}
\left( \int_{\tilde{X}} |f(x)|^p \right)^{\frac{1}{p}}, & 1\leq p < \infty \\
\esup_{\tilde{X}} |f(x)|, &  p = \infty.
\end{array}
\right.
\end{eqnarray*}
Then, $L^p (\tilde{X})$ denotes the Banach space of all measurable functions $f:\tilde{X} \rightarrow \R$ for which $\Vert f \Vert_{L^p (\tilde{X})} < \infty$.

Let $f,g \in L^1_{loc} (\tilde{X})$ be locally summable functions (i.e., $f,g$ have a finite integral over every compact subset of $\tilde{X}$). We say that $g$ is the $k$-th \emph{weak} (partial) derivative of $f$, if 
$$ \int_{\tilde{X}} f \ \partial^k_x \phi \ \diff x = (-1)^k \int_{\tilde{X}} g \ \phi \ \diff x, $$
for all test functions $\phi \in C^{\infty}_{c}(\tilde{X})$ (infinitely differentiable functions $\phi : \tilde{X} \rightarrow \R$ with compact support in $\tilde{X}$).

$H^k(\tilde{X})$ for $k \in \N$ is used to denote the Sobolev space $W^{k,2}(\tilde{X})$ consisting of functions $f \in L^2 (\tilde{X})$ whose \emph{weak} derivatives up to order $k$ exist and belong to $L^2 (\tilde{X})$. Note that $H^k(\tilde{X})$ is a Hilbert space. 

We use the subscript \emph{per} to denote the closed subspace of \textit{periodic} functions in the corresponding function space, e.g.,
\begin{align*}
L^p_{per}(\tilde{X}) = \{ f \in L^p(\tilde{X}): f(-1) = f(1) \} \text{ and } H^k_{per}(\tilde{X}) = \{ f \in H^k(\tilde{X}): f(-1) = f(1) \}.
\end{align*}
Similarly, we use the subscript \emph{ep} to denote the closed subspace of \textit{even periodic} functions in the corresponding function space, e.g.,
\begin{align*}
&L^p_{ep}(\tilde{X}) = \{ f \in L^p_{per}(\tilde{X}): f(-x) = f(x), \ \forall x \in \tilde{X} \}, \\
&H^k_{ep}(\tilde{X}) = \{ f \in H^k_{per}(\tilde{X}): f(-x) = f(x), \ \forall x \in \tilde{X} \}.
\end{align*}

We denote the dual space of $H^1_{per}(\tilde{X})$ by $H^{-1}_{per}(\tilde{X})$, that is, the space of bounded linear functionals on $H^1_{per}(\tilde{X})$. 
Moreover, we use $\langle\cdot,\cdot\rangle$ to denote the corresponding paring of $H^1_{per}(\tilde{X})$ and $H^{-1}_{per}(\tilde{X})$. 
That is, for $f \in H^1_{per}(\tilde{X})$ and $g \in H^{-1}_{per}(\tilde{X})$, we use $\langle g, f\rangle $ to denote the real number $g(f)$. 
Since periodic boundary condition allows for integration by parts without extra terms, $H^{-1}_{per}(\tilde{X})$ has most of the properties of the space $H^{-1}(\tilde{X})$, the dual space of $H^{-1}_0(\tilde{X})$; see~\cite[Section 5.9.1]{Evans} for a detailed description of the space $H^{-1}(\tilde{X})$. 
In particular, one can extend the result in~\cite[Section 5.9, Theorem 3]{Evans} to derive ~\cite[Theorem 3.8]{Chazelle15}. For reader's convenience, the corresponding theorem is presented below.

\begin{TThm}
\cite[Theorem 3.8]{Chazelle15} Let the function $f: \tilde{X}\times [0,T] \rightarrow \R$ be such that 
\begin{align*}
f \in L^2(0,T;H^{1}_{per}(\tilde{X})) \text{ and } f_t \in L^2(0,T;H^{-1}_{per}(\tilde{X})).
\end{align*} 
Then, $f \in C(0,T;L^{2}_{per}(\tilde{X}))$ after possibly being redefined on a set of measure zero. Moreover, the mapping $t\mapsto \Vert f(t) \Vert_{L^2 (\tilde{X})}^2$ is absolutely continuous, with
\begin{equation*}
\frac{\diff}{\diff t} \Vert f(t) \Vert_{L^2 (\tilde{X})}^2 = 2 \langle f_t,f \rangle,
\end{equation*}
for almost every $t \in [0,T]$.
\end{TThm}

%===============================================================================
\section{Approximate Solution to Stationary Equation} 
\label{Append:approx}
%===============================================================================

In order to provide an approximate solution to the stationary equation~\eqref{pdes}, we assume radicals are highly concentrated around a particular opinion value $x=A$. 
To be precise, we assume that the average opinion of radicals is 
$
A = \int_X x \ \rho_r (x) \ \diff x 
$ 
and the variance of radicals
$
\sigma_r^2 = \int_X (x-A)^2 \ \rho_r (x) \ \diff x 
$
is much smaller than the confidence range $R$.
It helps to think of the limit being a point mass of radicals located at opinion value $x=A$. 
We further assume that the noise level $\sigma$ is also much smaller than $R$ so that the inter-cluster influences (from other possible clusters) can be ignored.
Using these assumptions, we can expect this particular cluster of normal agents to be concentrated around $A$. 
This implies that in order to evaluate the integral in~\eqref{StaSol}, we only need to consider values of $y$ near $A$. 
Under these assumptions, for $R < A < L-R$, we can write 
\begin{align*}
\int_{0}^{x} w \star (\rho + M \rho_r) \ \diff z &= \int_{0}^{x} \int (z-y) \ \textbf{1}_{|y-z|\leq R} \ (\rho(y) + M \rho_r(y)) \ \diff y \diff z \nonumber \\
&\approx \int_{0}^{x} \int_{A-R}^{A+R} (z - A) \ \textbf{1}_{|z - A|\leq R} \ (\rho(y) + M \rho_r(y)) \ \diff y \diff z \nonumber \\
&=  \int_{0}^{x} (z - A)  \textbf{1}_{|z - A|\leq R} \ \diff z  \int_{A-R}^{A+R} (\rho(y) + M \rho_r(y)) \ \diff y \nonumber \\
&= \frac{1}{2} \left( (x - A)^2 - R^2 \right) \textbf{1}_{|x - A|\leq R} \int_{A-R}^{A+R}  (\rho(y) + M \rho_r(y)) \ \diff y \nonumber \\
&\approx \frac{M+1}{2} \ \left((x - A)^2 - R^2 \right) \ \textbf{1}_{|x - A|\leq R} .
\end{align*}
Inserting this result in~\eqref{StaSol}, we have
\begin{equation*}
\rho^s(x) = \frac{1}{K} \exp \left\{ - \frac{M+1}{\sigma^2} \ \left((x - A)^2 - R^2 \right) \ \textbf{1}_{|x - A|\leq R} \right\},
\end{equation*}
which can also be expressed as (by modefyying the normalizing constant $K$)
\begin{equation*}
\rho^s(x) = \frac{1}{K} \exp \left\{ -\frac{M+1}{\sigma^2} \min \left\{(x - A)^2, \ R^2 \right\} \right\}.
\end{equation*}
In Fig.~\ref{noise_distr}, this approximate solution is shown for $\sigma=0.03$ and $\sigma=0.04$, where the system has converged to a single cluster profile in both continuum- and discrete-agent models corresponding to the assumptions for derivation of the approximate solution. 
This result shows that the approximate solution is indeed a good approximation as it almost perfectly matches the numerical solution of the continuum-agent model.

%===============================================================================
\section{Euler-Maruyama Method for Discrete-agent Model} 
\label{Append: Euler-Maruyama}
%===============================================================================

The interacting SDEs considered for the simulation of the discrete-agent model in this study is
\begin{eqnarray} \label{sde2a}
\left\{
\begin{array}{l}
\diff x_i = -\frac{1}{N} \left( \sum_{j \in \mathcal{N}_i}(x_i - x_j^{ext}) + \sum_{j \in \mathcal{N}_i}(x_i - x_{r_j}^{ext})  \right) \diff t+ \sigma \  \diff W^{i}_t, \\
x_i(0) = x_{i_0}.
\end{array}
\right.
\end{eqnarray}
where $x_i^{ext}, \ i=1,\ldots N$ are the opinions of normal agents and $x_{r_i}^{ext}, \ i=1,\ldots N_r$ are the opinions of radical agents with $N_r = MN$. 
The superscript $ext$ corresponds to the even $2$-periodic extension as explained below.  
In order to solve~\eqref{sde2a} numerically, we employ the Euler-Maruyama method. 
Algorithm~\ref{EMAL} summarizes the numerical scheme.
As described in Section~\ref{sec:NumSim}, we assume that the radicals have a triangular distribution centered at $A$ with width $2S$. 
That is, we produce a random sample of radicals with size $N_r$ from the triangular distribution~\eqref{TriangStub} (Step 0).
In particular, for complete correspondence between the discrete- and continuum- agent models, we also consider the effect of even $2$-periodic extension in our simulations. 
To this end, we use even $2$-periodic extensions of $\textbf{x}$ and $\textbf{x}_r$ for calculating the sum on the r.h.s. of~\eqref{sde2a} (vectors denoted by $\textbf{x}^{ext}$ and $\textbf{x}_r^{ext}$ in Steps 0, 1 and 2). 
Also, due to periodicity, in each iteration, the opinion values outside the support $X = [0,1]$ are \emph{reflected} back to $X$ (Step 5).

\begin{algorithm}
   \caption{Euler-Maruyama method for even $2$-periodic extension of~\eqref{sde2a}}
   \label{EMAL}
\begin{algorithmic}
	\STATE {\bfseries Step 0.} $\textbf{x}_r = (x_{r_1},x_{r_2}, \cdots, x_{r_{N_r}})^T \sim \rho_r(x)$;
	
	\ \ \ \ \ \ \ \ \ \ \ \ $\textbf{x}_r^{ext} = [\textbf{x}_r;\  -\textbf{x}_r;\  2-\textbf{x}_r]$;
	\FOR{$t=0$ {\bfseries to} $t = \frac{T}{\Delta t}-1$:}   
   \STATE {\bfseries Step 1.} $\textbf{x}^{ext}(t) = [\textbf{x}(t);\  -\textbf{x}(t);\  2-\textbf{x}(t)]$, where $\textbf{x}(t) = (x_1(t),x_2(t), \cdots, x_N(t))^T$;
   \STATE {\bfseries Step 2.} $\dot{x}_i(t)  = -\frac{1}{N} \left( \sum_{j \in \mathcal{N}_i}(x_i - x^{ext}_j) + \sum_{j \in \mathcal{N}_i}(x_i - x_{r_j}^{ext})  \right)$;
   \STATE {\bfseries Step 3.} $dW^{i}_t = z_i \sqrt{\Delta t}$, where $z_i \sim \mathrm{N}(0,1)$;
   \STATE {\bfseries Step 4.} $x_{i}(t+1) = x_i(t) + \dot{x}_i(t) \cdot \Delta t + \sigma \ dW^{i}_t$;
   \STATE {\bfseries Step 5.} $x_{i}(t+1) = x_{i}(t+1) \bmod (2L)$;
   
   \ \ \ \ \ \ \ \ \ \ \ \ if $x_{i}(t+1)> L$, then $x_{i}(t+1) = 2-x_{i}(t+1)$.
   \ENDFOR
\end{algorithmic}
\end{algorithm}

%===============================================================================
\section{Pseudo-spectral Method for Continuum-agent Model} 
\label{Append: Psuedo-Spectral}
%===============================================================================

This is a modification of the algorithm given in~\cite{Wang17} for our model. 
Let us first recall the PDE~\eqref{pde2} for the continuum-agent model
\begin{eqnarray} \label{PS}
\rho_t = (\rho \ G)_x + \frac{\sigma^2}{2}\rho_{xx}  
\end{eqnarray} 
\noindent where
\begin{equation} \label{g}
G(x,t)= w \star (\rho + M \rho_r) = \int_{\tilde{X}} (x-y) \ \textbf{1}_{|x-y|\leq R} \ (\rho (y,t) + M \rho_r (y)) \ \diff y,
\end{equation}
Using the first $N_f$ terms of Fourier expansions of $\rho$ and $\rho_r$, we can write
\begin{equation*}
\rho(x,t) + M \rho_r(x) = \sum_{k=-N_f}^{N_f} ( \hat{\rho}_{k}(t)+ M \hat{\rho}_{r_{k}} ) \ e^{i\pi kx}.
\end{equation*}
Inserting this into~\eqref{g}, we obtain
\begin{eqnarray*}
G(x,t)&=&\sum_{-N_f \leq k \leq N_f, k \neq 0} - \frac{2iR}{\pi k} \ f_k \ ( \hat{\rho}_{k}(t)+ M \hat{\rho}_{r_{k}} ) \ e^{i\pi kx},
\end{eqnarray*}
where $f_k$ is given by~\eqref{fn}. 
Hence,
\begin{eqnarray*}
\hat{G}_{k}(t) = 
\left\{
\begin{array}{ll}
- \frac{2iR}{\pi k} \ f_k \ ( \hat{\rho}_{k}(t)+ M \hat{\rho}_{r_{k}} ),  &k\neq 0 \\
0, &k=0.
\end{array}
\right.
\end{eqnarray*}
With Fourier coefficients of $G$ in terms of Fourier coefficients of $\rho$ in hand, we can apply pseudo-spectral method for solving~\eqref{PS} as described in Algorithm~\ref{PSAl}. As shown, the multiplication $h = \rho \ G$ on the r.h.s. of~\eqref{PS} is performed in the time domain (Step 4), while the differentiations w.r.t. $x$ are performed in frequency domain (Step 5). Note that the symmetric nature of solution is preserved in the algorithm (Step 1). Also, preservation of mass is satisfied by setting $\hat{\rho}_{0}(t+1) = \hat{\rho}_{0}(t)$ (Step 5). 
We also note that the algorithm is semi-explicit (see the first equation in Step 5). 

\begin{algorithm}
   \caption{Pseudo-spectral method for~\eqref{PS}}
   \label{PSAl}
\begin{algorithmic}
	\STATE {\bfseries Step 0.} for $x \in [-1,0]$ set $\rho_r(x) = \rho_r(-x)$; \\
	\ \ \ \ \ \ \ \ \ \ \ \  $\hat{\rho}_{r_{k}} = \text{FFT} \left[ \rho_r(x) \right]$;
	\FOR{$t=0$ {\bfseries to} $t = \frac{T}{\Delta t}-1$:}   
   \STATE {\bfseries Step 1.} for $x \in [-1,0]$ set $\rho(x,t) = \rho(-x,t)$;
   \STATE {\bfseries Step 2.} $\hat{\rho}_k(t) = \text{FFT} \left[ \rho(x,t) \right]$;
   \STATE {\bfseries Step 3.} $\hat{G}_k(t) = - \frac{2iR}{\pi k} \ f_k \ ( \hat{\rho}_{k}(t)+ M \hat{\rho}_{r_{k}} ), \ \hat{G}_0(t) = 0$; \\
\ \ \ \ \ \ \ \ \ \ \ \  $G(x,t) = \text{iFFT}\left[\hat{G}_k(t) \right]$;
   \STATE {\bfseries Step 4.} $h(x,t) = \rho(x,t) \ G(x,t) $; \\
\ \ \ \ \ \ \ \ \ \ \ \  $\hat{h}_k(t) = \text{FFT} \left[ h(x,t) \right]$;
   \STATE {\bfseries Step 5.} $\hat{\rho}_{k}(t+1) = \left(i\pi k \hat{h}_k(t) - \frac{\pi^2 \sigma^2 k^2}{2}\hat{\rho}_{k}(t+1) \right) \cdot \Delta t + \hat{\rho}_{k}(t) $; \\
\ \ \ \ \ \ \ \ \ \ \ \  $\hat{\rho}_{0}(t+1) = \hat{\rho}_{0}(t)$; \\
\ \ \ \ \ \ \ \ \ \ \ \  $\rho(x,t) = \text{iFFT}\left[\hat{\rho}_k(t) \right]$;
   \ENDFOR
\end{algorithmic}
\end{algorithm}

%===============================================================================
	\bibliographystyle{IEEEtr} %{siam}
	\bibliography{ref}
%===============================================================================

\end{document}

%% file: AMIN_style.tex
\makeatletter
\def\@settitle{\begin{center}%
		\baselineskip14\p@\relax
		\normalfont\LARGE\bfseries
		\@title
	\end{center}%
}
\makeatletter

\def\subsection{\@startsection{subsection}{2}%
	\z@{.5\linespacing\@plus.7\linespacing}{.5\linespacing}%
	{\normalfont\bfseries}}

\def\subsubsection{\@startsection{subsubsection}{3}%
	\z@{.5\linespacing\@plus.7\linespacing}{.5\linespacing}%
	{\normalfont\itshape}}

\usepackage[usenames, dvipsnames]{color}
\definecolor{darkblue}{rgb}{0.0, 0.0, 0.45}

\usepackage[colorlinks	= true,
raiselinks	= true,
linkcolor	= darkblue, %MidnightBlue,
citecolor	= Mahogany,
urlcolor	= ForestGreen,
pdfauthor	= {M.A.S. Kolarijani},
pdftitle	= {},
pdfkeywords	= {},
pdfsubject	= {},
plainpages	= false]{hyperref}

\usepackage{dsfont,amsfonts,amssymb,amsmath,amsthm}
\usepackage{mathtools, mathrsfs}

\usepackage{algorithm,algorithmic}%,algorithmicx}
\usepackage{graphicx}
\usepackage[font=small,labelfont=rm]{subcaption}
\usepackage[font=small,margin=10pt]{caption}
\usepackage[T1]{fontenc}
\usepackage{cite}

\allowdisplaybreaks
\date{\today}

%% file: AMIN_commands.tex
\theoremstyle{plain}
\newtheorem{Thm}{Theorem}[section]
\newtheorem*{TThm}{Theorem}
\newtheorem{Prop}[Thm]{Proposition}

\newtheorem{Lem}[Thm]{Lemma}
\newtheorem{Cor}[Thm]{Corollary}

\newtheorem{Rem}[Thm]{Remark}

\newcommand{\R}{\mathbb{R}}

\newcommand{\N}{\mathbb{N}}

\newcommand{\Let}{\coloneqq}

\newcommand{\diff}{\mathrm{d}}

\newcommand{\ol}[1]{\overline{#1}}

\DeclareMathOperator{\e}{e}
\DeclareMathOperator{\esup}{ess\text{ }sup}

\DeclareMathOperator*{\argmax}{\arg\!\max}

\DeclareMathOperator{\sinc}{sinc}
\def\br{\bar{\rho}}